%% file: LaTeX_template.tex
\documentclass[reqno,11pt]{book}

\usepackage{amssymb}
\usepackage{amsthm}
\usepackage{amsmath}
\usepackage{amsfonts}
\usepackage{mathtools}
\usepackage[shortlabels]{enumitem}	% It supersedes both enumerate and mdwlist. The package option shortlabels is included to configure the labels like in enumerate.

\usepackage{pifont}					% Access to PostScript standard Symbol and Dingbats fonts
\usepackage{wasysym}				% additional characters
\usepackage{bm}						% bold fonts: \bm{...}
\usepackage{extarrows}				% may be replaced by tikz-cd
\usepackage[monochrome]{color}                  % makes everything black
\usepackage[colorlinks]{hyperref}
\usepackage[table,hyperref]{xcolor} 			% before tikz-cd. 
\usepackage{float}				%Improved interface for floating objects
\usepackage{caption,subcaption}
\usepackage{adjustbox}			% for me it is usually used in tables 
\usepackage{stackengine}		%baseline changes

\usepackage{tikz-cd}
\usepackage{graphicx}			% Enhanced support for graphics
\usepackage{stmaryrd}

\usepackage{array}
\usepackage{makecell}
\usepackage{multicol}
\usepackage{multirow}
\usepackage{diagbox}
\usepackage{longtable}

\usepackage{MA_Titlepage}

\numberwithin{equation}{section}

\theoremstyle{plain}
\newtheorem{theorem}{Theorem}[section]

\newtheorem{definition}[theorem]{Definition}
\newtheorem{lemma}[theorem]{Lemma}
\newtheorem{proposition}[theorem]{Proposition}

\newtheorem{eg}[theorem]{Example}

\numberwithin{equation}{section}

\theoremstyle{remark}

\newtheorem{remark}[theorem]{Remark}

\DeclareMathOperator{\Hom}{\operatorname{Hom}}	\DeclareMathOperator{\End}{\operatorname{End}}

\DeclareMathOperator{\diag}{\operatorname{diag}}	%for diagonal matrices

\DeclareMathOperator{\Rep}{\operatorname{Rep}}
\newcommand{\coker}{\operatorname{coker}}

\DeclareMathOperator{\Ext}{\operatorname{Ext}}

\DeclareMathOperator{\GL}{\operatorname{GL}}

\DeclareMathOperator{\Spec}{\operatorname{Spec}}
\DeclareMathOperator{\Coh}{\operatorname{Coh}}
\DeclareMathOperator{\Ind}{\operatorname{Ind}}
\DeclareMathOperator{\cInd}{\operatorname{c-Ind}}

\DeclareMathOperator{\Perf}{\operatorname{Perf}}
\DeclareMathOperator{\QCoh}{\operatorname{QCoh}}
\DeclareMathOperator{\Bun}{\operatorname{Bun}}
\DeclareMathOperator{\qc}{\operatorname{qc}}
\DeclareMathOperator{\Nilp}{\operatorname{Nilp}}
\DeclareMathOperator{\basic}{\operatorname{basic}}
\DeclareMathOperator{\lis}{\operatorname{lis}}
\DeclareMathOperator{\sss}{\operatorname{ss}}
\DeclareMathOperator{\Sing}{\operatorname{Sing}}

\DeclareMathOperator{\Modr}{\operatorname{-Mod}}

\newcommand{\Fr}{\operatorname{Fr}}
\newcommand{\Cont}{\operatorname{Cont}}
\renewcommand{\Modr}{\operatorname{Mod-}}
\newcommand{\Modl}{\operatorname{-Mod}}

\usepackage[colorinlistoftodos,textsize=footnotesize]{todonotes}
\usepackage[a4paper,left=3cm,right=3cm,bottom=4cm]{geometry}
\usepackage{indentfirst}	% Indent first paragraph after section header

\allowdisplaybreaks

\begin{document}

\include{contents}
\include{Introduction}

\include{MoLP}

\include{Supercuspidal}

\include{CLLC}

\bibliographystyle{alpha2}
\bibliography{reference}

\end{document}

%% file: contents.tex
%Name of the author of the thesis 
\authornew{Chenji Fu}
%Date of birth of the Author
\geburtsdatum{15th July 1998}
%Place of Birth
\geburtsort{Jiaxing, China}
%Date of submission of the thesis
\date{August 15, 2023}

%Name of the Advisor
% z.B.: Prof. Dr. Peter Koepke
\betreuer{Advisor: Prof. Dr. Peter Scholze}
%name of the second advisor of the thesis
\zweitgutachter{Second Advisor: Prof. Dr. Jessica Fintzen}

%Name of the Insitute of the advisor
%z.B.: Mathematisches Institut
\institut{Mathematisches Institut}
%\institut{Institut f\"ur Angewandte Mathematik}
%\institut{Institut f\"ur Numerische Simulation}
%\institut{Forschungsinstitut f\"ur Diskrete Mathematik}
%Title of the thesis 
\title{\Large On the categorical local Langlands conjectures for depth-zero regular supercuspidal representations}
%Do not change!
\ausarbeitungstyp{Master's Thesis  Mathematics}

\maketitle
\tableofcontents

%%%%%%%%%%%%%%%%%%%%%%%%%%%%%%%%%%%%%%%%%%%%%%%%%%%%%%%%%%%%%%%%%%%%%%%%%%%%%%%%%%%%%%%%%%%%%

%% file: Introduction.tex
\chapter{Introduction}

Let $F$ be a non-archimedean local field, of residue characteristic $p > 1$, with residue field  $\mathbb{F}_q$. Let $G$ be a connected reductive group over $F$. For simplicity, we assume that $G$ is split, semisimple, and simply connected. Let $\ell$ be a prime different from $p$. Let $\overline{\mathbb{Q}}_{\ell}$ be the algbebraic closure of the $\ell$-adic numbers $\mathbb{Q}_{\ell}$. Let $\Lambda=\overline{\mathbb{Z}}_{\ell}$ be the integral closure of $\mathbb{Z}_{\ell}$ in $\overline{\mathbb{Q}}_{\ell}$. Let $W_F$ be the Weil group of $F$ and $\hat{G}$ the Langlands dual group of $G$, considered as a split algebraic group over $\Lambda$. The categorical local Langlands conjecture predicts that there is a fully faithful embedding
$$\Rep_{\Lambda}(G(F)) \longrightarrow \QCoh(Z^1(W_F, \hat{G})_{\Lambda}/\hat{G})$$
from the category of smooth representations of the $p$-adic group $G(F)$ to the category of quasi-coherent sheaves on the stack of Langlands parameters. In this paper, we compute the two sides explicitly for depth-zero regular supercuspidal representations of the group $G$ and verify the categorical local Langlands conjecture for depth-zero supercuspidal part of $\GL_n$.\footnote{We will see that although $\GL_n$ is not simply connected, the theory still works without much change. Also, we do not need to assume that $G$ is simply connected for the results on the $L$-parameter side.}

Fixing an irreducible depth-zero regular supercuspidal representation $\pi \in \Rep_{\overline{\mathbb{F}}_{\ell}}(G(F))$,\footnote{Note that we start with a representation with $\overline{\mathbb{F}}_{\ell}$-coefficients instead of $\overline{\mathbb{Q}}_{\ell}$-coefficients, because we are interested in describing the ($\overline{\mathbb{Z}}_{\ell}$-)block of 
$\Rep_{\Lambda}(G(F))$.} the (classical) local Langlands conjecture predicts that it should correspond to a tame, regular semisimple, elliptic $L$-parameter (TRSELP for short, Definition \ref{Def TRSELP}) 
$$\varphi \in Z^1(W_F, \hat{G}(\overline{\mathbb{F}}_{\ell}))$$
(see also \cite{debacker2009depth}, but here we replace the coefficient $\mathbb{C}$ with $\overline{\mathbb{F}}_{\ell}$). Let $X_{\varphi}$ be the connected component of $Z^1(W_F, \hat{G})_{\Lambda}$ containing $\varphi$. As mentioned above, this paper focuses on the depth-zero regular supercuspidal part of the categorical local Langlands conjecture, which predicts a fully faithful embedding
$$\Rep_{\Lambda}(G(F))_{[\pi]} \longrightarrow \QCoh(X_{\varphi}/\hat{G})$$
from the block of $\Rep_{\Lambda}(G(F))$ containing $\pi$ to the category of quasi-coherent sheaves on the connected component $X_{\varphi}/\hat{G}$ of the stack of $L$-parameters $Z^1(W_F, \hat{G})_{\Lambda}/\hat{G}$ containing $\varphi$.

\section{$L$-parameter side}
Let $G$ be a connected split reductive group over $F$. Let $\varphi \in Z^1(W_F, \hat{G}(\overline{\mathbb{F}}_{\ell}))$ be a TRSELP (see Definition \ref{Def TRSELP}). In this section, we explain Chapter \ref{Chapter MoLP} on how to compute $\QCoh(X_{\varphi}/\hat{G})$. 

This is done by describing $X_{\varphi}/\hat{G}$ explicitly as a quotient stack over $\Lambda=\overline{\mathbb{Z}}_{\ell}$. 

\subsection{Heuristics on the component $X_{\varphi}/\hat{G}$}

In this subsection, we describe some heuristics on the component $X_{\varphi}/\hat{G}$ which help us to guess what this component should look like.

First, let us recall what is known over $\overline{\mathbb{Q}}_{\ell}$ instead of $\Lambda=\overline{\mathbb{Z}}_{\ell}$. Indeed, assuming that the center $Z(\hat{G})$ of $\hat{G}$ is finite, the connected component of the stack of $L$-parameters $Z^1(W_F, \hat{G})_{\overline{\mathbb{Q}}_{\ell}}/\hat{G}$ over $\overline{\mathbb{Q}}_{\ell}$ containing an elliptic $L$-parameter $\varphi'$ is known to be one point. More precisely, it is isomorphic to the quotient stack $[*/S_{\varphi'}]$, where $S_{\varphi'}=C_{\hat{G}}(\varphi')$ is the centralizer of $\varphi'$ (see \cite[Section X.2]{fargues2021geometrization}).

Second, let us explain the difference between the geometry of the connected components of the stack of $L$-parameters over $\overline{\mathbb{Q}}_{\ell}$ and $\overline{\mathbb{Z}}_{\ell}$. This can be seen from the example $G=\GL_1$. Indeed, 
$$Z^1(W_F, \widehat{\GL_1}) \cong \mu_{q-1} \times \mathbb{G}_m,$$
both over $\overline{\mathbb{Q}}_{\ell}$ and $\overline{\mathbb{Z}}_{\ell}$ (see Example \ref{Example GL_1}). However, $\mu_{q-1}$ is just $q-1$ discrete points over $\overline{\mathbb{Q}}_{\ell}$, while
the connected components of $\mu_{q-1}$ are isomorphic to $\mu_{\ell^k}$ (over $\overline{\mathbb{F}}_{\ell}$, hence also) over $\overline{\mathbb{Z}}_{\ell}$, where $k$ is the maximal integer such that $\ell^k$ divides $q-1$. So when describing the connected components of the stack of $L$-parameters over $\overline{\mathbb{Z}}_{\ell}$, there will be possibly some non-reduced part $\mu$ appearing.

These two features come together in the description of $X_{\varphi}/\hat{G}$, the connected component of $Z^1(W_F, \hat{G})/\hat{G}$ containing $\varphi$. Under mild assumptions, we prove that
$$X_{\varphi}/\hat{G} \cong [*/S_{\varphi}]\times \mu,$$
where $S_{\varphi}=C_{\hat{G}}(\varphi)$ and $\mu$ is some product of $\mu_{\ell^{k_i}}$'s. 

More precisely, to compute the connected component $X_{\varphi}$ of $Z^1(W_F, \hat{G})_{\overline{\mathbb{Z}}_{\ell}}$, we need to choose a lift $\psi$ of $\varphi$ over $\overline{\mathbb{Z}}_{\ell}$. We can now state the main theorem of Chapter \ref{Chapter MoLP} as follows (see Theorem \ref{Thm X/G}).

\begin{theorem}
	Assume that $Z(\hat{G})$ is finite. Let $\varphi \in Z^1(W_F, \hat{G}(\overline{\mathbb{F}}_{\ell}))$ be a TRSELP. Let $\psi \in Z^1(W_F, \hat{G}(\overline{\mathbb{Z}}_{\ell}))$ be a lift of $\varphi$.
	Let $X_{\varphi}$ be the connected component of $Z^1(W_F, \hat{G})_{\overline{\mathbb{Z}}_{\ell}}$ containing $\varphi$. Then we have an isomorphism of quotient stacks over $\overline{\mathbb{Z}}_{\ell}$
	\begin{equation}\label{Equation: X_phi}
		X_{\varphi}/\hat{G} \cong [*/S_{\psi}] \times \mu,
	\end{equation}
	where $S_{\psi}:=C_{\hat{G}}(\psi)$ is the schematic centralizer of $\psi$ in $\hat{G}$. 
\end{theorem}

\subsection{Ingredients of the computation}

Our computation follows the theory of moduli space of Langlands parameters developed in \cite[Section 2 and 4]{dhkm2020moduli} (see also \cite[Section 3, 4]{dat2022ihes} for a more gentle introduction). It is very helpful to do the example of $\GL_2$ first (see Chapter \ref{Chapter GL_n}).

To compute the component $X_{\varphi}/\hat{G}$ over $\overline{\mathbb{Z}}_{\ell}$, let us fix a lift $\psi \in Z^1(W_F, \hat{G}(\overline{\mathbb{Z}}_{\ell}))$ of $\varphi \in Z^1(W_F, \hat{G}(\overline{\mathbb{F}}_{\ell}))$. Denote $\psi_{\ell}$ the restriction of $\psi$ to the prime-to-$\ell$ inertia $I_F^{\ell}$ (Definition \ref{Definition: prime-to-ell inertia}).

Recall by \cite[Subsection 4.6]{dat2022ihes},
$$X_{\varphi}=X_{\psi} \cong \left(\hat{G} \times Z^1(W_F, N_{\hat{G}}(\psi_{\ell}))_{\psi_{\ell}, \overline{\psi}}\right)/C_{\hat{G}}(\psi_{\ell})_{\overline{\psi}},$$
where $Z^1(W_F, N_{\hat{G}}(\psi_{\ell}))_{\psi_{\ell}, \overline{\psi}}$  denotes the space of cocycles whose restriction to the prime-to-$\ell$ inertia $I_F^{\ell}$ equals $\psi_{\ell}$ and whose image in $Z^1(W_F, \pi_0(N_{\hat{G}}(\psi_{\ell})))$ is $\overline{\psi}$. 

Here, $Z^1(W_F, N_{\hat{G}}(\psi_{\ell}))_{\psi_{\ell}, \overline{\psi}}$ is essentially the space of cocycles of the split torus $$T:=N_{\hat{G}}(\psi_{\ell})^0=C_{\hat{G}}(\psi_{\ell})$$
since $\varphi$ is a TRSELP and $C_{\hat{G}}(\psi_{\ell})$ is generalized reductive, hence split over $\overline{\mathbb{Z}}_{\ell}$ (see Lemma \ref{Lem gen red}). We compute the space of tame cocycles of a commutative group scheme using the explicit presentation of the tame Weil group (see \eqref{Equation presentation of the tame Weil group} and \eqref{Equation space of tame cocycle}), we obtain that 
$$Z^1(W_F, N_{\hat{G}}(\psi_{\ell}))_{\psi_{\ell}, \overline{\psi}} \cong T \times \mu,$$
where $\mu=(T^{\Fr=(-)^q})^0$ is a product of $\mu_{\ell^{k_i}}$'s (see Theorem \ref{Thm X} for details).
In addition, it is not hard to see that
$$C_{\hat{G}}(\psi_{\ell})_{\overline{\psi}}=C_{\hat{G}}(\psi_{\ell})=T.$$

Therefore, we get 
$$X_{\varphi} \cong \left(\hat{G} \times T \times \mu\right)/T$$
(see Theorem \ref{Thm X}).
One needs to be a bit careful about the $T$ action on $T$, because here a twist by $\psi(\Fr)$ is involved. One can compute that 
$$[X_{\varphi}/\hat{G}] \cong [\left(T \times \mu\right)/T] \cong [T/T] \times \mu,$$
where $T$ acts on $T$ via twisted conjugacy.\footnote{Note that so far, we do not assume $Z(\hat{G})$ to be finite, hence the result also applies for $\GL_n$.} After that, assuming that $Z(\hat{G})$ is finite, we can work in the category of diagonalizable group schemes (whose structure is clear, see \cite[p70, Section 5]{brochard2014autour}) to identify $[T/T]$ with $[*/S_{\varphi}]$.

\section{Representation side}
Let $G$ be a connected reductive group over $F$. We assume that $G$ is split, semisimple, and simply connected. Let $\pi \in \Rep_{\overline{\mathbb{F}}_{\ell}}(G(F))$ be an irreducible depth-zero regular supercuspidal representation. In this section, we explain Chapter \ref{Chapter Rep} on how to compute the block $\Rep_{\Lambda}(G(F))_{[\pi]}$ of $\Rep_{\Lambda}(G(F))$ containing $\pi$.

\subsection{Equivalence to the block of a finite group of Lie type}

Recall that a depth-zero regular supercuspidal representation of $G(F)$ is of the form
$$\pi=\cInd_{G_x}^{G(F)}\rho$$
for some representation $\rho$ of the parahoric subgroup $G_x$ corresponding to a vertex $x$ in the Bruhat-Tits building of $G$ over $F$. Moreover, $\rho$ is the inflation of some regular supercuspidal representation $\overline{\rho}$ of the finite group of Lie type $\overline{G_x}:=G_x/G_x^+$.

Let $\mathcal{A}_{x,1}$ denote the block $\Rep_{\Lambda}(\overline{G_x})_{[\overline{\rho}]}$ of $\Rep_{\Lambda}(\overline{G_x})$ containing $\overline{\rho}$. Similarly, let 
$$\mathcal{B}_{x,1}:=\Rep_{\Lambda}(G_x)_{[\rho]},\qquad \mathcal{C}_{x,1}:=\Rep_{\Lambda}(G(F))_{[\pi]}.$$

Assume that the residue field of $F$ is $\mathbb{F}_q$. For simplicity, we assume that $q$ is greater than the Coxeter number of $\overline{G_x}$ (see Theorem \ref{Thm Broué} for reason). Then $\mathcal{A}_{x,1}$ is equivalent to a block of a finite torus via Broué's equivalence \ref{Thm Broué}. In addition, it is not hard to show that the inflation induces an equivalence of categories $\mathcal{A}_{x,1} \cong \mathcal{B}_{x,1}$.

The main theorem we will prove for the representation side is the following (see Theorem \ref{Thm Main}). 

\begin{theorem}
	Assume that $q$ is greater than the Coxeter number of $\overline{G_x}$. Then the compact induction induces an equivalence of categories
$$\cInd_{G_x}^{G(F)}: \mathcal{B}_{x,1} \to \mathcal{C}_{x,1}.$$
\end{theorem}

Once this is proven, $\mathcal{C}_{x,1}$ is equivalent to $\mathcal{A}_{x,1}$, hence admits an explicit description. The proof of the Theorem \ref{Thm Main} occupies the most of Chapter \ref{Chapter Rep}. 

\subsection{Proof of the main theorem for the representation side}

In the rest of the section, let us briefly explain the idea of the proof of Theorem \ref{Thm Main}.

The fully faithfulness of 
$$\cInd_{G_x}^{G(F)}: \mathcal{B}_{x,1} \to \mathcal{C}_{x,1}$$
is a usual computation by Frobenius reciprocity and Mackey's formula. Since a similar computation will be used later, we record it in Theorem \ref{Thm Hom}. The key point is that 
$$\Hom_G\left(\cInd_{G_x}^{G(F)}\rho_1, \cInd_{G_y}^{G(F)}\rho_2\right)$$
can be computed explicitly assuming that one of $\rho_1, \rho_2$ has supercuspidal reduction (i.e. $\overline{\rho_1}$ or $\overline{\rho_2}$ is supercuspidal).\footnote{There is a little subtlety that we want not only $\rho$ to have supercuspidal reduction but also any representation $\rho' \in \mathcal{B}_{x,1}$ to have supercuspidal reduction. This subtlety is dealt with in Theorem \ref{Thm SC Red}. This is why we need the \textbf{regular} supercuspidal assumption. }

%\begin{remark}
%	Although supercuspidal representations of $G(F)$ come from \textbf{supercuspidal} representation $\overline{\rho}$ of the finite group of Lie type $\overline{G_x}$, only the \textbf{cuspidality} of $\overline{\rho}$ is relevant in our proof of Theorem \ref{Thm Hom}. That's why we analyze cuspidality instead of supercuspidality in detail for finite group of Lie types in Chapter \ref{Chapter Rep}.
%\end{remark}

The difficulty lies in proving that
$$\cInd_{G_x}^{G(F)}: \mathcal{B}_{x,1} \to \mathcal{C}_{x,1}$$
is essentially surjective. For this, we prove that the compact induction $$\Pi_{x,1}:=\cInd_{G_x}^{G(F)}\sigma_{x,1}$$ 
is a projective generator of $\mathcal{C}_{x,1}$. 

The first key is that $\Pi_{x,1}$ is a summand of a projective generator of a larger category. Indeed, $\Pi_{x,1}$ is a summand of
$$\Pi:=\cInd_{G_x^+}^{G(F)}\Lambda,$$
where $G_x^+$ is the pro-unipotent radical of the parahoric subgroup $G_x$; and $\Pi$ is known to be a projective generator of the category $\Rep_{\Lambda}(G(F))_0$ of depth-zero representations. So we can write 
$$\Pi=\Pi_{x,1} \oplus \Pi^{x,1}.$$

The second key is that the complement $\Pi^{x,1}$ does not interfere with $\Pi_{x,1}$. More precisely, we can compute using Theorem \ref{Thm Hom} that 
$$\Hom_{G}(\Pi^{x,1}, \Pi_{x,1})=\Hom_{G}(\Pi_{x,1}, \Pi^{x,1})=0.$$

The above two keys allow us to conclude that $\Pi_{x,1}$ is a projective generator of $\mathcal{C}_{x,1}$.

\section{The example of $\GL_n$}

To illustrate the theory, we do the example of $\GL_n$ in Chapter \ref{Chapter GL_n}.\footnote{Although $\GL_n$ is not simply connected, the theory still works without much change.} It is quite concise once we have the theories developed in Chapter \ref{Chapter MoLP} and \ref{Chapter Rep}, so let us do not say anything more here. However, the readers are encouraged to do the example of $\GL_2$ throughout the paper, which will help to understand the theories in Chapter \ref{Chapter MoLP} and \ref{Chapter Rep}. 

\section{The categorical local Langlands conjecture for $\GL_n$}

As an application, we will deduce the categorical local Langlands conjecture in Fargues-Scholze's form (see \cite[Conjecture X.3.5]{fargues2021geometrization}) for depth-zero supercuspidal blocks of $\GL_n$ in Chapter \ref{Chapter CLLC}.\footnote{Notice that supercuspidal implies regular supercuspidal automatically in the $\GL_n$ case.}

The idea is that we can unravel both sides of the categorical conjecture explicitly using our computation in Chapter \ref{Chapter GL_n}. They both turn out to be
$$\bigoplus_{\mathbb{Z}}\Perf(\mathbb{G}_m \times \mu).$$ 
We want to show that the spectral action gives an equivalence. This reduces to the degree-zero part (of the $\mathbb{Z}$-grading) by compatibility of the spectral action with the $\mathbb{Z}$-grading (see Proposition \ref{Prop Spectral action}). Finally, the degree-zero part reduces to the theory of local Langlands in families (see \cite{helm2018converse}). Fortunately, several technical results we need to do the reductions are already available by \cite{zou2022categorical}.

\section{Preliminaries and notations}
Although the main chapters \ref{Chapter MoLP}, \ref{Chapter Rep}, \ref{Chapter CLLC} are almost logically independent of each other, let us introduce some notations that are used throughout the thesis.
\begin{enumerate}
	\item Let $G$ be a connected split reductive group.\footnote{The assumption ``split" is intended for simplicity. It is expected that the results generalize to more general reductive groups.} In Chapter \ref{Chapter Rep}, we moreover assume that $G$ is simply connected.\footnote{Again, the assumption ``simply connected" is intended for simplicity. So that we do not need to consider the extension from $G_x$ to its normalizer when doing compact induction.}
	\item Let $\overline{\mathbb{Z}}_{\ell}$ be the integral closure of $\mathbb{Z}_{\ell}$ in $\overline{\mathbb{Q}}_{\ell}$. $\overline{\mathbb{Z}}_{\ell}$ is also the valuation ring of  $\overline{\mathbb{Q}}_{\ell}$, i.e., it consists of the elements with valuations $\geq 0$. $\overline{\mathbb{Z}}_{\ell}$ is a local ring with the unique maximal ideal consisting of elements with valuations $> 0$. $\overline{\mathbb{Z}}_{\ell}$ is strictly henselian, hence all finite étale covers of $\overline{\mathbb{Z}}_{\ell}$ split. An important fact is that any reductive group scheme over $\overline{\mathbb{Z}}_{\ell}$ is split.\footnote{One drawback of $\overline{\mathbb{Z}}_{\ell}$ is that it is not Neotherian. However, in practice, things defined over $\overline{\mathbb{Z}}_{\ell}$ are always already defined over some finite extension $\mathcal{O}$ of $\mathbb{Z}_{\ell}$. So we can first argue for $\mathcal{O}$ and then base change to $\overline{\mathbb{Z}}_{\ell}$. As an aside, the author believes that it does not make a difference if we replace $\overline{\mathbb{Z}}_{\ell}$ by $W(\overline{\mathbb{F}}_{\ell})$, the ring of Witt vectors of $\overline{\mathbb{F}}_{\ell}$, which is the unique complete discrete valuation ring with residue field $\overline{\mathbb{F}}_{\ell}$.}
\end{enumerate}

\section{Logical dependence and suggestions for reading}
Chapter \ref{Chapter MoLP} and \ref{Chapter Rep} are logically independent of each other. Chapter \ref{Chapter CLLC} needs the description of $X_{\varphi}/\hat{G}$ for $G=\GL_n$ obtained in Chapter \ref{Chapter GL_n}, by applying the theory in Chapter \ref{Chapter MoLP} to the case $G=\GL_n$. Also, it might be easier to assume $G=\GL_2$ or $SL_2$ throughout the paper, and one would not miss many essential points in doing so.

\section{Acknowledgements}

It is a pleasure to thank my advisor Peter Scholze for giving me this topic to study and for sharing ideas related to it. Next, I would like to thank my families and friends for their love and support. I thank Johannes Anschütz, Jean-François Dat, Olivier Dudas, Jessica Fintzen, Linus Hamann, Eugen Hellmann, David Helm, Alexander Ivanov, Tasho Kaletha, David Schwein, Vincent Sécherre, Maarten Solleveld, and Marie-France Vignéras for some nice conversations related to this work. I thank my bachelor advisors Kei Yuen Chan and Haining Wang for some helpful suggestions. I thank Anne-Marie Aubert for some helpful discussions and for pointing out some errors. Special thanks go to Jiaxi Mo, Mingjia Zhang, Pengcheng Zhang, Xiaoxiang Zhou, and Konrad Zou for their consistent interest in this work as well as emotional support.

%% file: MoLP.tex
\chapter{TRSELP components of the stack of $L$-parameters} \label{Chapter MoLP}

    Let $\varphi \in Z^1(W_F, \hat{G}(\overline{\mathbb{F}}_{\ell}))$ be a TRSELP (see Definition \ref{Def TRSELP}). In this chapter, we compute the connected component $X_{\varphi}/\hat{G}$ of the stack of $L$-parameters $Z^1(W_F, \hat{G})_{\overline{\mathbb{Z}}_{\ell}}/\hat{G}$ containing $\varphi$. In Section \ref{Section X_phi}, following the theory developed in \cite[Section 3, 4]{dat2022ihes}, we first compute the connected component $X_{\varphi}$ of the space of $1$-cocycles $Z^1(W_F, \hat{G})_{\overline{\mathbb{Z}}_{\ell}}$ (without modulo out the $\hat{G}$-action):
    \begin{equation}\label{Equation X}
    	X_{\varphi} \cong (\hat{G} \times T \times \mu)/T
    \end{equation}
    where $T$ is a split torus and $\mu$ is some group scheme of roots of unity (see Theorem \ref{Thm X} for details). In Section \ref{Section X/hatG}, we use \eqref{Equation X} to obtain a description of $X_{\varphi}/\hat{G}$ under mild conditions (see Theorem \ref{Thm X/G}):
    \begin{equation}
      X_{\varphi}/\hat{G} \cong [*/S_{\psi}] \times \mu,
    \end{equation}
    where $\psi \in Z^1(W_F, \hat{G}(\overline{\mathbb{Z}}_{\ell}))$ is any lift of $\varphi$, and $S_{\psi}:=C_{\hat{G}}(\psi)$ is the centralizer of $\psi$.
    
	\section{The connected component $X_{\varphi}$ containing a TRSELP $\varphi$}\label{Section X_phi}
	
	The goal of this section is to compute the connected component $X_{\varphi}$ of $Z^1(W_F, \hat{G})_{\overline{\mathbb{Z}}_{\ell}}$ containing a TRSELP $\varphi$. In \ref{Subsection MoLP}, we recall the theory of moduli space of Langlands parameters. In \ref{Subsection TRSELP}, we define the class of $L$-parameters that we are interested in -- tame regular semisimple elliptic $L$-parameters (TRSELP for short). In \ref{Subsection the component}, we compute $X_{\varphi}$ explicitly as $(\hat{G}\times T\times \mu)/T$ using the theory of moduli space of Langlands parameters. In \ref{Subsection T-action}, we spell out the $T$-action on $(\hat{G}\times T\times \mu)$ to prepare for the next section.
	
	\subsection{Recollections on the moduli space of Langlands parameters}\label{Subsection MoLP}
	
	Since our computation heavily uses the theory of moduli space of Langlands parameters, we recollect some basic facts here. For more sophisticated knowledge that will be used, we refer to \cite[Section 3, 4]{dat2022ihes}, or \cite[Section 2, 4]{dhkm2020moduli}. 
%	we assume the readers to be familiar with the theory of the moduli space of Langlands parameters, see for example \cite[Section 3 and Section 4]{dat2022ihes}, or \cite[Section 2 and Section 4]{dhkm2020moduli}. 
%	(\textcolor{red}{we can also recollect the theory in the appendix.})
	
	Let us first fix some notations.
	\begin{itemize}
		\item Let $p \neq 2$ be a fixed prime number and $\ell \neq 2$ be a prime number different from $p$. 
		\item Let $F$ be a non-archimedean local field with residue field $\mathbb{F}_q$, where $q=p^r$ for some $r \in \mathbb{Z}_{\geq 1}$.
		\item Let $W_F$ be the Weil group of $F$, $I_F \subseteq W_F$ be the inertia subgroup, and $P_F \subseteq W_F$ be the wild inertia subgroup.
		\item Let $W_t:=W_F/P_F$ be the tame Weil group, $I_t:=I_F/P_F$ be the tame inertia subgroup in $W_t$.
		\item Let $G$ be a connected split reductive group over $F$.
	\end{itemize}
	     Fix $\Fr \in W_F$ any lift of the arithmetic Frobenius element. We will abuse the notation and denote by $\Fr$ the image of $\Fr$ in $W_t$. We have $W_t \cong I_t \rtimes \left<\Fr\right>$. Here, $I_t$ is non-canonically isomorphic to $\prod_{p'\neq p}\mathbb{Z}_{p'}$, which is procyclic. We fix such an isomorphism
	     \begin{equation}\label{Eq I_t}
	     	I_t \cong \prod_{p'\neq p}\mathbb{Z}_{p'}.
	     \end{equation}
    This gives rise to a topological generator $s_0$ of $I_t$, which corresponds to $(1, 1, ...)$ under the isomorphism \eqref{Eq I_t}. Let us recall the following relation in $I_F/P_F$:
	\begin{equation}\label{Eq Fr s_0}
		\Fr s_0 \Fr^{-1}=s_0^q.
	\end{equation}
	In fact, this is true for any $s \in I_t$ instead of $s_0$.
	
	Let 
	$$W_t^0:=\left<s_0, \Fr\right>=\mathbb{Z}[1/p]^{s_0} \rtimes \mathbb{Z}^{\Fr}$$ 
	be the subgroup of $W_t$ generated by $s_0$ and $\Fr$. Let $W_F^0$ denote the preimage of $W_t^0$ under the natural projection $W_F \to W_t$. $W_F^0$ is referred to as the discretization of the Weil group. To summarize, $W_t^0$ is generated by two elements $\Fr$ and $s_0$ with a single relation, i.e., 
	\begin{equation}\label{Equation presentation of the tame Weil group}
		W_t^0=\left<\Fr, s_0\;|\;\Fr s_0 \Fr^{-1}=s_0^q\right>.
	\end{equation} 
	
	Let $G$ be a connected split reductive group over $F$. Let $\hat{G}$ be its dual group over $\mathbb{Z}$. Then the space of cocycles from the discretization
	\begin{equation}\label{Equation space of tame cocycle}
		Z^1(W_t^0, \hat{G})=\underline{\Hom}(W_t^0, \hat{G})=\{(x, y) \in \hat{G} \times \hat{G}\;|\;yxy^{-1}=x^q\}
	\end{equation}
	is an explicit closed subscheme of $\hat{G} \times \hat{G}$ (see \cite[Section 3]{dat2022ihes}). An important fact (see \cite[Proposition 3.9]{dat2022ihes}) is that over a $\mathbb{Z}_{\ell}$-algebra $R$ (the cases $R=\overline{\mathbb{F}}_{\ell}, \overline{\mathbb{Z}}_{\ell}, \overline{\mathbb{Q}}_{\ell}$ are most relevant for us), the restriction from $W_t$ to $W_t^0$ induces an isomorphism
	$$Z^1(W_t, \hat{G}) \cong Z^1(W_t^0, \hat{G}).$$ 
	Therefore, we can compute $Z^1(W_t, \hat{G})$ using the explicit formula \eqref{Equation space of tame cocycle} above. This is fundamental for the study of the moduli space of Langlands parameters $Z^1(W_t, \hat{G})$. We refer the readers to \cite[Section 3, 4]{dat2022ihes} for the precise definition and properties of $Z^1(W_t, \hat{G})$.\footnote{Although we start with a split reductive group, the space of cocycles of certain non-split reductive group would occur when describing the TRSELP component $X_{\varphi}$ of $Z^1(W_F, \hat{G})$ (for example, the space $Z^1_{Ad(\psi)}(W_F, N_{\hat{G}}(\psi_{\ell})^0)$ occurring in the proof of Theorem \ref{Thm X}). We refer the reader to \cite{dat2022ihes} and \cite{dhkm2020moduli} for the definition of $Z^1(W_F, H)$ for general group $H$.} 
	
	\begin{eg}\label{Example GL_1}
		For $G=GL_1$,
	  \begin{equation}
	  \begin{aligned}
		&Z^1(W_t, \hat{G}) \cong Z^1(W_t^0, \hat{G})\\
		=\;&\{(x, y) \in GL_1 \times GL_1\;|\;yxy^{-1}=x^q\}\\
		=\;&\{(x, y) \in GL_1 \times GL_1\;|\;x=x^q\} \cong \mu_{q-1} \times \mathbb{G}_m.
	  \end{aligned}
      \end{equation}
      
      More generally, let $\hat{T}$ be a (possibly non-split) torus equipped with a $W_F$-action. We can compute similarly by tracing the image of $s_0$ and $\Fr$ that
      \begin{equation}\label{Equation: Z^1(W, T)}
      Z^1(W_t, \hat{T}) \cong \hat{T} \times \hat{T}^{\Fr=(-)^q},
      \end{equation} 
      where $\hat{T}^{\Fr=(-)^q}$ is the subscheme of $\hat{T}$ on which $\Fr$ acts by raising to $q$-th power.\footnote{See \cite[Example 3.14]{dat2022ihes} for details. See also the proof of Theorem \ref{Thm X} for an example -- $Z^1_{Ad(\psi)}(W_F, N_{\hat{G}}(\psi_{\ell})^0)$.}
	\end{eg}
	
	\begin{definition}\label{Definition: prime-to-ell inertia}
		Let $I_F^{\ell}$ be the prime-to-$\ell$ inertia subgroup of $W_F$, i.e., $I_F^{\ell}:=\ker(t_{\ell})$, where 
	$$t_\ell: I_F \to I_F/P_F \cong \prod_{p' \neq p}\mathbb{Z}_{p'} \to \mathbb{Z}_\ell$$
	is the composition.
	\end{definition} In other words, $I_F^{\ell}$ is the maximal subgroup of $I_F$ with pro-order prime to $\ell$. This property makes $I_F^{\ell}$ important when determining the connected components of $Z^1(W_F, \hat{G})$ over $\overline{\mathbb{Z}}_{\ell}$ (see \cite[Theorem 4.2 and Subsection 4.6]{dat2022ihes}). 
	
	\subsection{Tame regular semisimple elliptic $L$-parameters}\label{Subsection TRSELP}
	
	We want to define a class of $L$-parameters, called TRSELP, which are expected to correspond to depth-zero regular supercuspidal representations via the local Langlands correspondence. Before that, let us define the concept of schematic centralizer, which will be used throughout the chapter.
	
	\begin{definition}[Schematic centralizer]\label{Definition: Schematic centralizer}
	Let $H$ be an affine algebraic group over a ring $R$, and let $\Gamma$ be a finite group. Let $u \in Z^1(\Gamma, H(R'))$ be a $1$-cocycle for some $R$-algebra $R'$. Let 
	$$\alpha_u: H_{R'} \longrightarrow Z^1(\Gamma, H)_{R'}\qquad h \longmapsto hu(-)h^{-1}$$
	 be the orbit morphism. Then the schematic centralizer $C_H(u)$ is defined as the fiber of $\alpha_u$ at $u$.
	$$	
	\begin{tikzcd}
		{C_H(u)} \arrow[r, ""] \arrow[d, ""] & {H_{R'}} \arrow[d, "{\alpha_u}"] \\
		{\Spec(R')} \arrow[r, "u"]                & {Z^1(\Gamma, H)_{R'}}               
	\end{tikzcd}
	$$	
	\end{definition}
	
	One can show that its $R''$-valued points $C_H(u)(R'')=C_{H(R'')}(u)$ is the set-theoretic centralizer for all $R'$-algebra $R''$, see for example \cite[Appendix A]{dhkm2020moduli}.
	
	\begin{remark}\
		\begin{enumerate}
			\item Note that the above definition is enough for our applications where $\Gamma$ is more generally taken as a profinite group, because $u: \Gamma \to H$ will factor through a finite quotient $\Gamma'$ of $\Gamma$ in practice.
			\item We can define schematic normalizer $N_H(u)$ similarly such that it represents the set-theoretic normalizers (see \cite{conrad2014reductive}).
		\end{enumerate}
			
	\end{remark}
	
	Let us now define a tame, regular semisimple, elliptic Langlands parameter (TRSELP) over $\overline{\mathbb{F}}_{\ell}$, roughly in the sense of \cite[Section 3.4, 4.1]{debacker2009depth}, but with $\overline{\mathbb{F}}_{\ell}$-coefficients instead of $\mathbb{C}$-coefficients.
	
	\begin{definition}\label{Def TRSELP}
		A \textbf{tame regular semisimple elliptic $L$-parameter (TRSELP) over $\overline{\mathbb{F}}_{\ell}$} is a homomorphism $\varphi: W_F \to \hat{G}(\overline{\mathbb{F}}_{\ell})$ such that:
		\begin{enumerate}
			\item (smooth) $\varphi(I_F)$ is a finite subgroup of $\hat{G}(\overline{\mathbb{F}}_{\ell})$.
			\item (Frobenius semisimple) $\varphi(\Fr)$ is a semisimple element of $\hat{G}(\overline{\mathbb{F}}_{\ell})$.
			\item (tame) The restriction of $\varphi$ to $P_F$ is trivial.
			\item \label{regular semisimple}(regular semisimple) The centralizer of the inertia $C_{\hat{G}}(\varphi|_{I_F})$ is a torus over $\overline{\mathbb{F}}_{\ell}$ (in particular, connected).
			\item \label{elliptic} (elliptic) The identity component $C_{\hat{G}}(\varphi)^0$ of the centralizer $C_{\hat{G}}(\varphi)$ is equal to the identity  component $Z(\hat{G})^0$ of the center $Z(\hat{G})$.
			\end{enumerate}
	\end{definition}

    Concretely, a TRSELP consists of the following data:
    
    \begin{enumerate}
    	\item The restriction to the inertia $\varphi|_{I_F}$, which is essentially a direct sum of characters of some $\mathbb{F}_{q^n}^*$.\footnote{Think about the example of $GL_n$. For general $G$, choose a faithful embedding $\hat{G} \hookrightarrow GL(V)$.} Indeed, $I_F/P_F \cong \varprojlim\mathbb{F}_{q^n}^*$ and
    	$$\Hom_{\Cont}(I_F/P_F, \overline{\mathbb{F}}_{\ell}^*) \cong \Hom_{\Cont}(\varprojlim\mathbb{F}_{q^n}^*, \overline{\mathbb{F}}_{\ell}^*) \cong \varinjlim\Hom_{\Cont}(\mathbb{F}_{q^n}^*, \overline{\mathbb{F}}_{\ell}^*).$$
    	In particular, it factors through (the $\overline{\mathbb{F}}_{\ell}$-points of) some maximal torus, say $S$. Then $\varphi$ being regular semisimple means that $C_{\hat{G}(\overline{\mathbb{F}}_{\ell})}(\varphi(I_F))=S$.
    	\item The image of the Frobenius $\varphi(\Fr)$, which turns out to be an element of the normalizer $N_{\hat{G}(\overline{\mathbb{F}}_{\ell})}(S)$ (Since $\Fr.s.\Fr^{-1}=s^q \in I_t$ for any $s \in I_t$ implies that $\varphi(\Fr)$ normalizes $C_{\hat{G}(\overline{\mathbb{F}}_{\ell})}(\varphi(I_F))=S$.). 
        In addition, ``elliptic" means that the center $Z(\hat{G})$ has finite index in the centralizer $C_{\hat{G}}(\varphi)$. As we will see later, when $Z(\hat{G})$ is finite, ellipticity implies that $\hat{G}(\overline{\mathbb{F}}_{\ell})$ acts transitively on the connected component $X_{\varphi}(\overline{\mathbb{F}}_{\ell})$ of the moduli space of $L$-parameters containing $\varphi$ (see the proof of Lemma \ref{Lem epic}), which is essential for the description
        $X_{\varphi}/\hat{G} \cong [*/S_{\psi}]$ in Theorem \ref{Thm X/G}.
%        (roughly, see Theorem \ref{Thm X/G} for the precise statement)
%    	$$X_{\varphi}/\hat{G} \cong [*/S_{\varphi}],$$
%    	where $S_\varphi=C_{\hat{G}(\overline{\mathbb{F}}_{\ell})}(\varphi(W_F))$ is the centralizer of the whole $L$-parameter $\varphi$.
    \end{enumerate}

    \begin{eg}
    	For $G=GL_n$, a TRSELP is the same as an irreducible tame $L$-parameter. See Section \ref{Example Lparam} for the irreducible tame $L$-parameters of $GL_n$ expressed in explicit matrices.
    \end{eg}

  \begin{remark}\
  	\begin{enumerate}
  		\item Let $A \in \{\overline{\mathbb{Z}}_{\ell}, \overline{\mathbb{Q}}_{\ell}, \overline{\mathbb{F}}_{\ell}\}$. It is important for our purpose to distinguish between the set-theoretic centralizer (for example, $C_{\hat{G}(A)}(\varphi(W_F))$) and the schematic centralizer (for example, $C_{\hat{G}}(\varphi)$). However, we might still use $\hat{G}$ to mean $\hat{G}(A)$ sometimes by abuse of notation, which we hope the reader can recognize. One reason for doing so is that $\hat{G}$ is split over $A$; hence, $\hat{G}$ is completely determined by its $A$-points. Many statements can either be phrased in terms of the $A$-scheme or its $A$-points (for example, \ref{regular semisimple} and \ref{elliptic} in Definition \ref{Def TRSELP}).
  		\item As we will see later in Theorem \ref{Thm X}, $S=C_{\hat{G}(\overline{\mathbb{F}}_{\ell})}(\varphi(I_F))$ turns out to be the $\overline{\mathbb{F}}_{\ell}$-points of the split torus $T=C_{\hat{G}}(\psi|_{I_F^{\ell}})$ for any lift $\psi$ of $\varphi$ over $\overline{\mathbb{Z}}_{\ell}$.
  	\end{enumerate}
  \end{remark}

\subsection{Description of the component}\label{Subsection the component}

Now let us fix a TRSELP $\varphi \in Z^1(W_F, \hat{G}(\overline{\mathbb{F}}_{\ell}))$. Pick any lift $\psi \in Z^1(W_F, \hat{G}(\overline{\mathbb{Z}}_{\ell}))$ of $\varphi$, whose existence is ensured by the flatness of $Z^1(W_F, \hat{G})_{\overline{\mathbb{Z}}_{\ell}}$ (see Lemma \ref{Lem generalizing}). Let $\psi_{\ell}:=\psi|_{I_F^{\ell}}$ denote the restriction of $\psi$ to the prime-to-$\ell$ inertia $I_F^{\ell}$. Note that $\psi \in Z^1(W_F, \hat{G})$ factors through the schematic normalizer $N_{\hat{G}}(\psi_{\ell})$ (since $I_F^{\ell}$ is normal in $W_F$), so we can view $\psi$ as a point in $Z^1(W_F, N_{\hat{G}}(\psi_{\ell}))$. Let $\overline{\psi}:=j \circ \psi \in Z^1(W_F, \pi_0(N_{\hat{G}}(\psi_{\ell})))$, where $j: N_{\hat{G}}(\psi_{\ell}) \to \pi_0(N_{\hat{G}}(\psi_{\ell}))$ is the projection. Let $X_{\varphi}$ be the connected component of $Z^1(W_F, \hat{G})_{\overline{\mathbb{Z}}_{\ell}}$ containing $\varphi$. Note that $X_{\varphi}$ also contains $\psi$ since $\psi$ specializes to $\varphi$. Therefore,  we sometimes also denote $X_{\varphi}$ as $X_{\psi}$. Such a component is referred to as a TRSELP component.

We shall compute $X_{\psi}$ directly using the theory developed in \cite[Section 4]{dat2022ihes}. One might want to work out the example of $GL_2$ (see Example \ref{Example: GL_2}) to understand what is happening below.

It turns out that the component $X_{\varphi}=X_{\psi}$ of $Z^1(W_F, \hat{G})_{\overline{\mathbb{Z}}_{\ell}}$ consists of the $L$-parameters whose restriction to $I_F^{\ell}$ and whose image in $Z^1(W_F, \pi_0(N_{\hat{G}}(\psi_{\ell})))$ is $\hat{G}$-conjugate to $(\psi_{\ell}, \overline{\psi})$. This is the content of the next lemma.

\begin{lemma}\label{Lemma: X}
%	Assume the center $Z(\hat{G})$ is smooth over $\overline{\mathbb{Z}}_{\ell}$.\footnote{This is to make sure that the results in \cite[Section 5.4, 5.5]{dat2022ihes} apply.} 
	We have an isomorphism of schemes
	\begin{equation}\label{Equation: X_psi_dat_4.6}
		X_{\psi} \cong \left(\hat{G} \times Z^1(W_F, N_{\hat{G}}(\psi_{\ell}))_{\psi_{\ell}, \overline{\psi}}\right)/C_{\hat{G}}(\psi_{\ell})_{\overline{\psi}} \qquad g\eta(-)g^{-1} \mapsfrom (g, \eta),
	\end{equation}
	where $Z^1(W_F, N_{\hat{G}}(\psi_{\ell}))_{\psi_{\ell}, \overline{\psi}}$  denotes the space of cocycles whose restriction to $I_F^{\ell}$ equals $\psi_{\ell}$ and whose image in $Z^1(W_F, \pi_0(N_{\hat{G}}(\psi_{\ell})))$ is $\overline{\psi}$; where $C_{\hat{G}}(\psi_{\ell})_{\overline{\psi}}$ is the (schematic) stabilizer of $\overline{\psi}$ in $C_{\hat{G}}(\psi_{\ell})$; and where $C_{\hat{G}}(\psi_{\ell})_{\overline{\psi}}$ acts on $\hat{G} \times Z^1(W_F, N_{\hat{G}}(\psi_{\ell}))_{\psi_{\ell}, \overline{\psi}}$ by 
	$$(t, (g, \psi')) \mapsto (gt^{-1}, t\psi'(-)t^{-1}),$$
	where $t \in C_{\hat{G}}(\psi_{\ell})_{\overline{\psi}}$ and $(g, \psi') \in \hat{G} \times Z^1(W_F, N_{\hat{G}}(\psi_{\ell}))_{\psi_{\ell}, \overline{\psi}}$.
\end{lemma}

\begin{proof}
	This is proven in \cite[Subsection 4.6]{dat2022ihes}.\footnote{To apply \cite[Subsection 4.6]{dat2022ihes}, we need to ensure that the center of $H^0$ in Dat's notation is smooth over $\overline{\mathbb{Z}}_{\ell}$ so that we can find a lifting $\psi$ such that $\psi$ fixes a Borel pair of $H^0$. However, this is automatic in our case because here $H^0=N_{\hat{G}}(\psi_{\ell})^0$ will turn out to be a torus (see Lemma \ref{Lemma: T}).} As a rough outline, we first notice that $Z^1(W_F, \hat{G})_{[\psi_{\ell}]}$, the space of cocycles whose restriction to $I_F^{\ell}$ is $\hat{G}$-conjugate to $\psi_{\ell}$, is open and closed in $Z^1(W_F, \hat{G})$\footnote{This is done by considering the restriction map $Z^1(W_F, \hat{G}) \to Z^1(I_F^{\ell}, \hat{G})$, since we know the connected components of $Z^1(I_F^{\ell}, \hat{G})$ quite well, thanks to \cite[Theorem 4.2]{dat2022ihes}. Note that $I_F^{\ell}$ has pro-order prime to $\ell$, so that we can apply \cite[Theorem 4.2]{dat2022ihes}. This is the reason that we consider $I_F^{\ell}$, the maximal subgroup of $W_F$ with pro-order prime to $\ell$.}. Next, we notice that
	$g\eta(-)g^{-1} \mapsfrom (g, \eta)$
    defines an isomorphism
    $$Z^1(W_F, \hat{G})_{[\psi_{\ell}]} \cong \left(\hat{G} \times Z^1(W_F, \hat{G})_{\psi_{\ell}}\right)/C_{\hat{G}}(\psi_{\ell}),$$
    where $Z^1(W_F, \hat{G})_{\psi_{\ell}}$ is the space of cocycles whose restriction to $I_F^{\ell}$ is $\psi_{\ell}$. Thus, $\psi$ is contained in the open and closed subscheme $Z^1(W_F, \hat{G})_{[\psi_{\ell}]}$ of $Z^1(W_F, \hat{G})$. However, $Z^1(W_F, \hat{G})_{[\psi_{\ell}]}$ is usually not connected, since it maps to the discrete space 
    $$Z^1(W_F/P_F, \pi_0(N_{\hat{G}}(\psi_{\ell}))).$$
    Nevertheless, it turns out that this is the only obstruction for being connected, i.e., if we moreover consider the subspace of $Z^1(W_F, \hat{G})_{[\psi_{\ell}]}$ that consists of $L$-parameters whose image in $Z^1(W_F/P_F, \pi_0(N_{\hat{G}}(\psi_{\ell})))$ is $\hat{G}$-conjugate to $\overline{\psi}$, it becomes connected. Therefore, we obtain the desired formular \eqref{Equation: X_psi_dat_4.6}.
\end{proof}

To prepare for the main theorem of this subsection, let us do a lemma. We introduce the notation $T:=C_{\hat{G}}(\psi|_{I_F^{\ell}})$, which will be used throughout the rest of the chapter. $T$ has several equivalent definitions, as follows.

\begin{lemma}\label{Lemma: T}\
	\begin{enumerate}
		\item $T:=C_{\hat{G}}(\psi|_{I_F^{\ell}})$ is a split torus over $\overline{\mathbb{Z}}_{\ell}$. 
		\item The following equalities hold:
		\begin{equation}\label{Equation: T}
			C_{\hat{G}}(\psi|_{I_F^{\ell}}) = C_{\hat{G}}(\psi|_{I_F^{\ell}})^0 = N_{\hat{G}}(\psi|_{I_F^{\ell}})^0 = C_{\hat{G}}(\psi|_{I_F}).
		\end{equation}
		\item The schematic stabilizer $C_{\hat{G}}(\psi_{\ell})_{\overline{\psi}}$ of $\overline{\psi}$ in $C_{\hat{G}}(\psi_{\ell})$ is equal to $C_{\hat{G}}(\psi_{\ell})$.
	\end{enumerate}
\end{lemma}

\begin{proof}\
	\begin{enumerate}
		\item By \cite[Subsection 3.1]{dat2022ihes}, the centralizer $C_{\hat{G}}(\psi_{\ell})$ is generalized reductive (see Lemma \ref{Lem gen red}), hence split over $\overline{\mathbb{Z}}_{\ell}$. Therefore, we can determine $C_{\hat{G}}(\psi_{\ell})$ by computing its $\overline{\mathbb{F}}_{\ell}$-points. Indeed,
		$$C_{\hat{G}}(\psi_{\ell})(\overline{\mathbb{F}}_{\ell})=C_{\hat{G}(\overline{\mathbb{F}}_{\ell})}(\varphi(I_F^\ell))=C_{\hat{G}(\overline{\mathbb{F}}_{\ell})}(\varphi(I_F)),$$
		where the last equality follows since $I_F/I_F^{\ell}$ does not contribute to the image of $\varphi$ (see Lemma \ref{Lem I_F^ell}). Therefore, since $S=C_{\hat{G}(\overline{\mathbb{F}}_{\ell})}(\varphi(I_F))$ is a torus by assumption \ref{regular semisimple} in Definition \ref{Def TRSELP}, $C_{\hat{G}}(\psi_{\ell})$ is a split torus over $\overline{\mathbb{Z}}_{\ell}$ with $\overline{\mathbb{F}}_{\ell}$-points $S$.
		\item We just proved that $C_{\hat{G}}(\psi|_{I_F^{\ell}})$ is a split torus over $\overline{\mathbb{Z}}_{\ell}$ (in particular, connected), hence the first equality in \eqref{Equation: T}. For the second equality, we refer to \cite[Subsection 3.1]{dat2022ihes}. It remains to show the last equality. We will show that $C_{\hat{G}}(\psi|_{I_F^{\ell}})=C_{\hat{G}}(\psi|_{I_F})$. We first note that $C_{\hat{G}}(\psi|_{I_F}) \subseteq C_{\hat{G}}(\psi|_{I_F^{\ell}})=:T$ is an inclusion to a commutative group scheme. Since $\psi(I_F)=\psi(I_F/P_F)$ is abelian, 
		$$\psi(I_F) \subseteq C_{\hat{G}}(\psi|_{I_F}) \subseteq T.$$
		Therefore, $C_{\hat{G}}(\psi|_{I_F}) \supseteq T$ since $T$ is commutative, hence
		$$C_{\hat{G}}(\psi|_{I_F}) = T.$$
		\item Recall that $C_{\hat{G}}(\psi_{\ell})$ acts on $N_{\hat{G}}(\psi_{\ell})$ by conjugation, inducing an action of $C_{\hat{G}}(\psi_{\ell})$ on $\pi_0(N_{\hat{G}}(\psi_{\ell}))$ by \eqref{Equation: T}, and hence an action on $Z^1(W_F, \pi_0(N_{\hat{G}}(\psi_{\ell})))$ by conjugation. In addition, $C_{\hat{G}}(\psi_{\ell})_{\overline{\psi}}$ is by definition the stabilizer of $\overline{\psi} \in Z^1(W_F, \pi_0(N_{\hat{G}}(\psi_{\ell})))$ in $C_{\hat{G}}(\psi_{\ell})$. Now $C_{\hat{G}}(\psi_{\ell})=T$ is connected, hence acts trivially on the component group $\pi_0(N_{\hat{G}}(\psi_{\ell}))$ and acts trivially on $Z^1(W_F, \pi_0(N_{\hat{G}}(\psi_{\ell})))$. Therefore, the stabilizer $C_{\hat{G}}(\psi_{\ell})_{\overline{\psi}}=C_{\hat{G}}(\psi_{\ell})$.
	\end{enumerate}
	
\end{proof}

\begin{theorem}\label{Thm X}
	Let $\varphi \in Z^1(W_F, \hat{G}(\overline{\mathbb{F}}_{\ell}))$ be a TRSELP over $\overline{\mathbb{F}}_{\ell}$. Let $\psi \in Z^1(W_F, \hat{G}(\overline{\mathbb{Z}}_{\ell}))$ be any lifting of $\varphi$. 
%	Assume that the center $Z(\hat{G})$ is smooth over $\overline{\mathbb{Z}}_{\ell}$. 
	Then the connected component $X_{\varphi}=X_{\psi}$ of $Z^1(W_F, \hat{G})_{\overline{\mathbb{Z}}_{\ell}}$ containing $\varphi$ is isomorphic to 
%	$$\left(\hat{G} \times C_{\hat{G}}(\psi_{\ell})^0 \times \mu\right)/\;C_{\hat{G}}(\psi_{\ell})_{\overline{\psi}},$$
    $$\left(\hat{G} \times T \times \mu\right)/T,$$
	where $T=C_{\hat{G}}(\psi_{\ell})$, and 
	$\mu:=\left(T^{\Fr=(-)^q}\right)^0$ is the connected component of $T^{\Fr=(-)^q}$ containing $1$, where
	$$T^{\Fr=(-)^q}:=\{t \in T \;|\; \psi(\Fr)t\psi(\Fr)^{-1} \;=\; t^q\}$$
	is the subscheme of $T$ on which $\Fr$ acts (via $Ad(\psi)$) by raising to $q$-th power (see Equation \eqref{Equation: Z^1(W, T)}). Moreover, $\mu$ is a product of some $\mu_{\ell^{k_i}}$ (the group scheme of $\ell^{k_i}$-th roots of unity over $\overline{\mathbb{Z}}_{\ell}$), $k_i \in \mathbb{Z}_{\geq 0}$.\footnote{Note that $\mu$ can be trivial, depending on $\hat{G}$ and some congruence relations between $q, \ell$.} We will specify in the next subsection what the $T$-action on $\hat{G} \times T \times \mu$ is.
%	\begin{enumerate}
%		\item $C_{\hat{G}}(\psi_{\ell})^0$ is the identity component of the schematic centralizer $C_{\hat{G}}(\psi_{\ell})$. In addition, $C_{\hat{G}}(\psi_{\ell})=C_{\hat{G}}(\psi_{\ell})^0$ is a split torus $T$ over $\overline{\mathbb{Z}}_{\ell}$ with $\overline{\mathbb{F}}_{\ell}$-points $S=C_{\hat{G}(\overline{\mathbb{F}}_{\ell})}(\varphi(I_F))$.
%		\item $\mu:=\left(T^{\Fr=(-)^q}\right)^0$ is the identity component of $T^{\Fr=(-)^q}$ \footnote{This is the subscheme of $T$ on which $\Fr$ acts by raising to $q$-th power, see Equation \eqref{Equation: Z^1(W, T)}. See also \cite[Example 3.14]{dat2022ihes}.} containing $1$, which is a product of some $\mu_{\ell^{k_i}}$ (the group scheme of $\ell^{k_i}$-th roots of unity over $\overline{\mathbb{Z}}_{\ell}$), $k_i \in \mathbb{Z}_{\geq 0}$. \footnote{Note that $\mu$ can be trivial, depending on $\hat{G}$ and some congruence relations between $q, \ell$.}
%		\item $C_{\hat{G}}(\psi_{\ell})_{\overline{\psi}}$ is the (schematic) stabilizer of $\overline{\psi}$ in $C_{\hat{G}}(\psi_{\ell})$.
%	\end{enumerate}
%    In other words, we have the following isomorphism of schemes over $\overline{\mathbb{Z}}_{\ell}$:
%    $$X_{\varphi} \cong \left(\hat{G} \times T \times \mu\right)/T.$$

    \begin{proof}
    	We begin with an outline of the proof. The idea is to use the formula \eqref{Equation: X_psi_dat_4.6}. 
%    	We shall express the terms $Z^1(W_F, N_{\hat{G}}(\psi_{\ell}))_{\psi_{\ell}, \overline{\psi}}$ and $C_{\hat{G}}(\psi_{\ell})_{\overline{\psi}}$ on the right hand side of \eqref{Equation: X_psi_dat_4.6} explicitly. 
    	\begin{enumerate}
    		\item We show that $Z^1(W_F, N_{\hat{G}}(\psi_{\ell}))_{\psi_{\ell}, \overline{\psi}}$ is isomorphic to $Z^1_{Ad(\psi)}(W_F, N_{\hat{G}}(\psi_{\ell})^0)_{1_{I_F^{\ell}}}$, the space of cocycles (see (\textbf{Step 1}) below) whose restrictions to $I_F^{\ell}$ are trivial.
%    		\item We show that $C_{\hat{G}}(\psi_{\ell})$ is a split torus over $\overline{\mathbb{Z}}_{\ell}$ and that
%            $N_{\hat{G}}(\psi_{\ell})^0=C_{\hat{G}}(\psi_{\ell})^0=C_{\hat{G}}(\psi_{\ell})$.
            \item We know by Lemma \ref{Lemma: T} that $N_{\hat{G}}(\psi_{\ell})^0=T$ is a torus, so we can compute \\$Z^1_{Ad(\psi)}(W_F/P_F, T)$ as in Example \ref{Example GL_1}.
            \item We compute $Z^1_{Ad(\psi)}(W_F, T)_{1_{I_F^{\ell}}}$ as the identity component of $Z^1_{Ad(\psi)}(W_F/P_F, T)$.
%            \item We show that $C_{\hat{G}}(\psi_{\ell})_{\overline{\psi}}=C_{\hat{G}}(\psi_{\ell})$.
            \item We combine with Lemma \ref{Lemma: X} and \ref{Lemma: T} to conclude.
    	\end{enumerate}

    	(\textbf{Step 1}) We show that $\eta.\psi \mapsfrom \eta$ defines an isomorphism
    	$$Z^1(W_F, N_{\hat{G}}(\psi_{\ell}))_{\psi_{\ell}, \overline{\psi}} \cong Z^1_{Ad(\psi)}(W_F, N_{\hat{G}}(\psi_{\ell})^0)_{1_{I_F^{\ell}}}=:Z^1_{Ad(\psi)}(W_F, N_{\hat{G}}(\psi_{\ell})^0)_1,$$
    	where $Z^1_{Ad(\psi)}(W_F, N_{\hat{G}}(\psi_{\ell}))$ means the space of cocycles with $W_F$ acting on $N_{\hat{G}}(\psi_{\ell})$ via conjugacy action through $\psi$,\footnote{i.e., $W_F \times N_{\hat{G}}(\psi_{\ell}) \longrightarrow N_{\hat{G}}(\psi_{\ell}) \qquad (w, n) \mapsto {^wn}:=\psi(w)n\psi(w)^{-1}.$} 
    	and the subscript $1_{I_F^{\ell}}$ or $1$ means the cocycles whose restrictions to $I_F^{\ell}$ are trivial.\footnote{i.e.,
    	$Z^1_{Ad(\psi)}(W_F, N_{\hat{G}}(\psi_{\ell}))_{1}=\{\psi': W_F \to N_{\hat{G}}(\psi_{\ell}) \;|\; \psi'(w_1w_2)=\psi'(w_1)\left(^{w_1}\psi'(w_2)\right), \psi'|_{I_F^{\ell}}=1_{I_F^{\ell}}\}.$}
    	Indeed, this is clear by unraveling the definitions: two cocycles whose restriction to $I_F^\ell$ are both $\psi_{\ell}$ differ by something whose restriction to $I_F^{\ell}$ is trivial; two cocycles whose pushforward to $Z^1(W_F, \pi_0(N_{\hat{G}}(\psi_{\ell})))$ are both $\overline{\psi}$ differ by something whose pushforward to $Z^1(W_F, \pi_0(N_{\hat{G}}(\psi_{\ell})))$ is trivial, i.e., which factors through the identity component $N_{\hat{G}}(\psi_{\ell})^0$. The twisted action $Ad(\psi)$ occurs naturally when one tries to show that $\eta.\psi \mapsfrom \eta$ is well defined.
    	
%    	(\textbf{Step 2}) We show that $C_{\hat{G}}(\psi_{\ell})$ is a split torus over $\overline{\mathbb{Z}}_{\ell}$ and that $N_{\hat{G}}(\psi_{\ell})^0=C_{\hat{G}}(\psi_{\ell})^0=C_{\hat{G}}(\psi_{\ell})$. By \cite[Subsection 3.1]{dat2022ihes}, the centralizer $C_{\hat{G}}(\psi_{\ell})$ is generalized reductive (see Lemma \ref{Lem gen red}), hence split over $\overline{\mathbb{Z}}_{\ell}$, and $N_{\hat{G}}(\psi_{\ell})^0=C_{\hat{G}}(\psi_{\ell})^0$. Therefore, we can determine $C_{\hat{G}}(\psi_{\ell})$ by computing its $\overline{\mathbb{F}}_{\ell}$-points. Indeed,
%    	$$C_{\hat{G}}(\psi_{\ell})(\overline{\mathbb{F}}_{\ell})=C_{\hat{G}(\overline{\mathbb{F}}_{\ell})}(\varphi(I_F^\ell))=C_{\hat{G}(\overline{\mathbb{F}}_{\ell})}(\varphi(I_F)),$$
%    	where the last equality follows since $I_F/I_F^{\ell}$ does not contribute to the image of $\varphi$ (see Lemma \ref{Lem I_F^ell}). Therefore, $C_{\hat{G}}(\psi_{\ell})$ is a split torus over $\overline{\mathbb{Z}}_{\ell}$ with $\overline{\mathbb{F}}_{\ell}$-points $S=C_{\hat{G}(\overline{\mathbb{F}}_{\ell})}(\varphi(I_F))$. Denote $T=C_{\hat{G}}(\psi_{\ell})$. In particular, $C_{\hat{G}}(\psi_{\ell})$  is connected; hence, 
%    	\begin{equation}\label{Equation: T}
%    	N_{\hat{G}}(\psi_{\ell})^0=C_{\hat{G}}(\psi_{\ell})^0=C_{\hat{G}}(\psi_{\ell})=T.
%    	\end{equation}

    	(\textbf{Step 2}) As in Example \ref{Example GL_1}, we compute that
    	$$Z^1_{Ad(\psi)}(W_F/P_F, N_{\hat{G}}(\psi_{\ell})^0)=Z^1_{Ad(\psi)}(W_F/P_F, T) \cong T \times T^{\Fr=(-)^q},$$
    	where $T^{\Fr=(-)^q}$ is the subscheme of $T$ on which $\Fr$ acts (via $Ad(\psi)$) by raising to $q$-th power, see Equation \eqref{Equation: Z^1(W, T)}; where the last isomorphism is given by $\eta \mapsto (\eta(\Fr), \eta(s_0))$.
    	
    	(\textbf{Step 3}) We show that the identity component $(T^{\Fr=(-)^q})^0$ of $T^{\Fr=(-)^q}$ gives $\mu$ in the statement of the theorem. Note that $T^{\Fr=(-)^q}$ is a diagonalizable group scheme over $\overline{\mathbb{Z}}_{\ell}$ of dimension zero (this can be seen either by $\dim Z^1(W_F/P_F, T)=\dim T$, or by noticing that $\eta(s_0) \in T^{\Fr=(-)^q}$ is semisimple with finitely many possible eigenvalues), hence of the form $\prod_i\mu_{n_i}$ for some $n_i \in \mathbb{Z}_{\geq 0}$. Thus, its identity component $(T^{\Fr=(-)^q})^0$ over $\overline{\mathbb{Z}}_{\ell}$ is of the form $\prod_i\mu_{\ell^{k_i}},$ with $k_i$ maximal such that $\ell^{k_i}$ divides $n_i$. Therefore, 
    	\begin{equation}\label{Equation: Z^1_1}
    		Z^1_{Ad(\psi)}(W_F, N_{\hat{G}}(\psi_{\ell})^0)_1 \cong (T \times T^{\Fr=(-)^q})^0 \cong T \times (T^{\Fr=(-)^q})^0 \cong T \times \mu,
    	\end{equation}
    	(see Lemma \ref{Lem_Z^1()_1} for the first isomorphism) where $\mu$ is of the form $\prod_i\mu_{\ell^{k_i}}$.
    	
%    	(\textbf{Step 5}) We show that $C_{\hat{G}}(\psi_{\ell})_{\overline{\psi}}=C_{\hat{G}}(\psi_{\ell})$. Recall that $C_{\hat{G}}(\psi_{\ell})$ acts on $Z^1(W_F, N_{\hat{G}}(\psi_{\ell}))$ by conjugation, inducing an action of $C_{\hat{G}}(\psi_{\ell})$ on $Z^1(W_F, \pi_0(N_{\hat{G}}(\psi_{\ell}))).$ In addition, $C_{\hat{G}}(\psi_{\ell})_{\overline{\psi}}$ is by definition the stabilizer of $\overline{\psi} \in Z^1(W_F, \pi_0(N_{\hat{G}}(\psi_{\ell})))$ in $C_{\hat{G}}(\psi_{\ell})$. Now $C_{\hat{G}}(\psi_{\ell})=T$ is connected, hence acting trivially on the component group $\pi_0(N_{\hat{G}}(\psi_{\ell}))$ and acting trivially on $Z^1(W_F, \pi_0(N_{\hat{G}}(\psi_{\ell})))$. Therefore, the stabilizer $C_{\hat{G}}(\psi_{\ell})_{\overline{\psi}}=C_{\hat{G}}(\psi_{\ell})$.
    	
    	(\textbf{Step 4})
    	Combining \eqref{Equation: Z^1_1} with Lemma \ref{Lemma: X} and Lemma \ref{Lemma: T}, we conclude that
        \begin{equation}
    		X_{\varphi} \cong (\hat{G} \times Z^1_{Ad(\psi)}(W_F, N_{\hat{G}}(\psi_{\ell})^0)_1)/C_{\hat{G}}(\psi_{\ell})_{\overline{\psi}} \cong (\hat{G} \times T \times \mu)/T.
        \end{equation}
    \end{proof}
\end{theorem}

\begin{eg}\label{Example: GL_2}
	
    Let $p=q=11, \ell=5, G=GL_2$.\footnote{They are chosen such that $\mu$ turns out to be non-trivial.} Let $F_2$ be the unique degree $2$ unramified extension of $F$. Then the Weil group of $F_2$ is $W_{F_2} \cong I_F \rtimes \left<\Fr^2\right>$.
    
    We define a tame character $\eta: W_{F_2}/P_F \to \overline{\mathbb{F}}_{\ell}^*$ as follows. It suffices to define $\eta$ on $I_F/P_F$ and $\left<\Fr^2\right>$ respectively.
    Let 
    $$\eta|_{I_F/P_F}: I_F/P_F \cong \prod_{p'\neq 11}\mathbb{Z}_{p'} \to \mathbb{Z}_3 \to \mathbb{Z}/3\mathbb{Z} \to \overline{\mathbb{F}}_{5}^*$$
    be the composition, where the last map is a non-trivial character $\chi: \mathbb{Z}/3\mathbb{Z} \to \overline{\mathbb{F}}_{5}^*$. Let $\eta(\Fr^2):=1$.
    
    Let $\varphi:=\Ind_{W_{F_2}}^{W_F}\eta$. $\varphi \in Z^1(W_F, \hat{G}(\overline{\mathbb{F}}_{\ell}))$ is an irreducible tame $L$-parameter, hence a TRSELP of $G=GL_2$. 
    
    To compute the connected component of $Z^1(W_F, \hat{G})$ containing $\varphi$ over $\overline{\mathbb{Z}}_{\ell}$, let us choose a lift $\psi$ of $\varphi$, as follows. First, let us define a lift $\tilde{\eta}: W_{F_2}/P_F \to \overline{\mathbb{Z}}_{\ell}^*$ of $\eta$, as follows. Let 
    $$\tilde{\eta}|_{I_F/P_F}: I_F/P_F \cong \prod_{p'\neq 11}\mathbb{Z}_{p'} \to \mathbb{Z}_3 \to \mathbb{Z}/3\mathbb{Z} \to \overline{\mathbb{Z}}_{5}^*$$
    be the composition, where the last map is a non-trivial character $\tilde{\chi}: \mathbb{Z}/3\mathbb{Z} \to \overline{\mathbb{Z}}_{5}^*$ lifting $\chi$. Let $\tilde{\eta}(\Fr^2):=1$. Next, define $\psi:=\Ind_{W_{F_2}}^{W_F}\tilde{\eta}$.
    
    Under a nice basis, we can express $\psi: W_F \to GL_2(\overline{\mathbb{Z}}_{\ell})$ in terms of matrices, as follows:
    $$\psi(s_0)=
    \begin{pmatrix}
    	\tilde{\chi}(1) & 0 \\
    	0 & \tilde{\chi}^q(1) \\
    \end{pmatrix}
    =
    \begin{pmatrix}
    	\zeta_3 & 0 \\
    	0 & \zeta_3^2 \\
    \end{pmatrix}\qquad 
    \psi(\Fr)=
    \begin{pmatrix}
    	0 & 1 \\
    	\tilde{\eta}(\Fr^2) & 0 \\
    \end{pmatrix}
    =
    \begin{pmatrix}
    	0 & 1 \\
    	1 & 0 \\
    \end{pmatrix},
    $$
    where $\zeta_3$ is a primitive $3$-rd root of unity of $\overline{\mathbb{Z}}_{\ell}$.
    
    Recall that
    $$X_{\varphi} \cong (\hat{G} \times Z^1_{Ad(\psi)}(W_F, N_{\hat{G}}(\psi_{\ell})^0)_1)/C_{\hat{G}}(\psi_{\ell})_{\overline{\psi}}.$$
    
    We see that 
    $$\psi(I_F^{\ell})=\left\{
    \begin{pmatrix}
    	1 & 0 \\
    	0 & 1 \\
    \end{pmatrix},
    \begin{pmatrix}
    	\zeta_3 & 0 \\
    	0 & \zeta_3^2 \\
    \end{pmatrix},
    \begin{pmatrix}
    	\zeta_3^2 & 0 \\
    	0 & \zeta_3 \\
    \end{pmatrix}
    \right\}.$$ 
    In this case, $T=C_{\hat{G}}(\psi_{\ell})$ is the diagonal torus of $GL_2$, $N_{\hat{G}}(\psi_{\ell})$ is the normalizer of $T$, $N_{\hat{G}}(\psi_{\ell})^0=T$, and $C_{\hat{G}}(\psi_{\ell})_{\overline{\psi}}=T$ since $T=C_{\hat{G}}(\psi_{\ell})$ fixes $\overline{\psi}$.
    
    It remains to compute $Z^1_{Ad(\psi)}(W_F, N_{\hat{G}}(\psi_{\ell})^0)_1 \cong Z^1_{Ad(\psi)}(W_F, T)_1$. We first compute $Z^1_{Ad(\psi)}(W_F/P_F, T)$ (without the subscript $1$). Indeed, by Equation \eqref{Equation: Z^1(W, T)}, 
    $$Z^1_{Ad(\psi)}(W_F/P_F, T) \cong T \times T^{\Fr=(-)^q},$$
    where 
    $$T^{\Fr=(-)^q}=\{x \in T \;|\; \psi(\Fr)x\psi(\Fr)^{-1}=x^q\}
    =\left\{\begin{pmatrix}
    	t & 0 \\
    	0 & t^q \\
    \end{pmatrix} \;|\; t^{q^2-1}=1\right\}
    \cong \mu_{q^2-1}=\mu_{120}.$$
    
    $Z^1_{Ad(\psi)}(W_F, N_{\hat{G}}(\psi_{\ell})^0)_1$ turns out to be the connected component of $Z^1_{Ad(\psi)}(W_F/P_F, T) \cong T \times T^{\Fr=(-)^q}$ containing $1$. In our case,
    $$Z^1_{Ad(\psi)}(W_F, N_{\hat{G}}(\psi_{\ell})^0)_1 \cong (T \times \mu_{120})^0 \cong T \times \mu_5.$$
    
    Above all, 
    $$X_{\varphi} \cong (\hat{G} \times T \times \mu_5)/T,$$
    where $T$ is the diagonal torus of $\hat{G}=GL_2$.

\end{eg}

\subsection{The $T$-action on $\hat{G} \times T \times \mu$}\label{Subsection T-action}

We continue with the notations from the last subsection. Recall that $T:=C_{\hat{G}}(\psi|_{I_F^{\ell}})$ and $\mu:=\left(T^{\Fr=(-)^q}\right)^0$. The goal of this subsection is to specify the $T$-action on $(\hat{G} \times T \times \mu)$. 
%Before that, let us record a lemma on several equivalent definitions of $T$.

%\begin{lemma}\label{Lemma: T}
%	$T:=C_{\hat{G}}(\psi|_{I_F^{\ell}}) = C_{\hat{G}}(\psi|_{I_F^{\ell}})^0 = C_{\hat{G}}(\psi|_{I_F}).$
%\end{lemma}

%\begin{proof}
%	We have seen the first equality in Equation \eqref{Equation: T}. To see that $C_{\hat{G}}(\psi|_{I_F^{\ell}})=C_{\hat{G}}(\psi|_{I_F})$, we first note that $C_{\hat{G}}(\psi|_{I_F}) \subseteq C_{\hat{G}}(\psi|_{I_F^{\ell}})=:T$ is included in a commutative group scheme. Since $\psi|_{I_F}$ factors through the abelian group $I_F/P_F$, 
%	$$\psi(I_F) \subseteq C_{\hat{G}}(\psi|_{I_F}) \subseteq T.$$
%	Therefore, $C_{\hat{G}}(\psi|_{I_F}) \supseteq T$ since $T$ is commutative, and hence
%	$$C_{\hat{G}}(\psi|_{I_F}) = T.$$
%\end{proof}

%Now let us make explicit the $T$-action on $(\hat{G} \times T \times \mu)$.

First, recall (see \cite[Subsection 4.6]{dat2022ihes}) that the component $X_{\varphi}=X_{\psi}$ consists of the $L$-parameters $\psi'$ such that $(\psi'_{\ell}, \overline{\psi'})$ is $\hat{G}$-conjugate to $(\psi_{\ell}, \overline{\psi})$. Hence $X_{\varphi}$ is isomorphic to 
$$(\hat{G} \times Z^1(W_F, N_{\hat{G}}(\psi_{\ell}))_{\psi_{\ell}, \overline{\psi}})/C_{\hat{G}}(\psi_{\ell})_{\overline{\psi}}$$
via $g\eta(-)g^{-1} \mapsfrom (g, \eta)$, with $C_{\hat{G}}(\psi_{\ell})_{\overline{\psi}}$ acting on $\hat{G} \times Z^1(W_F, N_{\hat{G}}(\psi_{\ell}))_{\psi_{\ell}, \overline{\psi}}$ by 
$$(t, (g, \psi')) \mapsto (gt^{-1}, t\psi'(-)t^{-1}),$$
where $t \in C_{\hat{G}}(\psi_{\ell})_{\overline{\psi}} \cong T$ and $(g, \psi') \in \hat{G} \times Z^1(W_F, N_{\hat{G}}(\psi_{\ell}))_{\psi_{\ell}, \overline{\psi}}$.

Next, recall that $\eta.\psi \mapsfrom \eta \mapsto (\eta(\Fr), \eta(s_0))$ induces isomorphisms
\begin{equation}\label{Equation: Z^1_Ad}
	Z^1(W_F, N_{\hat{G}}(\psi_{\ell}))_{\psi_{\ell}, \overline{\psi}} \cong Z^1_{Ad(\psi)}(W_F, N_{\hat{G}}(\psi_{\ell})^0)_1 \cong T \times \mu.
\end{equation}
Therefore, 
$$X_{\varphi} \cong (\hat{G} \times Z^1_{Ad(\psi)}(W_F, N_{\hat{G}}(\psi_{\ell})^0)_1)/C_{\hat{G}}(\psi_{\ell})_{\overline{\psi}} \cong (\hat{G} \times T \times \mu)/T.$$

For clarification, let us denote $T$ by $T_1$ when we consider $T$ as $C_{\hat{G}}(\psi_{\ell})_{\overline{\psi}}$, and denote $T$ by $T_2$ when we consider $T$ as a summand of $Z^1_{Ad(\psi)}(W_F, N_{\hat{G}}(\psi_{\ell})^0)_1$ (corresponding to the image of $\Fr$) via \eqref{Equation: Z^1_Ad}. We are going to make explicit the $T_1$-action on $(\hat{G} \times T_2 \times \mu)$.

Let us focus on the isomorphism $\eta.\psi \mapsfrom \eta$:
$$Z^1(W_F, N_{\hat{G}}(\psi_{\ell}))_{\psi_{\ell}, \overline{\psi}} \cong Z^1_{Ad(\psi)}(W_F, N_{\hat{G}}(\psi_{\ell})^0)_1.$$
Recall that $T_1 \subseteq \hat{G}$ acts on $Z^1(W_F, N_{\hat{G}}(\psi_{\ell}))_{\psi_{\ell}, \overline{\psi}}$ by conjugation. Hence, \eqref{Equation: Z^1_Ad} induces an $T_1$-action on $Z^1_{Ad(\psi)}(W_F, N_{\hat{G}}(\psi_{\ell})^0)_1$:
\begin{equation}\label{Equation: twisted conjugation}
	(t, \eta) \mapsto \left(t(\eta(-)\psi(-)) t^{-1}\right)\psi(-)^{-1}.
\end{equation}

Hence in $(\hat{G} \times T_2 \times \mu)/T_1$, we compute by tracking the isomorphisms \eqref{Equation: Z^1_Ad} that 
\begin{enumerate}
	\item $T_1$ acts on $\hat{G}$ via $(t, g) \mapsto gt^{-1}$.
	\item $T_1$ acts on $T_2 \subseteq (T_2 \times \mu)$ (corresponding to $\eta(\Fr)$) by twisted conjugacy (due to the isomorphisms $\eta.\psi \mapsfrom \eta \mapsto (\eta(\Fr), \eta(s_0))$), i.e., 
	$$(t, t') \mapsto \left(t(t'n)t^{-1}\right)n^{-1}=tt'(nt^{-1}n^{-1})=t(nt^{-1}n^{-1})t'=(tnt^{-1}n^{-1})t',$$
	where $n=\psi(\Fr)$; Note that $n$, a prior lies in $\hat{G}$, actually lies in $N_{\hat{G}}(T)$ (since $\Fr.s.\Fr^{-1}=s^q$ implies that $\psi(\Fr)$ normalizes $C_{\hat{G}}(\psi|_{I_F^{\ell}})=T$). To summarize, $t \in T_1$ acts on $T_2$ via multiplication by $tnt^{-1}n^{-1}$.
	\item $T_1$ acts trivially on $\mu \subseteq (T_2 \times \mu)$ (corresponding to $\eta(s_0)$). Indeed, $\eta(s_0) \in T$ and $\psi(s_0) \in T$. Therefore, the twisted (by $\psi$) conjugacy action \eqref{Equation: twisted conjugation} of $T_1$ on $\mu$ is trivial.
\end{enumerate}

We summarize the above into the following commutative diagram, where $\psi'=\eta\psi$, $(t', x)=(\eta(\Fr), \eta(s_0))$, and $^ab:=aba^{-1}$.

% https://q.uiver.app/#q=WzAsMTIsWzEsMCwiWl4xKFdfRiwgTl97XFxoYXR7R319KFxccHNpX3tcXGVsbH0pKV97XFxwc2lfe1xcZWxsfSwgXFxvdmVybGluZXtcXHBzaX19Il0sWzEsMiwiWl4xX3tBZChcXHBzaSl9KFdfRiwgTl97XFxoYXR7R319KFxccHNpX3tcXGVsbH0pXjApXzEiXSxbMSw0LCJUIFxcdGltZXMgXFxtdSJdLFszLDAsInRcXHBzaScoLSl0XnstMX0iXSxbMywyLCJ0XFxldGEoLSkoe157XFxwc2koLSl9dF57LTF9fSkiXSxbMyw0LCIodHQnKF57XFxwc2koXFxGcil9dF57LTF9KSwgeCkiXSxbMiwwLCJcXHBzaScgIl0sWzIsMiwiXFxldGEgIl0sWzIsNCwiKHQnLHgpIl0sWzAsMCwiWl4xKFdfRiwgTl97XFxoYXR7R319KFxccHNpX3tcXGVsbH0pKV97XFxwc2lfe1xcZWxsfSwgXFxvdmVybGluZXtcXHBzaX19Il0sWzAsMiwiWl4xX3tBZChcXHBzaSl9KFdfRiwgTl97XFxoYXR7R319KFxccHNpX3tcXGVsbH0pXjApXzEiXSxbMCw0LCJUIFxcdGltZXMgXFxtdSJdLFswLDFdLFsxLDJdLFszLDQsIiIsMSx7InN0eWxlIjp7InRhaWwiOnsibmFtZSI6Im1hcHMgdG8ifX19XSxbNCw1LCIiLDEseyJzdHlsZSI6eyJ0YWlsIjp7Im5hbWUiOiJtYXBzIHRvIn19fV0sWzYsMywiIiwxLHsic3R5bGUiOnsidGFpbCI6eyJuYW1lIjoibWFwcyB0byJ9fX1dLFs3LDQsIiIsMSx7InN0eWxlIjp7InRhaWwiOnsibmFtZSI6Im1hcHMgdG8ifX19XSxbOCw1LCIiLDEseyJzdHlsZSI6eyJ0YWlsIjp7Im5hbWUiOiJtYXBzIHRvIn19fV0sWzYsNywiIiwxLHsic3R5bGUiOnsidGFpbCI6eyJuYW1lIjoibWFwcyB0byJ9fX1dLFs3LDgsIiIsMSx7InN0eWxlIjp7InRhaWwiOnsibmFtZSI6Im1hcHMgdG8ifX19XSxbOSwxMF0sWzEwLDExXSxbOSwwXSxbMTAsMV0sWzExLDJdXQ==

\begin{tikzcd}
	{Z^1(W_F, N_{\hat{G}}(\psi_{\ell}))_{\psi_{\ell}, \overline{\psi}}} & {Z^1(W_F, N_{\hat{G}}(\psi_{\ell}))_{\psi_{\ell}, \overline{\psi}}} & {\psi' } & {t\psi'(-)t^{-1}} \\
	\\
	{Z^1_{Ad(\psi)}(W_F, N_{\hat{G}}(\psi_{\ell})^0)_1} & {Z^1_{Ad(\psi)}(W_F, N_{\hat{G}}(\psi_{\ell})^0)_1} & {\eta } & {t\eta(-)({^{\psi(-)}t^{-1}})} \\
	\\
	{T \times \mu} & {T \times \mu} & {(t',x)} & {(tt'(^{\psi(\Fr)}t^{-1}), x)}
	\arrow[from=1-2, to=3-2]
	\arrow[from=3-2, to=5-2]
	\arrow[maps to, from=1-4, to=3-4]
	\arrow[maps to, from=3-4, to=5-4]
	\arrow[maps to, from=1-3, to=1-4]
	\arrow[maps to, from=3-3, to=3-4]
	\arrow[maps to, from=5-3, to=5-4]
	\arrow[maps to, from=1-3, to=3-3]
	\arrow[maps to, from=3-3, to=5-3]
	\arrow[from=1-1, to=3-1]
	\arrow[from=3-1, to=5-1]
	\arrow[from=1-1, to=1-2]
	\arrow[from=3-1, to=3-2]
	\arrow[from=5-1, to=5-2]
\end{tikzcd}

On the other hand, recall that we have the natural $\hat{G}$-action on $Z^1(W_F, \hat{G})$ by conjugation, hence the $\hat{G}$-action on this component $X_{\varphi}$. Under the isomorphism $X_{\varphi} \cong (\hat{G} \times T_2 \times \mu)/T_1$, the $\hat{G}$-action becomes
$$(g', (g, t, m)) \mapsto  (g'g, t, m), \text{ for any } g' \in \hat{G} \text{ and } (g, t, m) \in (\hat{G} \times T_2 \times \mu)/T_1.$$

Note that the $T_1$-action and the $\hat{G}$-action on $(\hat{G} \times T_2 \times \mu)$ commute with each other; we thus have the following isomorphisms of quotient stacks:\footnote{Since we have specified the action of $T_1=T$ on $T_2=T$, we go back to the notation $T$ in the statement of Proposition \ref{Proposition: T times mu/T}.}
$$X_{\varphi}/\hat{G} \cong \left((\hat{G} \times T \times \mu)/T\right)/\hat{G} \cong \left((\hat{G} \times T \times \mu)/\hat{G}\right)/T \cong (T \times \mu)/T,$$ 
with $t \in T$ acting on $T$ via multiplication by $tnt^{-1}n^{-1}$, where $n=\psi(\Fr)$, and $t \in T$ acting trivially on $\mu$. Therefore, we have the following Proposition.

\begin{proposition}\label{Proposition: T times mu/T} We have the following isomorphisms of (quotient) stacks
	\begin{equation}\label{Equation: T/T}
		X_{\varphi}/\hat{G} \cong (T \times \mu)/T \cong [T/T] \times \mu,
	\end{equation}
	with $t \in T$ acting on $T$ via multiplication by $tnt^{-1}n^{-1}$, where $n=\psi(\Fr)$.
\end{proposition}

\subsection{Some lemmas}

\begin{lemma}\label{Lem generalizing}
	Let $\varphi' \in Z^1(W_t, \hat{G}(\overline{\mathbb{F}}_{\ell}))$. Then there exists $\psi' \in Z^1(W_t, \hat{G}(\overline{\mathbb{Z}}_{\ell}))$ such that $\psi'$ is a lift of $\varphi'$.
\end{lemma}

\begin{proof}
	In the statement, $Z^1(W_t, \hat{G})$ is the abbreviation for $Z^1(W_t, \hat{G})_{\overline{\mathbb{Z}}_{\ell}}$. Recall that $Z^1(W_t, \hat{G}) \to \overline{\mathbb{Z}}_{\ell}$ is flat (see \cite[Proposition 3.3]{dat2022ihes}), hence generalizing (see \cite[Stack, Tag 01U2]{stacks-project}). Therefore, given $\varphi' \in Z^1(W_t, \hat{G}(\overline{\mathbb{F}}_{\ell}))$, there exists $\xi \in Z^1(W_t, \hat{G}(\overline{\mathbb{Q}}_{\ell}))$ such that $\xi$ specializes to $\varphi'$. In other words, $\ker(\xi) \subseteq \ker(\varphi')$. We will show that $\xi: W_t \to \hat{G}(\overline{\mathbb{Q}}_{\ell})$ factors through  $\hat{G}(\overline{\mathbb{Z}}_{\ell})$.
	
	This is true by the following more general statement: Let $Y=\Spec(R)$ be an affine scheme over $\overline{\mathbb{Z}}_{\ell}$, let $y_{\eta} \in Y(\overline{\mathbb{Q}}_{\ell})$ specializing to $y_s \in Y(\overline{\mathbb{F}}_{\ell})$.  Then, $y_{\eta} \in Y(\overline{\mathbb{Q}}_{\ell})=\Hom(R, \overline{\mathbb{Q}}_{\ell})$ factors through $\overline{\mathbb{Z}}_{\ell}$.
	
    To prove the above statement, let $\mathfrak{p}:=\ker(y_\eta)$ and $\mathfrak{q}:=\ker(y_s)$ be the corresponding prime ideals. Then ``$y_{\eta}$ specializes to $y_s$" translates to ``$\mathfrak{p} \subseteq \mathfrak{q}$". Recall that we are going to show that $y_{\eta}: R \to \overline{\mathbb{Q}}_{\ell}$ factors through $\overline{\mathbb{Z}}_{\ell}$. We argue by contradiction. Otherwise there is some element $f \in R$ mapping to $\ell^{-m}u$ for some $m \in \mathbb{Z}_{\geq 1}$ and $u \in \overline{\mathbb{Z}}_{\ell}^*$. Hence 
    \begin{equation}\label{eq ell}
    	\ell^mu^{-1}f-1 \in \ker(y_{\eta}) \subseteq \ker(y_s).
    \end{equation}
    However, $\ell \in \ker(y_s)$ since $y_s \in Y(\overline{\mathbb{F}}_{\ell})$. This together with equation \eqref{eq ell} implies that $1 \in \ker(y_s)$. Contradiction!
\end{proof}

\begin{lemma}\label{Lem gen red}
	The schematic centralizer $C_{\hat{G}}(\psi_{\ell})$ is a generalized reductive group scheme over $\overline{\mathbb{Z}}_{\ell}$.
\end{lemma}

\begin{proof}
	We are going to use \cite[Lemma 3.2]{dat2022ihes}. We first note that $$C_{\hat{G}}(\psi_{\ell})=C_{\hat{G}}(\psi(I_F^{\ell})),$$
	where $C_{\hat{G}}(\psi(I_F^{\ell}))$ is the schematic centralizer of the subgroup $\psi(I_F^{\ell}) \subseteq \hat{G}(\overline{\mathbb{Z}}_{\ell})$ in $\hat{G}$. Indeed, this can be checked by the Yoneda Lemma on $R$-valued points for any $\overline{\mathbb{Z}}_{\ell}$-algebra $R$.
	
	Then, we can conclude by \cite[Lemma 3.2]{dat2022ihes}. Indeed, $\psi_{\ell}$ factors through some finite quotient $Q$ of $I_F^{\ell}$, which has order invertible in the base $\overline{\mathbb{Z}}_{\ell}$. Therefore, the assumptions of \cite[Lemma 3.2]{dat2022ihes} are satisfied (for details, see Remark \ref{Remark condition} below). 
\end{proof}

\begin{remark}\label{Remark condition}\
	\begin{enumerate}
		\item While \cite[Lemma 3.2]{dat2022ihes} is phrased in the setting that $R$ is a normal subring of a number field, it still works for $\overline{\mathbb{Z}}_{\ell} \subseteq \overline{\mathbb{Q}}_{\ell}$ instead of $\mathbb{Z} \subseteq \mathbb{Q}$. Indeed, $\psi_{\ell}$ factors through some finite quotient $Q$ of $I_F^{\ell}$, say of order $|Q|=N$ (note that $N$ is coprime to $\ell$ since $Q$ is a quotient of $I_F^{\ell}$). Then we can use \cite[Lemma 3.2]{dat2022ihes} to conclude that $C_{\hat{G}}(\psi_{\ell})$ is generalized reductive over $\mathbb{Z}[1/pN]$. Hence $C_{\hat{G}}(\psi_{\ell})$ is also generalized reductive over $\overline{\mathbb{Z}}_{\ell}$ by base change.
		\item There is also a small issue that $\overline{\mathbb{Z}}_{\ell}$ is not finite over $\mathbb{Z}_{\ell}$, but this can be resolved since everything is already defined over some sufficiently large finite extension $\mathcal{O}$ of $\mathbb{Z}_{\ell}$.
	\end{enumerate}
\end{remark}

\begin{lemma}\label{Lem I_F^ell}
	$$C_{\hat{G}}(\psi_{\ell})(\overline{\mathbb{F}}_{\ell})=C_{\hat{G}(\overline{\mathbb{F}}_{\ell})}(\varphi(I_F^\ell))=C_{\hat{G}(\overline{\mathbb{F}}_{\ell})}(\varphi(I_F)).$$
\end{lemma}

\begin{proof}
	The first equation is by definition of the schematic centralizer and that $C_{\hat{G}}(\psi_{\ell})$ represents the set-theoretic centralizer. See Definition \ref{Definition: Schematic centralizer}.
	
	For the second equation, choose a faithful embedding $\hat{G} \hookrightarrow GL(V)$. Then, $\varphi|_{I_t}=\gamma_1 + ...+ \gamma_d$ is a direct sum of characters (since $I_t \cong \prod_{p'\neq p}\mathbb{Z}_{p'}$), so it suffices to show that each $\gamma_i$ is trivial on the summand $\mathbb{Z}_{\ell}$ of $I_t\cong \prod_{p'\neq p}\mathbb{Z}_{p'}$.
	Indeed,
	$$\Hom_{\Cont}(\mathbb{Z}_{\ell}, \overline{\mathbb{F}}_{\ell}^*)=\Hom_{\Cont}(\varprojlim\mathbb{Z}/\ell^n\mathbb{Z}, \overline{\mathbb{F}}_{\ell}^*)=\varinjlim\Hom(\mathbb{Z}/\ell^n\mathbb{Z}, \overline{\mathbb{F}}_{\ell}^*)=\{1\}.$$
\end{proof}

\begin{lemma}\label{Lem_Z^1()_1}
	$Z^1_{Ad(\psi)}(W_F, N_{\hat{G}}(\psi_{\ell})^0)_1 \cong (T \times T^{\Fr=(-)^q})^0.$
\end{lemma}

\begin{proof}
	We have omitted from the notations but here everything is over $\overline{\mathbb{Z}}_{\ell}$.
	Recall that $N_{\hat{G}}(\psi_{\ell})^0=C_{\hat{G}}(\psi_{\ell})^0=T$ and that
	$$Z^1_{Ad(\psi)}(W_F/P_F, N_{\hat{G}}(\psi_{\ell})^0) \cong T \times T^{\Fr=(-)^q}.$$
	
	By \cite[Section 5.4, 5.5]{dat2022ihes}, $Z^1_{Ad(\psi)}(W_F, N_{\hat{G}}(\psi_{\ell})^0)_1$ is connected (over $\overline{\mathbb{Z}}_{\ell}$). We need to check that the assumptions of \cite[Section 5.4, 5.5]{dat2022ihes} are satisfied. Indeed, since $N_{\hat{G}}(\psi_{\ell})^0=T$ is a connected torus, the $W_t^0$-action on $T$ automatically fixes a Borel pair of $T$. Moreover, $s_0$ acts trivially on $N_{\hat{G}}(\psi_{\ell})^0=T$ via $\psi$, so in particular the action of $s_0$ (which is denoted by $s$ in \cite[Section 5.5]{dat2022ihes}) has order a power of $\ell$ (which is $1 = \ell^0$).
	
	Therefore, 
	$$Z^1_{Ad(\psi)}(W_F, N_{\hat{G}}(\psi_{\ell})^0)_1 \subseteq Z^1_{Ad(\psi)}(W_F/P_F, N_{\hat{G}}(\psi_{\ell})^0)^0 \cong (T \times T^{\Fr=(-)^q})^0.$$
	
	On the other hand, by \cite[Section 4.6]{dat2022ihes}, 
	$$Z^1_{Ad(\psi)}(W_F, N_{\hat{G}}(\psi_{\ell})^0)_1 \hookrightarrow Z^1_{Ad(\psi)}(W_F/P_F, N_{\hat{G}}(\psi_{\ell})^0)$$
	is open and closed. Indeed, this can be seen by considering the restriction to the prime-to-$\ell$ inertia $I_F^{\ell}$, and then using \cite[Theorem 4.2]{dat2022ihes}.
	
	Therefore, 
	$$Z^1_{Ad(\psi)}(W_F, N_{\hat{G}}(\psi_{\ell})^0)_1 = Z^1_{Ad(\psi)}(W_F/P_F, N_{\hat{G}}(\psi_{\ell})^0)^0 \cong (T \times T^{\Fr=(-)^q})^0.$$
	
\end{proof}

\section{Main Theorem: description of $X_{\varphi}/\hat{G}$}\label{Section X/hatG}

We have already seen a description of $X_{\varphi}/\hat{G}$ in Proposition \ref{Proposition: T times mu/T}. The goal of this section is to describe $X_{\varphi}/\hat{G}$ explicitly under the assumption that the center $Z(\hat{G})$ of $\hat{G}$ is finite (see Theorem \ref{Thm X/G} for the precise statement).

Let $F$ be a non-archimedean local field, $G$ be a connected split reductive group over $F$. Let $\varphi \in Z^1(W_F, \hat{G}(\overline{\mathbb{F}}_{\ell}))$ be a TRSELP. Recall that this means that the centralizer 
$$C_{\hat{G}(\overline{\mathbb{F}}_{\ell})}(\varphi(I_F)) =: S \subseteq \hat{G}(\overline{\mathbb{F}}_{\ell})$$ 
is a maximal torus, and $\varphi(\Fr) \in N_{\hat{G}}(S)$ (and $\varphi$ is tame and elliptic). 

%Assume that $Z(\hat{G})$ is finite.
%\begin{enumerate}
%	\item \label{assumption 1} The center $Z(\hat{G})$ is smooth over $\overline{\mathbb{Z}}_{\ell}$.
%	\item \label{assumption 2} $Z(\hat{G})$ is finite.
%%	\item \label{assumption 3} $\ell$ doesn't divide the order of $w=\overline{\varphi(\Fr)}$ in the Weyl group $N_{\hat{G}}(S)/S$.
%\end{enumerate}

Let $\psi \in Z^1(W_F, \hat{G}(\overline{\mathbb{Z}}_{\ell}))$ be any lifting of $\varphi$. Let $\psi_{\ell}$ denote the restriction $\psi|_{I_F^{\ell}}$, and $\overline{\psi}$ denote the image of $\psi$ in $Z^1(W_F, \pi_0(N_{\hat{G}}(\psi_{\ell})))$. Recall that the schematic centralizer $C_{\hat{G}}(\psi_{\ell})=T$ is a split torus over $\overline{\mathbb{Z}}_{\ell}$ with $\overline{\mathbb{F}}_{\ell}$-points $C_{\hat{G}(\overline{\mathbb{F}}_{\ell})}(\varphi(I_F)) = S$.

Our main theorem is the following.

\begin{theorem}\label{Thm X/G}
	Assume that $Z(\hat{G})$ is finite.
	Let $X_{\varphi}$($=X_{\psi}$) be the connected component of $Z^1(W_F, \hat{G})_{\overline{\mathbb{Z}}_{\ell}}$ containing $\varphi$ (hence also containing $\psi$). Then we have isomorphisms of quotient stacks
	\begin{equation}\label{Equation: X_phi}
		X_{\varphi}/\hat{G} \cong (T \times \mu)/T \cong [T/T] \times \mu \cong [*/{C_T(n)}] \times \mu \cong [*/S_{\psi}] \times \mu,
	\end{equation}
	where $C_T(n)$ is the schematic centralizer of $n=\psi(\Fr)$ in $T=C_{\hat{G}}(\psi|_{I_F^{\ell}})$, $\mu=(T^{\Fr=(-)^q})^0 \cong \prod_{i=1}^m\mu_{\ell^{k_i}}$ for some $k_i \in \mathbb{Z}_{\geq 1}$, $m \in \mathbb{Z}_{\geq 0}$ is a product of group schemes of roots of unity, and $S_{\psi}:=C_{\hat{G}}(\psi)$ is the schematic centralizer of $\psi$ in $\hat{G}$. 
	
%	If we moreover assume that
%    $\ell$ does not divide the order of $w=\overline{\varphi(\Fr)}$ in the Weyl group $N_{\hat{G}}(S)/S$,
%	then 
%	$$[X_{\varphi}/\hat{G}] \cong [(T \times \mu)/T] \cong [*/\underline{S_{\varphi}(\overline{\mathbb{F}}_{\ell})}] \times \mu,$$
%	where $S_{\varphi}(\overline{\mathbb{F}}_{\ell})=C_{\hat{G}(\overline{\mathbb{F}}_{\ell})}(\varphi(W_F))$, and $\underline{S_{\varphi}(\overline{\mathbb{F}}_{\ell})}$ is the corresponding constant group scheme.\footnote{By abuse of notation, we sometimes denote $S_{\varphi}(\overline{\mathbb{F}}_{\ell})$ simply by $S_{\varphi}$.}
\end{theorem}

\begin{proof}
	By Proposition \ref{Proposition: T times mu/T}, we are reduced to computing $T/T$, with $t \in T$ acting on $T$ via multiplication by $tnt^{-1}n^{-1}$, where $n=\psi(\Fr)$.
	
    Consider the morphism
	$$f: T^{(1)} := T \longrightarrow T =: T^{(2)} \qquad s \longmapsto sns^{-1}n^{-1}.\footnote{For clarification, let us denote the source torus $T$ as $T^{(1)}$ and the target torus $T$ as $T^{(2)}$.}$$
	This is surjective on $\overline{\mathbb{F}}_{\ell}$-points by our assumption that $Z(\hat{G})$ is finite and $\varphi$ is elliptic (see Lemma \ref{Lem epic} below). Hence $f$ is an epimorphism in the category of diagonalizable $\overline{\mathbb{Z}}_{\ell}$-group schemes (see Lemma \ref{Lem epic} below). Therefore, $f$ induces an isomorphism 
	\begin{equation}\label{eq_T}
		T^{(1)}/\ker(f) \cong T^{(2)}
	\end{equation}
	as diagonalizable $\overline{\mathbb{Z}}_{\ell}$-group schemes. Moreover, if we let $t \in T$ act on $T^{(1)}$ by left multiplication by $t$, and on $T^{(2)}$ via multiplication by $tnt^{-1}n^{-1}$, this isomorphism induced by $f$ is $T$-equivariant.
	
	Note that $T^{(1)}=T$ is commutative, so the $T$-action (via multiplication by $tnt^{-1}n^{-1}$) and the $\ker(f)$-action (via left multiplication) on $T$ commute with each other. Hence by the $T$-equivariant isomorphism \eqref{eq_T}, we have
	$$T/T = T^{(2)}/T \cong \left(T^{(1)}/\ker(f)\right)/T \cong \left(T^{(1)}/T\right)/\ker(f) \cong [*/\ker(f)] = [*/C_T(n)].$$ 
	
%	Finally, notice that 
%	$$\ker(f)=C_T(n)=C_{\hat{G}}(\psi),$$
%	since $T=C_{\hat{G}}(\psi|_{I_F^{\ell}})$ (\textcolor{red}{is this same as $C_{\hat{G}}(\psi|_{I_F})$ ?}) and $n=\psi(\Fr)$ (\textcolor{red}{see Lemma ? below}).
	
%	Now let's proof the last assertion: To show $\ker(f)=\underline{S_{\varphi}}$ under assumption \ref{assumption 3} (Need adjust) -- $\ell$ does not divide the order of $w$ in the Weyl group $N_{\hat{G}}(T)/T$ (\textcolor{red}{Use T or S?}). Then $\ker(f) \cong \underline{S_{\varphi}}$ is the constant group scheme of the finite abelian group $S_{\varphi}=C_{\hat{G}(\overline{\mathbb{F}}_{\ell})}(\varphi(W_F))$, \textcolor{red}{see Lemma below}. We win!

    Moreover, recall that we have $T:=C_{\hat{G}}(\psi|_{I_F^{\ell}})=C_{\hat{G}}(\psi|_{I_F})$ (see Lemma \ref{Lemma: T}). So 
    $$C_T(n) \cong C_{\hat{G}}(\psi(I_F), \psi(\Fr)) \cong C_{\hat{G}}(\psi) =: S_{\psi}.$$
\end{proof}

\begin{remark}
	One might wonder when is $S_{\psi}$ a constant group scheme over $\overline{\mathbb{Z}}_{\ell}$. Indeed, one can prove that this is the case when $\ell$ does not divide the order of $\psi(\Fr)$ in the Weyl group $N_{\hat{G}}(T)/T$.
\end{remark}

\begin{lemma}\label{Lem epic}
	The morphism 
	$$f: T^{(1)} = T \longrightarrow T = T^{(2)} \qquad s \longmapsto sns^{-1}n^{-1}$$
	is epimorphic in the category of diagonalizable $\overline{\mathbb{Z}}_{\ell}$-group schemes. Moreover, $f$ induces an isomorphism $T^{(1)}/\ker(f) \cong T^{(2)}$ as diagonalizable $\overline{\mathbb{Z}}_{\ell}$-group schemes.
\end{lemma}

\begin{proof}
	Recall that $T$ is a split torus over $\overline{\mathbb{Z}}_{\ell}$, hence a diagonalizable $\overline{\mathbb{Z}}_{\ell}$-group scheme. Note that $f$ is a morphism of $\overline{\mathbb{Z}}_{\ell}$-group schemes and hence a morphism of diagonalizable $\overline{\mathbb{Z}}_{\ell}$-group schemes. Recall that the category of diagonalizable $\overline{\mathbb{Z}}_{\ell}$-group schemes is equivalent to the category of abelian groups (see \cite[p70, Section 5]{brochard2014autour} or \cite{conrad2014reductive}) via
	$$D \mapsto \Hom_{\overline{\mathbb{Z}}_{\ell}-GrpSch}(D, \mathbb{G}_m),$$
	and the inverse is given by 
	$$\overline{\mathbb{Z}}_{\ell}[M] \mapsfrom M,$$
	where $\overline{\mathbb{Z}}_{\ell}[M]$ is the group algebra of $M$ with $\overline{\mathbb{Z}}_{\ell}$-coefficients.
	
	Therefore, we can argue in the category of abelian groups via the above equivalence of categories: $f$ is epimorphic if and only if the map $f^*$ in the category of abelian groups is monomorphic. Since $\varphi$ is elliptic and $Z(\hat{G})$ is finite, $S_{\varphi}$ is finite; hence, 
	$$\ker(f)(\overline{\mathbb{F}}_{\ell})=C_T(n)(\overline{\mathbb{F}}_{\ell})=S_{\varphi}(\overline{\mathbb{F}}_{\ell})$$
	is finite (where the first equality is by definition of $f$, and the second equality holds because $T(\overline{\mathbb{F}}_{\ell})=C_{\hat{G}(\overline{\mathbb{F}}_{\ell})}(\varphi(I_F))$ and $n = \psi(\Fr)$ maps to $\varphi(\Fr) \in \hat{G}(\overline{\mathbb{F}}_{\ell})$ under the natural map $\hat{G}(\overline{\mathbb{Z}}_{\ell}) \to \hat{G}(\overline{\mathbb{F}}_{\ell})$). Accordingly, $\coker(f^*)$ is finite. Therefore, 
	$$f^*:\Hom(T^{(2)}, \mathbb{G}_m) \to \Hom(T^{(1)}, \mathbb{G}_m)$$
	is injective (i.e., monomorphism). Indeed, otherwise $\ker(f^*)$ would be a nonzero sub-$\mathbb{Z}$-module of the finite free $\mathbb{Z}$-module $\Hom(T^{(2)}, \mathbb{G}_m)$, hence a free $\mathbb{Z}$-module of positive rank, which contradicts $\coker(f^*)$ being finite.
	
	The statement on the quotient follows from the corresponding result in the category of abelian groups: $f^*$ induces an isomorphism
	$$\Hom(T^{(1)}, \mathbb{G}_m)/\Hom(T^{(2)}, \mathbb{G}_m) \cong \coker(f^*)$$
	(see \cite[p71, Subsection 5.3]{brochard2014autour}).
\end{proof}

%% file: Supercuspidal.tex
		\chapter{Depth-zero regular supercuspidal blocks} \label{Chapter Rep}
		
		The goal of this chapter is to describe the block $\Rep_{\overline{\mathbb{Z}}_{\ell}}(G(F))_{[\pi]}$ (denoted $\mathcal{C}_{x,1}$ later) of $\Rep_{\overline{\mathbb{Z}}_{\ell}}(G(F))$ containing a  depth-zero regular supercuspidal representation $\pi$.
		
		Assume that $G$ is semisimple and simply connected. Recall that a depth-zero regular supercuspidal representation $\pi$ is of the form
		$$\pi=\cInd_{G_x}^{G(F)}\rho,$$
		where $\rho$ is a representation of $G_x$ whose reduction $\overline{\rho}$ to the finite reductive group
		$\overline{G_x}=G_x/G_x^+$ is supercuspidal.
		
		In the end, the block $\Rep_{\overline{\mathbb{Z}}_{\ell}}(G(F))_{[\pi]}$ would be equivalent to the block $\Rep_{\overline{\mathbb{Z}}_{\ell}}(\overline{G_x})_{[\overline{\rho}]}$ (denoted $\mathcal{A}_{x,1}$ later) of $\Rep_{\overline{\mathbb{Z}}_{\ell}}(\overline{G_x})$ containing $\overline{\rho}$. In addition, $\mathcal{A}_{x,1}$ has an explicit description via the Broué equivalence \ref{Thm Broué}.
		
		Indeed, let $\Rep_{\overline{\mathbb{Z}}_{\ell}}(G_x)_{[\rho]}$ (denoted $\mathcal{B}_{x,1}$ later) be the block of $\Rep_{\overline{\mathbb{Z}}_{\ell}}(G_x)$ containing $\rho$. It is not hard to see that inflation along $G_x \to \overline{G_x}$ induces an equivalence of categories 
		$\mathcal{A}_{x,1} \cong \mathcal{B}_{x,1}$. The main theorem we prove in this chapter is that compact induction induces an equivalence of categories
		$$\cInd_{G_x}^{G(F)}: \mathcal{B}_{x,1} \cong \mathcal{C}_{x,1}.$$
		The proof of this main theorem \ref{Thm Main} would occupy most of this chapter, from Section \ref{Section cInd} to \ref{Section projective generator}. The proof relies on three theorems. In Section \ref{Section cInd}, we prove the main theorem modulo the three theorems. The proofs of the three theorems are given in Sections \ref{Sec Reg Cusp}, \ref{Sec Pf Thm Hom}, \ref{Section projective generator}, respectively.

		\section{The compact induction induces an equivalence}\label{Section cInd}
		In this section, we prove the Main Theorem \ref{Thm Main} modulo Theorem \ref{Thm SC Red} \ref{Thm Hom} \ref{Thm Proj}.
		
		Let $G$ be a connected split reductive group scheme over $\mathbb{Z}$, which is semisimple and simply connected. Let $F$ be a non-archimedean local field, with the ring of integers $\mathcal{O}_F$ and residue field $k_F \cong \mathbb{F}_q$ of characteristic $p$. For simplicity, we assume that $q$ is greater than the Coxeter number of $\overline{G_x}$ for any vertex $x$ of the Bruhat-Tits building of $G$ over $F$ (see Theorem \ref{Thm Broué} for reason).
		
		Let $x$ be a vertex of the Bruhat-Tits building $\mathcal{B}(G, F)$. Let $G_x$ be the parahoric subgroup associated to $x$, and $G_x^+$ be its pro-unipotent radical. Recall that $\overline{G_x}:=G_x/G_x^+$ is a generalized Levi subgroup of $G(k_F)$ with root system $\Phi_x$, see \cite[Theorem 3.17]{rabinoff2003bruhat}. 
		
		Let $\Lambda=\overline{\mathbb{Z}}_{\ell}$, with $\ell \neq p$. Let $\rho \in \Rep_{\Lambda}(G_x)$ be an irreducible representation of $G_x$, which is trivial on $G_x^+$ and whose reduction to the finite group of Lie type $\overline{G_x}=G_x/G_x^+$ is  
		regular supercuspidal. Here \textbf{regular supercuspidal} (see Definition \ref{Def regular supercuspidal}) means that $\rho$ is supercuspidal and lies in a regular block of $\Rep_{\Lambda}(\overline{G_x})$, in the sense of \cite{broue1990isometries}. The reason we want the regularity assumption is that we want to work with a block of $\Rep_{\Lambda}(\overline{G_x})$ which consists purely of supercuspidal representations. See Section \ref{Sec Reg Cusp} for details. We make this a definition for later use.
		
		\begin{definition}
			Let $\rho \in \Rep_{\Lambda}(G_x)$. We say $\rho$ \textbf{has supercuspidal reduction} (resp. \textbf{has regular supercuspidal reduction}), if $\rho$ is trivial on $G_x^+$ and whose reduction to the finite group of Lie type $\overline{G_x}=G_x/G_x^+$ is supercuspidal (resp. regular supercuspidal). Let us denote the reduction of $\rho$ modulo $G_x^+$ by $\overline{\rho} \in \Rep_{\Lambda}(\overline{G_x})$.
		\end{definition}
		
		Let $\mathcal{B}_{x,1}$ be the block of $\Rep_{\Lambda}(G_x)$ containing $\rho$. Let $\mathcal{C}_{x,1}$ be the block of $\Rep_{\Lambda}(G(F))$ containing $\pi:=\cInd_{G_x}^{G(F)}\rho$. Now we can state the main theorem of this chapter.
		
		\begin{theorem}[Main Theorem]\label{Thm Main}
			Let $x$ be a vertex of the Bruhat-Tits building $\mathcal{B}(G, F)$. Let $\rho \in \Rep_{\Lambda}(G_x)$ be a representation which has regular supercuspidal reduction. Let $\mathcal{B}_{x,1}$ be the block of $\Rep_{\Lambda}(G_x)$ containing $\rho$. Let $\mathcal{C}_{x,1}$ be the block of $\Rep_{\Lambda}(G(F))$ containing $\pi:=\cInd_{G_x}^{G(F)}\rho$. Then the compact induction $\cInd_{G_x}^{G(F)}$ induces an equivalence of categories $\mathcal{B}_{x,1} \cong \mathcal{C}_{x,1}$. 
		\end{theorem}
		
		As mentioned before, the reason we want the regular supercuspidal assumption is the following theorem, which is proven in Section \ref{Sec Reg Cusp}. 
		
		\begin{theorem}\label{Thm SC Red}
			Let $\rho \in \Rep_{\Lambda}(G_x)$ be an irreducible representation of $G_x$, which has regular supercuspidal reduction. Let $\mathcal{B}_{x,1}$ be the block of $\Rep_{\Lambda}(G_x)$ containing $\rho$. Then any $\rho' \in \mathcal{B}_{x,1}$ has supercuspidal reduction.
		\end{theorem}
		
		The proof of the Main Theorem \ref{Thm Main} splits into two parts -- fully faithfulness and essentially surjectivity. It is convenient to have the following theorem available at an early stage, which implies fully faithfulness immediately and is also used in the proof of essentially surjectivity.
		
		\begin{theorem}\label{Thm Hom}
			Let $x, y$ be two vertices of the Bruhat-Tits building $\mathcal{B}(G, F)$. Let $\rho_1$ be a representation of the parahoric $G_x$ which is trivial on the pro-unipotent radical $G_x^+$. Let $\rho_2$ be a representation of $G_y$ which is trivial on $G_y^+$. Assume that one of them has supercuspidal reduction. Then exactly one of the following happens:
			\begin{enumerate}
				\item If there exists an element $g \in G(F)$ such that $g.x=y$, then
				$$\Hom_G(\cInd_{G_x}^{G(F)}\rho_1, \cInd_{G_y}^{G(F)}\rho_2)=\Hom_{G_x}(\rho_1, \rho_2(g-g^{-1})).$$
				\item If there is no elements $g \in G(F)$ such that $g.x=y$, then
				$$\Hom_G(\cInd_{G_x}^{G(F)}\rho_1, \cInd_{G_y}^{G(F)}\rho_2)=0.$$
			\end{enumerate}
		\end{theorem}
		
		The proof of the above theorem is a computation using Mackey's formula. See Section \ref{Sec Pf Thm Hom}.
		
		\begin{proof}[Proof of Theorem \ref{Thm Main}]
			
			Now we proceed by steps towards our goal: The compact induction $\cInd_{G_x}^{G(F)}$ induces an equivalence of categories $\mathcal{B}_{x,1} \cong \mathcal{C}_{x,1}$. 
			
			First, we show that $\cInd_{G_x}^{G(F)}: \mathcal{B}_{x,1} \to \mathcal{C}_{x,1}$ is well-defined. We need to show that the image of $\mathcal{B}_{x,1}$ under $\cInd_{G_x}^{G(F)}$ lies in $\mathcal{C}_{x,1}$. By Theorem \ref{Thm SC Red} and Theorem \ref{Thm Hom} above, $$\cInd_{G_x}^{G(F)}|_{\mathcal{B}_{x,1}}: \mathcal{B}_{x,1} \to \Rep_{\Lambda}(G(F))$$
			is fully faithful (see Lemma \ref{Lem Thm Hom implies fully faithful}, note that here we used Theorem \ref{Thm SC Red} that any representation in $\mathcal{B}_{x,1}$ has supercuspidal reduction so that we can apply Theorem \ref{Thm Hom}), hence an equivalence onto the essential image. Since $\mathcal{B}_{x,1}$ is indecomposable as an abelian category, so is its essential image (see Lemma \ref{Lem Indec}). Hence, its essential image is contained in a single block of $\Rep_{\Lambda}(G(F))$. But such a block must be $\mathcal{C}_{x,1}$ since $\cInd_{G_x}^{G(F)}$ maps $\rho$ to $\pi \in \mathcal{C}_{x,1}$. Therefore, $\cInd_{G_x}^{G(F)}: \mathcal{B}_{x,1} \to \mathcal{C}_{x,1}$ is well-defined.
			
			Second, we show that $\cInd_{G_x}^{G(F)}: \mathcal{B}_{x,1} \to \mathcal{C}_{x,1}$ is fully faithful. This is already noticed in the proof of ``well-defined" in the last paragraph. Indeed, 
			$$\Hom_G(\cInd_{G_x}^{G(F)}\rho_1, \cInd_{G_x}^{G(F)}\rho_2)=\Hom_{G_x}(\rho_1, \rho_2)$$
			by Theorem \ref{Thm SC Red} and Theorem \ref{Thm Hom} (see Lemma \ref{Lem Thm Hom implies fully faithful}.). Therefore, $\cInd_{G_x}^{G(F)}: \mathcal{B}_{x,1} \to \mathcal{C}_{x,1}$ is fully faithful.
			
			Finally, we show that $\cInd_{G_x}^{G(F)}: \mathcal{B}_{x,1} \to \mathcal{C}_{x,1}$ is essentially surjective. This will occupy the rest of this section. 
			
			The idea is to find a projective generator of $\mathcal{C}_{x,1}$ and show that it is in the essential image. Fix a vertex $x$ of the Bruhat-Tits building $\mathcal{B}(G, F)$ as before. Let $V$ be the set of equivalence classes of vertices of the Bruhat-Tits building $\mathcal{B}(G, F)$ up to $G(F)$-action. For $y \in V$, let $\sigma_y:=\cInd_{G_y^+}^{G_y}\Lambda$. Let $\Pi:=\bigoplus_{y \in V}\Pi_y$ where $\Pi_y:=\cInd_{G_y^+}^{G(F)}\Lambda$. Then $\Pi$ is a projective generator of the category of depth-zero representations $\Rep_{\Lambda}(G(F))_0$, see \cite[Appendix]{dat2009finitude}. Let $\sigma_{x,1}:=(\sigma_x)|_{\mathcal{B}_{x,1}} \in \mathcal{B}_{x,1} \xhookrightarrow{summand} \Rep_{\Lambda}(G_x)$ be the $\mathcal{B}_{x,1}$-summand of $\sigma_x$. Let $\Pi_{x,1}:=\cInd_{G_x}^{G(F)}\sigma_{x,1}$. Note that $\Pi_{x,1}$ is a summand of $\Pi_x=\cInd_{G_x}^{G(F)}\sigma_x$, hence a summand of $\Pi$. Using Theorem \ref{Thm Hom}, one can show that the rest of the summands of $\Pi$ do not interfere with $\Pi_{x,1}$ (see Lemma \ref{Lem Ortho} and Lemma \ref{Lem Gen} for precise meaning), hence $\Pi_{x,1}$ is a projective generator of $\mathcal{C}_{x,1}$. Let us state it as a Theorem, see Section \ref{Section projective generator} for details.
			
			\begin{theorem}\label{Thm Proj}
				$\Pi_{x,1}=\cInd_{G_x}^{G(F)}\sigma_{x,1}$ is a projective generator of $\mathcal{C}_{x,1}$.
			\end{theorem}
			
			Now we have found a projective generator $\Pi_{x,1}=\cInd_{G_x}^{G(F)}\sigma_{x,1}$ of $\mathcal{C}_{x,1}$, and it is clear that $\Pi_{x,1}$ is in the essential image of $\cInd_{G_x}^{G(F)}$. We now deduce from this that $\cInd_{G_x}^{G(F)}: \mathcal{B}_{x,1} \to \mathcal{C}_{x,1}$ is essentially surjective. Indeed, for any $\pi' \in \mathcal{C}_{x,1}$, we can resolve $\pi'$ by some copies of $\Pi_{x,1}$:
			$$\Pi_{x,1}^{\oplus I} \xrightarrow{f} \Pi_{x,1}^{\oplus J} \to \pi' \to 0.$$
			Using Theorem \ref{Thm Hom} and $\cInd_{G_x}^{G(F)}$ commutes with arbitrary direct sums (see Lemma \ref{Lem Sum}) we see that $f \in \Hom_G(\Pi_{x,1}^{\oplus I}, \Pi_{x,1}^{\oplus J})$ comes from a morphism $g \in \Hom_{G_x}(\sigma_{x,1}^{\oplus I}, \sigma_{x,1}^{\oplus J})$. Using $\cInd_{G_x}^{G(F)}$ is exact we see that $\pi'$ is the image of $\coker(g) \in \mathcal{B}_{x,1}$ under $\cInd_{G_x}^{G(F)}$. Therefore, $\cInd_{G_x}^{G(F)}: \mathcal{B}_{x,1} \to \mathcal{C}_{x,1}$ is essentially surjective.
			
		\end{proof}
		
		\begin{lemma}\label{Lem Thm Hom implies fully faithful}
			$\cInd_{G_x}^{G(F)}|_{\mathcal{B}_{x,1}}: \mathcal{B}_{x,1} \to \Rep_{\Lambda}(G(F))$ is fully faithful.
		\end{lemma}
		
		\begin{proof}
			Let $\rho_1, \rho_2 \in \mathcal{B}_{x,1}$. By the regular supercuspidal reduction assumption of $\rho$ and Theorem \ref{Thm SC Red}, $\rho_1, \rho_2$ has supercuspidal reduction. Hence the assumptions of Theorem \ref{Thm Hom} are satisfied and we compute using the first case of Theorem \ref{Thm Hom} that
			$$\Hom_G(\cInd_{G_x}^{G(F)}\rho_1, \cInd_{G_x}^{G(F)}\rho_2) \cong \Hom_{G_x}(\rho_1, \rho_2).$$
			In other words, $\cInd_{G_x}^{G(F)}|_{\mathcal{B}_{x,1}}: \mathcal{B}_{x,1} \to \Rep_{\Lambda}(G(F))$ is fully faithful.
		\end{proof}
		
		\begin{lemma}\label{Lem Indec}
			The image of $\mathcal{B}_{x,1}$ under $\cInd_{G_x}^{G(F)}$ is indecomposable as an abelian category.
		\end{lemma}
		
		\begin{proof}
			The point is that $\cInd_{G_x}^{G(F)}|_{\mathcal{B}_{x,1}}: \mathcal{B}_{x,1} \to \Rep_{\Lambda}(G(F))$ is not only fully faithful, i.e., an equivalence of categories onto the essential image, but also an equivalence of \textbf{abelian} categories onto the essential image. Indeed, it suffices to show that $\cInd_{G_x}^{G(F)}|_{\mathcal{B}_{x,1}}: \mathcal{B}_{x,1} \to \Rep_{\Lambda}(G(F))$ preserves kernels, cokernels, and finite (bi-)products. But this follows from the next Lemma \ref{Lem Sum}.
			
			Assume otherwise that the essential image of $\mathcal{B}_{x,1}$ under $\cInd_{G_x}^{G(F)}$ is decomposable, then so is $\mathcal{B}_{x,1}$. But $\mathcal{B}_{x,1}$ is a block, hence indecomposable, contradiction!
		\end{proof}
		
		\begin{lemma}\label{Lem Sum}
			$\cInd_{G_x}^{G(F)}$ is exact and commutes with arbitrary direct sums.
		\end{lemma}
		
		\begin{proof}
			For the statement that $\cInd_{G_x}^{G(F)}$ is exact, we refer to \cite[I.5.10]{vigneras1996representations}.
			
			We show that $\cInd_{G_x}^{G(F)}$ commutes with arbitrary direct sums. Indeed, $\cInd_{G_x}^{G(F)}$ is a left adjoint (see \cite[I.5.7]{vigneras1996representations}), hence commutes with arbitrary colimits. In particular, it commutes with arbitrary direct sums.
		\end{proof}

		\section{Regular supercuspidal blocks for finite groups of Lie type}\label{Sec Reg Cusp}
		
		In this section, we prove Theorem \ref{Thm SC Red}. As mentioned before, we made the \textbf{regular} assumption so that the conclusion of Theorem \ref{Thm SC Red} -- all representations in such a block have supercuspidal reduction -- is true. So the readers are welcome to skip this section for a first reading and pretend that we begin with a block in which all representations have supercuspidal reduction.
		
		Fix a prime number $p$. Let $\ell$ be a prime number different from $p$. Let $q$ be a power of $p$. Let $\Lambda:=\overline{\mathbb{Z}}_{\ell}$ be the coefficients of representations.
		
		The main body of this section is to define regular supercuspidal blocks with $\Lambda=\overline{\mathbb{Z}}_{\ell}$-coefficients of a finite group of Lie type, and to show that a regular supercuspidal block consists purely of supercuspidal representations.

		\begin{definition}[{\cite[I.4.1]{vigneras1996representations}}]
			\begin{enumerate}Let $\Lambda'$ be any ring.
				\item Let $H$ be a profinite group, a \textbf{representation of $H$ with $\Lambda'$-coefficients} $(\pi, V)$ is a $\Lambda'$-module $V$, together with a $H$-action $\pi: H \to GL_{\Lambda'}(V)$.
				\item A representation of $H$ with $\Lambda'$-coefficients is called \textbf{smooth} if for any $v \in V$, the stabilizer $Stab_H(v) \subseteq H$ is open.
			\end{enumerate}
		\end{definition}
		
		From now on, all representations are assumed to be smooth. The category of smooth representations of $H$ with $\Lambda'$-coefficients is denoted by $\Rep_{\Lambda'}(H)$.
		
		\subsection{Regular blocks}
		
		\textbf{The following notations are used in this subsection only.} Let $\mathcal{G}$ be a split reductive group scheme over $\mathbb{Z}$. Let $\mathbb{G}:=\mathcal{G}(\overline{\mathbb{F}_q})$, $G:=\mathbb{G}^F=\mathcal{G}(\mathbb{F}_q)$, where $F$ is the Frobenius. By abuse of notation, we sometimes identify the group scheme $\mathcal{G}_{\overline{\mathbb{F}_q}}$ with its $\overline{\mathbb{F}_q}$-points $\mathbb{G}$. Let $\mathbb{G}^*$ be the dual group (over $\overline{\mathbb{F}_q}$) of $\mathbb{G}$, and $F^*$ the dual Frobenius (see \cite[Section 4.2]{carter1985finite}). Fix an isomorphism $\overline{\mathbb{Q}}_{\ell} \cong \mathbb{C}$. 
		
		The definition of regular supercuspidal blocks and regular supercuspidal representations of a finite group of Lie type $\Gamma$ involves modular Deligne-Lusztig theory and block theory. We refer to \cite{deligne1976representations}, \cite{carter1985finite}, and \cite{digne2020representations} for Deligne-Lusztig theory, \cite{michel1989bloc} and \cite{broue1990isometries} for modular Deligne-Lusztig theory, and \cite[Appendix B]{bonnafe2010representations} for generalities on blocks. 
		
		First, let us recall a result in Deligne-Lusztig theory (see \cite[Proposition 11.1.5]{digne2020representations}). 
		
		\begin{proposition}\label{Prop dual torus}
			The set of $\mathbb{G}^F$-conjugacy classes of pairs $(\mathbb{T}, \theta)$, where  $\mathbb{T}$ is a $F$-stable maximal torus of  $\mathbb{G}$ and $\theta \in \widehat{\mathbb{T}^F}$, is in non-canonical bijection with the set of $\mathbb{G^*}^{F^*}$-conjugacy classes of pairs $(\mathbb{T}^*, s)$, where $s$ is a semisimple element of $\mathbb{G}^*$ and $\mathbb{T}^*$ is a $F^*$-stable maximal torus of $\mathbb{G}^*$ such that $s \in {\mathbb{T}^*}^{F^*}$.  Moreover, we can and will fix a compatible system of isomorphisms $\mathbb{F}_{q^n}^* \cong \mathbb{Z}/(q^n-1)\mathbb{Z}$ to pin down this bijection.
		\end{proposition}
		
		Now let $s$ be a \textbf{strongly regular semisimple} 
		%(\textcolor{red}{Is this the standard terminology?}) 
		element of $G^*={\mathbb{G}^*}^{F^*}$ (note that we require $s$ to be fixed by $F^*$ here), i.e., the centralizer $C_{\mathbb{G}^*}(s)$ is a $F^*$-stable maximal torus, denoted $\mathbb{T}^*$. Let $\mathbb{T}$ be the dual torus of $\mathbb{T}^*$. Let $T=\mathbb{T}^F$ and $T^*={\mathbb{T}^*}^{F^*}$. Let $T_\ell$ denote the $\ell$-part of $T$.
		
		Recall for $s$ strongly regular semisimple, the (rational) Lusztig series $\mathcal{E}(G, (s))$ consists of only one element, namely, $\pm R_T^G(\hat{s})$, where $\hat{s}=\theta$ is such that $(\mathbb{T}, \theta)$ corresponds to $(\mathbb{T}^*, s)$ via the bijection in Proposition \ref{Prop dual torus}. Here and after, the sign $\pm$ is taken such that $\pm R_T^G(\hat{s})$ is an honest representation (see \cite[Section 7.5]{carter1985finite}).
		%	(This follows from, for example, Broué's equivalence. See Theorem \ref{Thm Broué} below.
		%	% Better explanation?
		%	)
		
		\textbf{From now on, we assume moreover that $s \in {\mathbb{G}^*}^{F^*}$ has order prime to $\ell$.} In other words, we assume that $s \in G^*={\mathbb{G}^*}^{F^*}$ is a \textbf{strongly regular semisimple $\ell'$-element}. We are going to define regular blocks. We refer to \cite[Appendix B]{bonnafe2010representations} for generalities on blocks.
		
		Define the \textbf{$\ell$-Lusztig series} 
		$$\mathcal{E}_\ell(G, (s)):=\{\pm R_T^G(\hat{s}\eta)\;|\; \eta \in \widehat{T_\ell}\}.$$ Note the notation $\mathcal{E}_\ell(T, (s))$ also makes sense by putting $G=T$.
		
		By \cite{michel1989bloc}, $\mathcal{E}_\ell(G, (s))$ is a union of $\ell$-blocks of $\Rep_{\overline{\mathbb{Q}}_{\ell}}(G)$. Such a block (or more precisely, a union of blocks) is called a \textbf{($\ell$-)regular block}. Let $e_s^G$ denotes the corresponding central idempotent in the group algebra $\overline{\mathbb{Z}}_{\ell}G$. Note $e_s^T$ also makes sense by putting $G=T$. We shall see later that a regular block is indeed a block, i.e., indecomposable.\footnote{This follows from, for example, Broué's equivalence. See Theorem \ref{Thm Broué} below.}
		
		\begin{definition}[Regular blocks]\label{Def Regular Block}
			Let $s \in G^*={\mathbb{G}^*}^{F^*}$ be a strongly regular semisimple $\ell'$-element.
			We call the summand $\overline{\mathbb{Z}}_{\ell}Ge_s^G$ of the group algebra $\overline{\mathbb{Z}}_{\ell}G$ corresponding to the central idempotent $e_s^G\overline{\mathbb{Z}}_{\ell}G \in \overline{\mathbb{Z}}_{\ell}G$ the \textbf{regular block} associated to $s$. Let $\mathcal{A}_s:=\overline{\mathbb{Z}}_{\ell}Ge_s^G\Modl$ be the corresponding category of modules, this is also referred to as a regular block, by abuse of notation.
			
%			Similarly, the block $\overline{\mathbb{F}}_{\ell}Ge_s^G$ is called a $\overline{\mathbb{F}}_{\ell}$-block. (However, this notion won't be used later.)
			
		\end{definition}
		
%		\begin{remark}
%			Above all, ``a block" can have three different meanings: $\ell$-block, $\overline{\mathbb{Z}}_{\ell}$-block, and $\overline{\mathbb{F}}_{\ell}$-block. But they are in one-one correspondence to each other, so we often abuse the notation and simply call it ``a block".
%		\end{remark}

%		\begin{remark}
%			We will see later in Theorem \ref{Theorem Pure Cuspidality} that a regular supercuspidal block consists only of supercuspidal representations. In the end, the above definition is equivalent to requiring the torus $\mathbb{T}^F$ to be elliptic, i.e., not contained in any proper parabolic subgroup of $\mathbb{G}^F$ (see Lemma \ref{Lemma Q_l-bar cuspidal}). This is because $\overline{\mathbb{Z}}_{\ell}$-cuspidality can be checked over $\overline{\mathbb{Q}}_{\ell}$ (see the proof of Theorem \ref{Theorem Pure Cuspidality}).
%		\end{remark}
		
		Thanks to \cite{broue1990isometries}, we understand the block $\mathcal{A}_s=\overline{\mathbb{Z}}_{\ell}Ge_s^G\Modl$ quite well. Roughly speaking, it is equivalent to the block a finite torus, via Deligne-Lusztig induction. This is what we are going to explain now.
		
		Let $\mathbb{B} \subseteq \mathbb{G}$ be a Borel subgroup containing our torus $\mathbb{T}$, and let $\mathbb{U}$ be the unipotent radical of $\mathbb{B}$. Let $X_{\mathbb{U}}$ be the Deligne-Lusztig variety defined by
		$$X_{\mathbb{U}}:=\{g \in \mathbb{G} \;|\; g^{-1}F(g) \in \mathbb{U}\}.$$
		
		The main result of \cite{broue1990isometries} is the following: The Deligne-Lusztig induction 
		$$\pm R_T^G: \overline{\mathbb{Z}}_{\ell}T\Modl \to \overline{\mathbb{Z}}_{\ell}G\Modl$$ induces an equivalence of categories between the blocks $\overline{\mathbb{Z}}_{\ell}Te_s^T\Modl$ and $\overline{\mathbb{Z}}_{\ell}Ge_s^G\Modl$.\footnote{In particular, one can deduce that the irreducible objects in $\overline{\mathbb{F}}_{\ell}Ge_s^G\Modl$ lift to $\overline{\mathbb{Z}}_{\ell}$.} More precisely, let us state it as the following theorem.
		
		\begin{theorem}[Broué's equivalence, {\cite[Theorem 3.3]{broue1990isometries}}]\label{Thm Broué}
			With the previous assumptions and notations, assume that $X_{\mathbb{U}}$ is affine of dimension $d$ (which is the case if $q$ is greater than the Coxeter number of $\mathbb{G}$.). Then the cohomology complex $R\Gamma_c(X_{\mathbb{U}}, \overline{\mathbb{Z}}_{\ell})=R\Gamma_c(X_{\mathbb{U}}, {\mathbb{Z}_\ell}) \otimes_{\mathbb{Z}_\ell}$$\overline{\mathbb{Z}}_{\ell}$ is concentrated in degree $d=dimX_{\mathbb{U}}$. In addition, the $(\overline{\mathbb{Z}}_{\ell}Ge_s^G, \overline{\mathbb{Z}}_{\ell}Te_s^T)$-bimodule $e_s^GH_c^d(X_{\mathbb{U}}, \overline{\mathbb{Z}}_{\ell})e_s^T$ induces an equivalence of categories
			$$e_s^GH_c^d(X_{\mathbb{U}}, \overline{\mathbb{Z}}_{\ell})e_s^T \otimes_{\overline{\mathbb{Z}}_{\ell}Te_s^T}-: \overline{\mathbb{Z}}_{\ell}Te_s^T\Modl \longrightarrow \overline{\mathbb{Z}}_{\ell}Ge_s^G\Modl.$$
		\end{theorem}
		
		\textbf{From now on, we assume that the above theorem holds for all finite groups of Lie type that we encounter in this paper.} We hope that this is not a severe restriction. This is the case at least when $q$ is greater than the Coxeter number of $\mathbb{G}$.
		
		\begin{remark}\label{Remark: T_{ell}}
			The category $\overline{\mathbb{Z}}_{\ell}Te_s^T\Modl$ is equivalent to the category $\overline{\mathbb{Z}}_{\ell}T_{\ell}\Modl$, where $T_{\ell}$ is the order-$\ell$-part of $T$. $\overline{\mathbb{Z}}_{\ell}T_{\ell}\Modl$ is essentially the category of representations of some product of $\mathbb{Z}/\ell^{k_i}\mathbb{Z}$'s. In particular, it has a unique irreducible representation (simple object), which is already defined over $\overline{\mathbb{F}}_{\ell}$. Let us denote its corresponding character by $\theta_s: T \to \overline{\mathbb{F}}_{\ell}^*$. Accordingly, $\overline{\mathbb{Z}}_{\ell}Ge_s^G\Modl$ has a unique simple object $\pm R_T^G(\theta_s)$.
		\end{remark}
		
%		We now define regular supercuspidal representations as those representations that occur in some regular cuspidal block. The term ``cuspidal" in the name ``regular cuspidal" shall be justified later by Theorem \ref{Theorem Pure Cuspidality}.

		\subsection{Regular supercuspidal blocks}
		
		Let us first recall the definition of supercuspidal representations.
		
		\begin{definition}\label{Def supercuspidal}
			
			\begin{enumerate}
				\item An irreducible representation is called \textbf{supercuspidal} if it does not occur as a subquotient of any proper parabolic induction.
				\item A representation is called \textbf{supercuspidal} if all its irreducible subquotients are supercuspidal.
			\end{enumerate}
		\end{definition}
		
		Now let us define regular supercuspidal blocks and regular supercuspidal representations.
		
		\begin{definition}\label{Definition regular supercuspidal block}
			%			By a \textbf{regular cuspidal block}, we mean a regular block which contains a cuspidal representation.
			By a \textbf{regular supercuspidal block}, we mean a regular block $\mathcal{A}_s$ whose unique simple object $\pm R_T^G(\theta_s)$ (see Remark \ref{Remark: T_{ell}} for definition) is supercuspidal.
		\end{definition}
		
		\begin{definition}\label{Def regular supercuspidal}\
			
			\begin{enumerate}
				\item An irreducible representation is called \textbf{regular supercuspidal} if it lies in a regular supercuspidal block.
				\item A representation is called \textbf{regular supercuspidal} if all its irreducible subquotients are regular supercuspidal.
			\end{enumerate}
			%			Let $G$ be a finite group of Lie type. Let $\Lambda=\overline{\mathbb{Z}}_{\ell}$. Let $\rho \in \Rep_{\Lambda}(G)$. Then $\rho$ is called \textbf{regular supercuspidal} if each of its irreducible subquotient $\rho_i$ is cuspidal (see Definition \ref{Def Cuspidal}) and lies in a regular supercuspidal $\overline{\mathbb{Z}}_{\ell}$-block $\mathcal{A}_{s_i}$ of $G$.
		\end{definition}
		
        It is clear from the definitions that we have the following proposition.
		
		\begin{proposition}\label{Theorem Pure SC}
			Let $\mathcal{A}_s$ be a regular supercuspidal block. Then any representation in this block is supercuspidal.
		\end{proposition}
		
		\begin{proof}
			By definition of supercuspidality, it suffices to check that any irreducible representation in this block is supercuspidal. But as we noted before in Remark \ref{Remark: T_{ell}}, $\mathcal{A}_s$ has only one irreducible representation -- $\pm R_T^G(\theta_s)$ (see Remark \ref{Remark: T_{ell}}), which we assumed to be supercuspidal in the definition of regular supercuspidal block. So we win!
		\end{proof}

    \subsection{Proof of Theorem \ref{Thm SC Red} on supercuspidal reduction}
		
		We now apply the previous results on finite groups of Lie type to representations of parahoric subgroups of a $p$-adic group. For this, we show that the inflation induces an equivalence of categories between (certain summand of) the category of representations of a finite reductive group and the corresponding parahoric subgroup (see Subsection \ref{Subsection_inflation}). 
		
		\textbf{Let us get back to the notations at the beginning of this chapter.}
		
		Let $G$ be a connected split reductive group scheme over $\mathbb{Z}$, which is simply connected. Let $F$ be a non-archimedean local field, with ring of integers $\mathcal{O}_F$ and residue field $k_F \cong \mathbb{F}_q$ of residue characteristic $p$. Let $x$ be a vertex of the Bruhat-Tits building $\mathcal{B}(G, F)$, $G_x$ the parahoric subgroup associated to $x$, $G_x^+$ its pro-unipotent radical. Recall that $\overline{G_x}:=G_x/G_x^+$ is a generalized Levi subgroup of $G(k_F)$ with root system $\Phi_x$, see \cite[Theorem 3.17]{rabinoff2003bruhat}.
		
		Let $\Lambda=\overline{\mathbb{Z}}_{\ell}$, with $\ell \neq p$. Let $\rho \in \Rep_{\Lambda}(G_x)$ be an irreducible representation of $G_x$, which is trivial on $G_x^+$ and whose reduction to the finite group of Lie type $\overline{G_x}=G_x/G_x^+$ is regular supercuspidal. 
		%We make this a definition for later use.
		
		%	\begin{definition}
			%		Let $\rho \in \Rep_{\Lambda}(G_x)$. We say $\rho$ \textbf{has cuspidal reduction} (resp. \textbf{has regualr cuspidal reduction}), if $\rho$ is trivial on $G_x^+$ and whose reduction to the finite group of Lie type $\overline{G_x}=G_x/G_x^+$ is cuspidal (resp. regular cuspidal). Let's denote the reduction of $\rho$ modulo $G_x^+$ by $\overline{\rho} \in \Rep_{\Lambda}(\overline{G_x})$.
			%	\end{definition}
		
		In other words, we start with an irreducible representation $\rho \in \Rep_{\Lambda}(G_x)$ that has regular supercuspidal reduction. Let $\mathcal{B}_{x,1}$ be the ($\overline{\mathbb{Z}}_{\ell}$-)block of $\Rep_{\Lambda}(G_x)$ containing $\rho$. We can now prove Theorem \ref{Thm SC Red}, which we restate as follows.
		
		\begin{theorem} \label{Thm SC Red restate}
			Let $\rho \in \Rep_{\Lambda}(G_x)$ be an irreducible representation of $G_x$, which has regular supercuspidal reduction. Let $\mathcal{B}_{x,1}$ be the $\overline{\mathbb{Z}}_{\ell}$-block of $\Rep_{\Lambda}(G_x)$ containing $\rho$. Then any $\rho' \in \mathcal{B}_{x,1}$ has supercuspidal reduction.
		\end{theorem}
		
		\begin{proof}
			Let $\overline{\rho} \in \Rep_{\Lambda}(\overline{G_x})$ be the reduction of $\rho$ modulo $G_x^+$. $\overline{\rho}$ is irreducible (since $\rho$ is) and regular supercuspidal by assumption, so it is of the form $\pm R_T^G(\theta_s)$ (see Remark \ref{Remark: T_{ell}}), for some strongly regular semisimple $\ell'$-element $s$ of the finite dual group $\overline{G_x}^*$ (see Definition \ref{Def regular supercuspidal}).  
			
			Let $\Rep_{\Lambda}(G_x)_0$ be the full subcategory of $\Rep_{\Lambda}(G_x)$ consisting of representations of $G_x$ that are trivial on $G_x^+$. The key observation is that $\Rep_{\Lambda}(G_x)_0$ is a summand (as abelian category) of $\Rep_{\Lambda}(G_x)$ (see Lemma \ref{Lem Summand}).
			
			Then since $\rho \in \Rep_{\Lambda}(G_x)_0$, its block $\mathcal{B}_{x,1}$ is a summand of $\Rep_{\Lambda}(G_x)_0$.
			
			On the other hand, notice that the inflation functor induces an equivalence of categories between $\Rep_{\Lambda}(\overline{G_x})$ and $\Rep_{\Lambda}(G_x)_0$, with inverse the functor of reduction modulo $G_x^+$.
			So the blocks of $\Rep_{\Lambda}(\overline{G_x})$ and $\Rep_{\Lambda}(G_x)_0$ are in one-to-one correspondence. Let $\mathcal{A}_{x,1}$ be the corresponding block of $\Rep_{\Lambda}(\overline{G_x})$ to $\mathcal{B}_{x,1}$. Then $\mathcal{A}_{x,1}$ is the regular supercuspidal block $\mathcal{A}_s$ corresponding to $s$ (recall that $\overline{\rho}=\pm R_T^G(\theta_s)$). By Theorem \ref{Theorem Pure SC}, $\mathcal{A}_s$ consists purely of supercuspidal representation. Therefore, $\mathcal{B}_{x,1}$ consists purely of representations that have supercuspidal reductions. 
		\end{proof}

     \subsection{Inflation induces an equivalence}  \label{Subsection_inflation}

		\begin{lemma}\label{Lem Summand}
			Let $\Rep_{\Lambda}(G_x)_0$ be the full subcategory of $\Rep_{\Lambda}(G_x)$ consisting of representations of $G_x$ that are trivial on $G_x^+$. Then $\Rep_{\Lambda}(G_x)_0$ is a summand as abelian category of $\Rep_{\Lambda}(G_x)$.
		\end{lemma}
		
		\begin{remark}
			A similar proof as \cite[Appendix]{dat2009finitude} should work. Nevertheless, I include here an alternative proof.\footnote{Note that one of the main difficulty is that $\mathcal{H}_{\Lambda}(G_x)$ might be non-unital.}
		\end{remark}
		
		\begin{proof}
			Note that $G_x^+$ is pro-$p$ (see \cite[II.5.2.(b)]{vigneras1996representations}), in particular, it has pro-order invertible in $\Lambda$. So we have a normalized Haar measure $\mu$ on $G_x$ such that $\mu(G_x^+)=1$ (see \cite[I.2.4]{vigneras1996representations}). The characteristic function $e:=1_{G_x^+}$ is an idempotent of the Hecke algebra $\mathcal{H}_{\Lambda}(G_x)$ under convolution with respect to the Haar measure $\mu$. We shall show that $e=1_{G_x^+}$ cuts out $\Rep_{\Lambda}(G_x)_0$ as a summand of $\Rep_{\Lambda}(G_x) \cong \mathcal{H}_{\Lambda}(G_x)\Modl$.
			
			Recall that we have a descending filtration $\{G_{x,r} \;|\; r\in \mathbb{R}_{>0}\}$ of $G_x$ such that 
			\begin{enumerate}
				\item $\forall r \in \mathbb{R}_{>0}, G_{x,r}$ is an open compact pro-$p$ normal subgroup of $G_x$.
				\item $G_{x,r}$ form a neighborhood basis of $1$ inside $G_x$. 
				\item The set $\{G_{x,r} \;|\; r\in \mathbb{R}_{>0}\}$ is countable.
			\end{enumerate}
			(See \cite[II.5.1]{vigneras1996representations}.)
			
			Let $\{r_n \;|\; n \in \mathbb{Z}_{\geq 0}\}$ be the jumps of the filtration $\{G_{x,r} \;|\; r\in \mathbb{R}_{>0}\}$. We get a descending filtration $\{K_n:=G_{x, r_n} \;|\; n \in \mathbb{Z}_{\geq 0}\}$ of $G_x$, with $K_0=G_x^+$. Note that the relative Hecke algebra $\mathcal{H}_{\Lambda}(G_x, K_n)$ (consisting of bi-$K_n$-invariant functions) is unital with unit $e_n$, the normalized characteristic function of $K_n$ (in particular, $e_o=e$). A direct computation shows that $e_i*e_j=e_i$ for all $i<j$. Moreover, one can deduce that the $e_n$'s are central, i.e.,
			$$e_n*f=f*e_n,$$
			for all $f \in \mathcal{H}_{\Lambda}(G_x)$, by writing $f$ as a linear combination of characteristic functions that are $G_x$-translations of $e_n$'s. In addition, $\{e_n \;|\; n \in \mathbb{Z}_{\geq 0}\}$ generates $\mathcal{H}_{\Lambda}(G_x)$ as a $\mathcal{H}_{\Lambda}(G_x)$-module.
			
			To summarize, we have a set of central idempotents $\{e_n \;|\; n \in \mathbb{Z}_{\geq 0}\}$ that generates $\mathcal{H}_{\Lambda}(G_x)$. Now we modify them such that they become orthogonal. Indeed, we put $\tilde{e_0}:=e_0=e$ and $\tilde{e_n}:=e_n-e_{n-1}$. Then we can check that for $i < j$,
			\begin{equation*}
				\begin{aligned}
					&\tilde{e_i}*\tilde{e_j}=(e_i-e_{i-1})*(e_j-e_{j-1})\\
					=\;&e_i*e_j-e_{i-1}*e_j-e_i*e_{j-1}+e_{i-1}*e_{j-1}\\
					=\;& e_i-e_{i-1}-e_i+e_{i-1}=0.
				\end{aligned}
			\end{equation*}
			
			Therefore, 
			$$\mathcal{H}_{\Lambda}(G_x)\Modl \cong \bigoplus_{n \in \mathbb{Z}_{\geq 0}}\tilde{e_n}\mathcal{H}_{\Lambda}(G_x)\tilde{e_n}\Modl.$$ In particular, $\Rep_{\Lambda}(G_x)_0 \cong e\mathcal{H}_{\Lambda}(G_x)e\Modl$ is a summand of $\Rep_{\Lambda}(G_x) \cong \mathcal{H}_{\Lambda}(G_x)\Modl$.

		\end{proof}
		
		\begin{lemma}\label{Lemma A to B}
			The inflation induces an equivalence of categories between $\Rep_{\Lambda}(\overline{G_x})$ and $\Rep_{\Lambda}(G_x)_0$. In particular, let $\rho$ be as in Theorem \ref{Thm SC Red restate} and let $\mathcal{A}_{x,1}$ be the block of $\Rep_{\Lambda}(\overline{G_x})$ containing $\overline{\rho}$, then the inflation induces an equivalence of categories 
			$$\mathcal{A}_{x,1} \cong \mathcal{B}_{x,1}.$$
		\end{lemma}
		
		\begin{proof}
			The inverse functor is given by the reduction modulo $G_x^+$. One can check by hand that they are equivalences of categories.
		\end{proof}

		\section{$\Hom$ between compact inductions}\label{Sec Pf Thm Hom}
		
		Let us now prove Theorem \ref{Thm Hom} which computes the $\Hom$ between compact inductions of $\rho_1$ and $\rho_2$, assuming that one of them has supercuspidal reduction.
		
		\begin{proof}[Proof of Theorem \ref{Thm Hom}]
			\begin{equation*}
				\begin{aligned}
					&\Hom_G(\cInd_{G_x}^{G(F)}\rho_1, \cInd_{G_y}^{G(F)}\rho_2)\\
					=\;&\Hom_{G_x}\left(\rho_1,(\cInd_{G_y}^{G(F)}\rho_2)|_{G_x}\right)\\
					=\;& \Hom_{G_x}\left(\rho_1, \bigoplus_{g \in {G_y\backslash G(F)/G_x}}\cInd_{G_x \cap g^{-1}G_yg}^{G_x}\rho_2(g-g^{-1})\right)
				\end{aligned}
			\end{equation*}
			
			Recall that $g^{-1}G_yg=G_{g^{-1}.y}$. So it suffices to show that for $g \in G(F)$ with $G_x \cap g^{-1}G_yg \neq G_x$, or equivalently, for $g \in G(F)$ with $g.x \neq y$ (since $x$ and $y$ are vertices), it holds that
			$$\Hom_{G_x}\left(\rho_1, \cInd_{G_x \cap g^{-1}G_yg}^{G_x}\rho_2(g-g^{-1})\right)=0.$$
			
			Note $G_x/(G_x \cap g^{-1}G_yg)$ is compact, hence $\cInd_{G_x \cap g^{-1}G_yg}^{G_x}=\operatorname{Ind}_{G_x \cap g^{-1}G_yg}^{G_x}$, and we have Frobenius reciprocity in the other direction
			$$\Hom_{G_x}\left(\rho_1, \cInd_{G_x \cap g^{-1}G_yg}^{G_x}\rho_2(g-g^{-1})\right) \cong \Hom_{G_x \cap g^{-1}G_yg}\left(\rho_1, \rho_2(g-g^{-1})\right).$$
			
			So it suffices to show that for $g \in G(F)$ with $g.x \neq y$,
			$$\Hom_{G_x \cap g^{-1}G_yg}\left(\rho_1, \rho_2(g-g^{-1})\right)=0.$$
			Note now this expression is symmetric with respect to $\rho_1$ and $\rho_2$, and so is the following argument.
			
			First, if $\rho_2$ has supercuspidal reduction (denoted $\overline{\rho_2}$),
			\begin{align*}    	
				& \Hom_{G_x \cap g^{-1}G_yg}\left(\rho_1, \rho_2(g-g^{-1})\right) \\
				=\;& \Hom_{G_x \cap G_{g^{-1}.y}}\left(\rho_1, \rho_2(g-g^{-1})\right) \\
				\subseteq\;& \Hom_{G_x^+ \cap G_{g^{-1}.y}}\left(\rho_1, \rho_2(g-g^{-1})\right) && %\text{By \eqref{eq:1}}
				\\
				=\;& \Hom_{G_x^+ \cap G_{g^{-1}.y}}(1^{\oplus d_1}, \rho_2(g-g^{-1})) && \text{$\rho_1$ is trivial on $G_x^+$ }\\
				=\;& \Hom_{G_{g.x}^+ \cap G_y}(1^{\oplus d_1}, \rho_2) && \text{Conjugate by $g^{-1}$}\\
				=\;& \Hom_{U_y(g.x)}(1^{\oplus d_1}, \overline{\rho_2}) && \text{Reduction modulo $G_y^+$. See below.}\\
				=\;& 0 && \text{$\overline{\rho_2}$ is supercuspidal. See below.}
			\end{align*}
			
			The last two equalities need some explanation. 
			
			The former one uses the following consequence from Bruhat-Tits theory: If $x_1$ and $x_2$ are two different vertices of the Bruhat-Tits building, then $\overline{G_{x_i}}:=G_{x_i}/G_{x_i}^+$ is a generalized Levi subgroup of $\overline{G}=G(\mathbb{F}_q)$, for $i=1, 2$. Moreover, $G_{x_1} \cap G_{x_2}$ projects onto a proper parabolic subgroup $P_{x_1}(x_2)$ of $\overline{G_{x_1}}$ under the reduction map $G_{x_1} \to \overline{G_{x_1}}$. In addition, $G_{x_1} \cap G_{x_2}^+$ projects onto $U_{x_1}(x_2)$, the unipotent radical of $P_{x_1}(x_2)$, under the reduction map $G_{x_1} \to \overline{G_{x_1}}$. For details, see Lemma \ref{Lem Passage to Residue Field} below. Note that the assumption of Lemma \ref{Lem Passage to Residue Field} is satisfied since without loss of generality we may assume that $x_1=x$ and $x_2=y$ lie in the closure of a common alcove (since $G$ acts simply transitively on the set of alcoves).
			
			The latter one uses that for a supercuspidal representation $\rho$ of a finite group of Lie type $\Gamma$, 
			$$\Hom_U(1, \rho|_U)=\Hom_U(\rho|_U, 1)=0,$$
			for the unipotent radical $U$ of $P$, where $P$ is any proper parabolic subgroup of $\Gamma$. For details, see Lemma \ref{Lem Hom_U(1_U, SC)} below.
			
			Symmetrically, a similar argument works if $\rho_1$ has supercuspidal reduction. Indeed, if $\rho_1$ has supercuspidal reduction (denoted $\overline{\rho_1}$),
			\begin{align*}    	
				& \Hom_{G_x \cap g^{-1}G_yg}\left(\rho_1, \rho_2(g-g^{-1})\right) \\
				=\;& \Hom_{gG_xg^{-1} \cap G_y}\left(\rho_1(g^{-1}-g), \rho_2\right) && \text{Conjugate by $g^{-1}$}\\ 
				\subseteq\;& \Hom_{gG_xg^{-1} \cap G_y^+}\left(\rho_1(g^{-1}-g), \rho_2\right) && %\text{By \eqref{eq:1}}
				\\
				=\;& \Hom_{gG_xg^{-1} \cap G_y^+}(\rho_1(g^{-1}-g), 1^{\oplus d_2}) && \text{$\rho_2$ is trivial on $G_y^+$ }\\
				=\;& \Hom_{G_x \cap g^{-1}G_y^+g}(\rho_1, 1^{\oplus d_2}) && \text{Conjugate by $g$}\\
				=\;& \Hom_{G_x \cap G_{g^{-1}.y}^+}(\rho_1, 1^{\oplus d_2}) && \\
				=\;& \Hom_{U_x(g^{-1}.y)}(\overline{\rho_1}, 1^{\oplus d_2}) && \text{Reduction modulo $G_x^+$}\\
				=\;& 0 && \text{$\overline{\rho_1}$ is supercuspidal. }
			\end{align*}
			
		\end{proof}

		\begin{lemma}\label{Lem Passage to Residue Field}
			Let $x_1$ and $x_2$ be two points of the Bruhat-Tits building $\mathcal{B}(G, F)$. Assume that they lie in the closure of the same alcove.
			\begin{enumerate}
				\item[(i)]   The image of $G_{x_1} \cap G_{x_2}$ in $\overline{G_{x_1}}$ is a parabolic subgroup of $\overline{G_{x_1}}$. Let us denote it by $P_{x_1}(x_2)$. Moreover, the image of $G_{x_1} \cap G_{x_2}^+$ in $\overline{G_{x_1}}$ is the unipotent radical of $P_{x_1}(x_2)$. Let us denote it by $U_{x_1}(x_2)$.
				\item[(ii)] 	Assume moreover that $x_1$ and $x_2$ are two different vertices of the building. Then $P_{x_1}(x_2)$ is a proper parabolic subgroup of $\overline{G_{x_1}}$.
			\end{enumerate}
		\end{lemma}
		
		\begin{proof}
			(i) is \cite[II.5.1.(k)]{vigneras1996representations}.
			
			Let us prove (ii). It suffices to show that $G_{x_1} \neq G_{x_2}$. Assume otherwise that $G_{x_1}=G_{x_2}$, then $x_1$ and $x_2$ lie in the same facet, which contradicts the assumption that $x_1$ and $x_2$ are two different vertices.
		\end{proof}
		
		\begin{lemma}\label{Lem Hom_U(1_U, SC)}
			Let $\overline{\rho}$ be a supercuspidal representation of a finite group of Lie type $\Gamma$. Let $P$ be a proper parabolic subgroup of $\Gamma$, with unipotent radical $U$. Then
			$$Hom_U(1_U, \overline{\rho})=Hom_U(\overline{\rho}, 1_U)=0.$$
		\end{lemma}
		
		\begin{proof}
			$\Hom_U(\overline{\rho}|_U, 1_U)=\Hom_{\Gamma}(\overline{\rho}, Ind_P^{\Gamma}(\sigma))=0$, where $\sigma=Ind_U^P(1_U)$. The last equality holds because $\overline{\rho}$ is assumed to be supercuspidal. A similar argument shows that $Hom_U(1_U, \overline{\rho})=0$.  
			
%			Moreover, since $U$ is a successive extension of additive groups, $U$ is of order a power of $p$. In particular, $\ell$ does not divide the order of $U$. Hence the category of representations of $U$ with $\Lambda=\overline{\mathbb{Z}}_{\ell}$-coefficients is semisimple \textcolor{red}{No! Even $\overline{\mathbb{Z}}_{\ell}$ is not semisimple as a $\overline{\mathbb{Z}}_{\ell}$ module.}, and
%			$$\Hom_U(1_U, \overline{\rho})=\Hom_U(\overline{\rho}, 1_U)=0.$$
%			\textcolor{red}{No! $\Hom_1(\overline{\mathbb{Z}}_{\ell}, \overline{\mathbb{F}}_{\ell}) \neq \Hom_1(\overline{\mathbb{F}}_{\ell}, \overline{\mathbb{Z}}_{\ell})$, where $1$: trivial group.}
		\end{proof}

		\section{$\Pi_{x,1}$ is a projective generator}\label{Section projective generator}
		
		In this subsection, we prove Theorem \ref{Thm Proj}: $\Pi_{x,1}$ is a projective generator of $\mathcal{C}_{x,1}$. Before doing this, let us recall the setting. Fix a vertex $x$ of the building of $G$. Let $\rho \in \Rep_{\Lambda}(G_x)$ be a representation which is trivial on $G_x^+$ and whose reduction to $\overline{G_x}=G_x/G_x^+$ is regular supercuspidal, $\pi=\cInd_{G_x}^{G(F)}\rho$ as before. Let $\mathcal{B}_{x,1}$ be the block of $\Rep_{\Lambda}(G_x)$ containing $\rho$, and $\mathcal{C}_{x,1}$ the block of $\Rep_{\Lambda}(G(F))$ containing $\pi$. 
		
		Let $V$ be the set of equivalence classes of vertices of the Bruhat-Tits building $\mathcal{B}(G, F)$ up to $G(F)$-action. For $y \in V$, let $\sigma_y:=\cInd_{G_y^+}^{G_y}\Lambda$. Let $\Pi:=\bigoplus_{y \in V}\Pi_y$ where $\Pi_y:=\cInd_{G_y^+}^{G(F)}\Lambda$. Then $\Pi$ is a projective generator of the category of depth-zero representations $\Rep_{\Lambda}(G(F))_0$, see \cite[Appendix]{dat2009finitude}. Let $\sigma_{x,1}:=(\sigma_x)|_{\mathcal{B}_{x,1}} \in \mathcal{B}_{x,1} \xhookrightarrow{summand} \Rep_{\Lambda}(G_x)$ be the $\mathcal{B}_{x,1}$-summand of $\sigma_x$. Let $\Pi_{x,1}:=\cInd_{G_x}^{G(F)}\sigma_{x,1}$.
		
		Let us summarize the setting in the following diagram.
		
		\[\begin{tikzcd}
			{\Rep_{\Lambda}(G_x)} & {\Rep_{\Lambda}(G(F))} \\
			{\Rep_{\Lambda}(G_x)_0} & {\Rep_{\Lambda}(G(F))_0} \\
			{\mathcal{B}_{x,1}} & {\mathcal{C}_{x,1}} \\
%			{\text{block of } \rho} & {\text{block of }\pi}
			\arrow[from=2-1, to=2-2]
			\arrow["{\cInd_{G_x}^{G(F)}}", from=1-1, to=1-2]
			\arrow["\subseteq"{description}, sloped, draw=none, from=2-1, to=1-1]
			\arrow["\subseteq"{description}, sloped, draw=none, from=3-1, to=2-1]
			\arrow["\subseteq"{description}, sloped, draw=none, from=3-2, to=2-2]
			\arrow["\subseteq"{description}, sloped, draw=none, from=2-2, to=1-2]
%			\arrow["{=:}"{description}, sloped, draw=none, from=4-1, to=3-1]
%			\arrow["{:=}"{description}, sloped, draw=none, from=3-2, to=4-2]
			\arrow[from=3-1, to=3-2]
		\end{tikzcd}\]
		
		\begin{theorem}
			$\Pi_{x,1}=\cInd_{G_x}^{G(F)}\sigma_{x,1}$ is a projective generator of $\mathcal{C}_{x,1}$.
		\end{theorem}
		
		\begin{proof}
			First, let $\Rep_{\Lambda}(G_x)_0$ be the full subcategory of $\Rep_{\Lambda}(G_x)$ consisting of representations that are trivial on $G_x^+$ (Do not confuse with $\Rep_{\Lambda}(G(F))_0$, the depth-zero category of $G$). Note that $\Rep_{\Lambda}(G_x)_0$ is a summand of $\Rep_{\Lambda}(G_x)$ (see Lemma \ref{Lem Summand}).
			
			Second, note that $\Rep_{\Lambda}(G_x)_0 \cong \Rep_{\Lambda}(\overline{G_x})$. We may assume that $$\Rep_{\Lambda}(G_x)_0=\mathcal{B}_{x,1} \oplus ... \oplus \mathcal{B}_{x,m}$$
			is its block decomposition. So that $\sigma_x=\sigma_{x,1}\oplus...\oplus\sigma_{x,m}$ accordingly. Write $\sigma_x^1:=\sigma_{x,2}\oplus...\oplus\sigma_{x,m}$. Then $\sigma_x=\sigma_{x,1} \oplus \sigma_x^1$, and $\Pi_x=\Pi_{x,1} \oplus \Pi_x^1$ accordingly, where $\Pi_x^1:=\cInd_{G_x}^{G(F)}\sigma_x^1$. Moreover,
			$$\Pi=\Pi_{x,1}\oplus \Pi_x^1 \oplus \Pi^x,$$
			where $\Pi^x:=\bigoplus_{y \in V, y \neq x}\Pi_y$. Let $\Pi^{x,1}:=\Pi_x^1 \oplus \Pi^x$, then we have
			$$\Pi=\Pi_{x,1} \oplus \Pi^{x,1}.$$
			
			Recall that $\Pi$ is a projective generator of the category of depth-zero representations $\Rep_{\Lambda}(G(F))_0$. This implies that 
			$$\Hom_G(\Pi, -): \Rep_{\Lambda}(G(F))_0 \to \Modr\End_G(\Pi)$$
			is an equivalence of categories. See \cite[Lemma 22]{bernsteindraft}.
			
			Next, it is not hard to see that Theorem \ref{Thm Hom} implies that 
			$$\Hom_G(\Pi_{x,1}, \Pi^{x,1})=\Hom_G(\Pi^{x,1}, \Pi_{x,1})=0,$$
			see Lemma \ref{Lem Ortho}. This implies that $$\Modr\End_G(\Pi) \cong \Modr\End_G(\Pi_{x,1}) \oplus \Modr\End_G(\Pi^{x,1})$$ is an equivalence of categories.
			
			Now we can combine the above to show that $\Pi^{x,1}$ does not interfere with $\Pi_{x,1}$, i.e.,
			$$\Hom_G(\Pi^{x,1}, X)=0,$$
			for any object $X \in \mathcal{C}_{x,1}$ (see Importent Lemma \ref{Lem Gen}).
			
			However, since $\Pi$ is a projective generator of $\Rep_{\Lambda}(G(F))_0$, we have
			$$\Hom_G(\Pi, X) \neq 0,$$
			for any $X \in \mathcal{C}_{x,1}$. This together with the last paragraph implies that 
			$$\Hom_G(\Pi_{x,1}, X) \neq 0,$$
			for any $X \in \mathcal{C}_{x,1}$, i.e. $\Pi_{x,1}$ is a generator of $\mathcal{C}_{x,1}$.
			
			Finally, note that $\Pi_{x,1}$ is projective in $\Rep_{\Lambda}(G(F))_0$ since it is a summand of the projective object $\Pi$. Hence $\Pi_{x,1}$ is projective in $\mathcal{C}_{x,1}$. This together with the last paragraph implies that $\Pi_{x,1}$ is a projective generator of $\mathcal{C}_{x,1}$.

		\end{proof}

		\begin{lemma}\label{Lem Ortho}
			$$\Hom_G(\Pi_{x,1}, \Pi^{x,1})=\Hom_G(\Pi^{x,1}, \Pi_{x,1})=0.$$
		\end{lemma}
		
		\begin{proof}
			Recall that
			$\Pi^{x,1}:=\Pi_x^1 \oplus \Pi^x$.
			
			First, we compute
			$$\Hom_G(\Pi_{x,1}, \Pi_x^1)=\Hom_{G_x}(\sigma_{x,1}, \sigma_x^1)=0,$$
			where the first equality is the first case of Theorem \ref{Thm Hom} (note that $\sigma_{x,1} \in \mathcal{B}_{x,1}$, hence has supercuspidal reduction by Theorem \ref{Thm SC Red}, and hence the condition of Theorem \ref{Thm Hom} is satisfied), and the second equality is because $\sigma_{x,1}$ and $\sigma_x^1$ lie in different blocks of $\Rep_{\Lambda}(G_x)$ by definition.
			
			Second, recall that $\Pi_{x,1}=\cInd_{G_x}^{G(F)}\sigma_{x,1}$ with $\sigma_{x,1}$ having supercuspidal reduction, and $\Pi_y=\cInd_{G_y}^{G(F)}\sigma_y$. We compute 
			$$\Hom_G(\Pi_{x,1}, \Pi^x)=\bigoplus_{y \in V, y \neq x}\Hom_G(\Pi_{x,1}, \Pi_y)=0,$$
			by the second case of Theorem \ref{Thm Hom}.
			
			Combining the above three paragraphs, we get $\Hom_G(\Pi_{x,1}, \Pi^{x,1})=0$.
			
			A same argument shows that $\Hom_G(\Pi^{x,1}, \Pi_{x,1})=0$.
		\end{proof}
		
		\begin{lemma}[Important Lemma]\label{Lem Gen}
			$\Hom_G(\Pi^{x,1}, X)=0,$
			for any object $X \in \mathcal{C}_{x,1}$.
		\end{lemma}
		
		\begin{proof}
			Recall that 
			$$\Hom_G(\Pi, -): \Rep_{\Lambda}(G(F))_0 \to \Modr\End_G(\Pi) \cong \Modr\End_G(\Pi_{x,1}) \oplus \Modr\End_G(\Pi^{x,1})$$ 
			is an equivalence of categories. It is even an equivalence of abelian categories since $\Hom_G(\Pi, -)$ is exact and commutes with direct product. Hence the image of $\mathcal{C}_{x,1}$ must be indecomposable as $\mathcal{C}_{x,1}$ is indecomposable, i.e., 
			$$\Hom_G(\Pi, -)=\Hom_G(\Pi_{x,1}, -) \oplus \Hom_G(\Pi^{x,1}, -)$$
			can map $\mathcal{C}_{x,1}$ nonzeroly to only one of $\Modr\End_G(\Pi_{x,1})$ and $\Modr\End_G(\Pi^{x,1})$ (see the diagram below). 
			
			\[\begin{tikzcd}
				{\Rep_{\Lambda}(G(F))_0} &&&& {\Modr \End_G(\Pi)} \\
				\\
				{\mathcal{C}_{x,1}} &&&& {\Modr \End_G(\Pi_{x,1}) \oplus \Modr \End_G(\Pi^{x,1})}
				\arrow["{\Hom_G(\Pi, -)}", from=1-1, to=1-5]
				\arrow["{\Hom_G(\Pi_{x,1}, -) \oplus \Hom_G(\Pi^{x,1}, -)}", from=3-1, to=3-5]
				\arrow["\subseteq", sloped, draw=none, from=3-1, to=1-1]
				\arrow["\cong", sloped, draw=none, from=3-5, to=1-5]
			\end{tikzcd}\]
			
			Then it must be $\Modr\End_G(\Pi_{x,1})$ (that $\Hom_G(\Pi, -)$ maps $\mathcal{C}_{x,1}$ nonzeroly to) since 
			$$\Hom_G(\Pi_{x,1}, \pi)=\Hom_{G_x}(\sigma_{x,1}, \rho)=\Hom_{G_x}(\sigma_x, \rho) \neq 0.$$
			In other words, $\Hom_G(\Pi^{x,1}, -)$ is zero on $\mathcal{C}_{x,1}$.
			
		\end{proof}

		\section{Application: description of the block $\Rep_{\Lambda}(G(F))_{[\pi]}$}\label{Section rep application}
		
		Recall we denote $\mathcal{A}_{x,1}=\Rep_{\Lambda}(\overline{G_x})_{[\overline{\rho}]}$, $\mathcal{B}_{x,1}=\Rep_{\Lambda}(G_x)_{[\rho]}$, and $\mathcal{C}_{x,1}=\Rep_{\Lambda}(G(F))_{[\pi]}$.
		
		We have proven that the inflation along $G_x \to \overline{G_x}$ induces an equivalence of categories 
		$$\mathcal{A}_{x,1} \cong \mathcal{B}_{x,1},$$
		see Lemma \ref{Lemma A to B}. In addition, we have also proven that compact induction induces an equivalence of categories
		$$\cInd_{G_x}^{G(F)}: \mathcal{B}_{x,1} \cong \mathcal{C}_{x,1}.$$
		
		Hence $\mathcal{C}_{x,1} \cong \mathcal{A}_{x,1}$, where the latter is isomorphic to the block of a finite torus via Broué's equivalence \ref{Thm Broué}.
		
		We will see in the example (see Chapter \ref{Chapter GL_n}) of $GL_n$ that (up to central characters) such a block of a finite torus corresponds to $\QCoh(\mu)$, where $\mu$ is the group scheme of roots of unity appearing in the computation of the $L$-parameter side (see Theorem \ref{Thm X/G}).

%% file: CLLC.tex
\chapter{Example: $GL_n(F)$}\label{Chapter GL_n}

Let's apply the theories in the previous chapters to the example of $GL_n(F)$. Throughout this chapter, $G:=GL_n$.

That said, there is a little mismatch between the theories before and the example here. Namely, we assumed for simplicity in the theories that $G$ is semisimple and simply connected, while this is not the case for $G=GL_n$. However, there is only some minor difference due to the center $\mathbb{G}_m$ of $GL_n$. Let us leave it as an exercise for the readers to figure out the details.

\section{$L$-parameter side} \label{Example Lparam}
Let $\varphi \in Z^1(W_F, \hat{G}(\overline{\mathbb{F}}_{\ell}))$ be an irreducible tame $L$-parameter. Let $\psi \in Z^1(W_F, \hat{G}(\overline{\mathbb{Z}}_{\ell}))$ be any lift of $\varphi$. Let $C_{\varphi}$ be the connected component of $Z^1(W_F, \hat{G})_{\overline{\mathbb{Z}}_{\ell}}/\hat{G}$ containing $\varphi$. By Proposition \ref{Proposition: T times mu/T}, we have
%$$C_{\varphi} \cong \left(\hat{G} \times Z^1(W_F, N_{\hat{G}}(\psi_{\ell}))_{\psi_{\ell}, \overline{\psi}}\right)/C_{\hat{G}}(\psi_{\ell})_{\overline{\psi}}.$$
%Here 
%$$Z^1(W_F, N_{\hat{G}}(\psi_{\ell}))_{\psi_{\ell}, \overline{\psi}} \cong Z^1_{Ad(\psi)}(W_F, N_{\hat{G}}(\psi_{\ell})^0)_{1_{I_F^{\ell}}}.$$
%In our case, $N_{\hat{G}}(\psi_{\ell})^0$ is the diagonal torus $T$ of $GL_n$.
$$C_{\varphi} \cong [T/T] \times \mu,$$
where $T=C_{\hat{G}}(\psi_{\ell})$ is a maximal torus of $GL_n$, and $\mu=(T^{Fr=(-)^q})^0$, and $t \in T$ acting on $T$ via multiplication by $tnt^{-1}n^{-1}$, where $n=\psi(\Fr)$. To go further, let's choose a nice basis for the Weil group representations $\varphi$ and $\psi$.

Indeed, every irreducible tame $L$-parameter $\varphi$ with $\overline{\mathbb{F}}_{\ell}$-coefficients of $GL_n$ are of the form $\varphi=\Ind_{W_E}^{W_F}\eta$, where $E$ is a degree $n$ unramified extension of $F$, $W_E \cong I_F \rtimes \left<\Fr^n\right>$ is the Weil group of $E$, and $\eta: W_E \to \overline{\mathbb{F}}_{\ell}^*$ is a tame (i.e., trivial on $P_E=P_F$) character of $W_E$ such that $\{\eta, \eta^q, ..., \eta^{q^{n-1}}\}$ are distinct (see \cite[Chapter 3]{macdonald1980zeta}). To find a lift of $\varphi$ with $\overline{\mathbb{Z}}_{\ell}$-coefficients, we let $\tilde{\eta}: W_E \to \overline{\mathbb{Z}}_{\ell}^*$ be any lift of $\eta$, and let $\psi:=\Ind_{W_E}^{W_F}\tilde{\eta}$. Then under a nice basis, we can specify the matrices corresponding to the topological generator $s_0$ and the Frobenius $\Fr$:
$$\psi(s_0)=
\begin{bmatrix}\label{Matrices}
	\tilde{\eta}(s_0) & 0                   & 0      & \dots  & 0 \\
	0                 & \tilde{\eta}(s_0)^q & 0      & \dots  & 0 \\
	\vdots            & \vdots              & \vdots & \ddots & \vdots \\
	0                 & 0                   & 0      & \dots   & \tilde{\eta}(s_0)^{q^{n-1}}
\end{bmatrix}$$
and 
$$\psi(\Fr)=
\begin{bmatrix}
	0                   & 1      & 0      & \dots  & 0 \\
	0                   & 0      & 1      & \dots  & 0 \\
	\vdots              & \vdots & \vdots & \ddots & \vdots \\
	0                   & 0      & 0      & \dots  & 1 \\
	\tilde{\eta}(\Fr^n) & 0      & 0      & \dots  & 0
\end{bmatrix}
.$$
Under this basis, $T=C_{\hat{G}}(\psi_{\ell})$ is the diagonal torus of $GL_n$, with $\Fr$ acting by conjugation via $\psi$, i.e., 
$$\Fr. \diag(t_1, t_2, ..., t_{n-1}, t_{n}) = \psi(\Fr)\diag(t_1, t_2, ..., t_{n-1}, t_{n})\psi(\Fr)^{-1} = \diag(t_{2}, t_{3}, ..., t_{n}, t_{1}).$$
So one can compute that 
$$T^{\Fr=(-)^q}\cong \mu_{q^n-1},$$
and that
$$(T^{\Fr=(-)^q})^0 \cong \mu_{\ell^k},$$
where $k \in \mathbb{Z}$ is maximal such that $\ell^k$ divides $q^n-1$.

To compute the quotient $[T/T]$, we note that $T$ acts on $T$ via twisted conjugation
$$(t, t') \mapsto (tnt^{-1}n^{-1})t',$$
where $n=\psi(Fr)$. So in our case, this action is 
$$(t_1, t_2, ..., t_n).(t'_1, t'_2, ..., t'_n)=(t_n^{-1}t_1t'_1, t_1^{-1}t_2t'_2, ..., t_{n-1}^{-1}t_nt'_n).$$ 
We see that the orbits of this action are determined by the determinants (hence are in bijection with $\mathbb{G}_m$), and the center $\mathbb{G}_m \cong Z(\hat{G}) \subseteq T$ acts trivially. Therefore,
$$[T/T] \cong [\mathbb{G}_m/\mathbb{G}_m],$$
where $\mathbb{G}_m$ acts trivially on $\mathbb{G}_m$.

In conclusion, we have that the connected component of $Z^1(W_F, \hat{G})_{\overline{\mathbb{Z}}_{\ell}}$ containing $\varphi$ is
$$C_{\varphi} \cong [\mathbb{G}_m/\mathbb{G}_m] \times \mu_{\ell^k},$$
where $\mathbb{G}_m$ acts trivially on $\mathbb{G}_m$, and $k \in \mathbb{Z}$ is maximal such that $\ell^k$ divides $q^n-1$.

\section{Representation side}

By modular Deligne-Lusztig theory, the block $\mathcal{A}_{x,1}$ of $GL_n(\mathbb{F}_q)$ containing an irreducible supercuspidal representation $\sigma$ is equivalent to the block of an elliptic torus. Such an elliptic torus is isomorphic to $\mathbb{F}_{q^n}^*$. So this block is equivalent to $\overline{\mathbb{Z}}_{\ell}[s]/(s^{\ell^k}-1)\Modl$, where $k \in \mathbb{Z}$ is maximal such that $\ell^k$ divides $q^n-1$.

$\mathcal{A}_{x,1}$ inflats to a block of $K:=GL_n(\mathcal{O}_F)$ containing the inflation $\tilde{\sigma}$\footnote{Since we started with an irreducible supercuspidal representation $\sigma$, its inflation $\tilde{\sigma}$ automatically has supercuspidal reduction.} of $\sigma$, and further corresponds to a block $\mathcal{B}_{x,1}$ of $KZ$ containing $\rho$, an extension of $\tilde{\sigma}$ to $KZ$, where $Z$ is the center of $GL_n(F)$. We have
$$\mathcal{B}_{x,1} \cong \mathcal{A}_{x,1} \otimes \Rep_{\overline{\mathbb{Z}}_{\ell}}(\mathbb{Z}) \cong \overline{\mathbb{Z}}_{\ell}[s]/(s^{\ell^k}-1) \otimes \overline{\mathbb{Z}}_{\ell}[t, t^{-1}]\Modl,$$
because
$$KZ \cong K \times \{\diag(\pi^m, ..., \pi^m) \;|\; m \in \mathbb{Z}\} \cong K \times \mathbb{Z}.$$
Argue as in the proof of Theorem \ref{Thm Main} we see that the compact induction $\cInd_{KZ}^G$ induces an equivalence of categories
$$\mathcal{B}_{x,1} \cong \mathcal{C}_{x,1},$$
where $\mathcal{C}_{x,1}$ is the block of $\Rep_{\overline{\mathbb{Z}}_{\ell}}(G(F))$ containing $\pi:=\cInd_{KZ}^G\rho$.

Since every irreducible depth-zero supercuspidal representation $\pi$ arises as above, we have that the block containing an irreducible depth-zero supercuspidal representation $\pi$ satisfies
$$\Rep_{\overline{\mathbb{Z}}_{\ell}}(G(F))_{[\pi]} \cong \mathcal{C}_{x,1} \cong \overline{\mathbb{Z}}_{\ell}[s]/(s^{\ell^k}-1) \otimes \overline{\mathbb{Z}}_{\ell}[t, t^{-1}]\Modl,$$
where $k \in \mathbb{Z}$ is maximal such that $\ell^k$ divides $q^n-1$.

\chapter{The categorical local Langlands conjecture} \label{Chapter CLLC}

In this chapter, we prove the categorical local Langlands conjecture for depth-zero supercuspidal part of $G=GL_n$ with coefficients $\Lambda=\overline{\mathbb{Z}}_{\ell}$ in Fargues-Scholze's form (see \cite[Conjecture X.3.5]{fargues2021geometrization}).

Let $\varphi \in Z^1(W_F, \hat{G}(\overline{\mathbb{F}}_{\ell}))$ be an irreducible tame $L$-parameter. Let $C_{\varphi}$ be the connected component of $[Z^1(W_F, \hat{G})_{\overline{\mathbb{Z}}_{\ell}}/\hat{G}]$ containing $\varphi$. 

The goal is to show that there is an equivalence
$$\mathcal{D}_{\lis}^{C_{\varphi}}(\Bun_G, \overline{\mathbb{Z}}_{\ell})^{\omega} \cong \mathcal{D}^{b, \qc}_{\Coh, \Nilp}(C_{\varphi})$$
of derived categories.

As a first step, let's unravel the definitions of both sides and describe them explicitly.

\section{Unraveling definitions}

\subsection{$L$-parameter side}

Let us first state a lemma that makes the decorations in $\mathcal{D}^{b, \qc}_{\Coh, \Nilp}(C_{\varphi})$ go away. We postpone its proof to Subsection \ref{Subsection Nilp}.

\begin{lemma} \label{Lemma 1}
	$\mathcal{D}^{b, \qc}_{\Coh, \Nilp}(C_{\varphi}) \cong \mathcal{D}^{b}_{\Coh, \Nilp}(C_{\varphi}) \cong \mathcal{D}^b_{\Coh, \{0\}}(C_{\varphi}) \cong \Perf(C_{\varphi}).$
\end{lemma} 
	
Let us assume the lemma for the moment. By our computation before,
$$C_{\varphi} \cong [\mathbb{G}_m/\mathbb{G}_m] \times \mu_{\ell^k} \cong \mathbb{G}_m \times [*/\mathbb{G}_m] \times \mu_{\ell^k},$$
where $k \in \mathbb{Z}_{\geq 0}$ is maximal such that $\ell^k$ divides $q^n-1$. So
$$\Perf(C_{\varphi}) \cong \Perf(\mathbb{G}_m \times [*/\mathbb{G}_m] \times \mu_{\ell^k}) \cong \Perf(\mathbb{G}_m) \otimes \Perf([*/\mathbb{G}_m]) \otimes \Perf(\mu_{\ell^k}).$$
Here, since the category of algebraic representations of the algebraic group $\mathbb{G}_m$ is semisimple (see for example, \cite[I.2.11]{jantzen2003representation}),
$$\Perf([*/\mathbb{G}_m]) \cong \bigoplus_{\chi}\Perf(\overline{\mathbb{Z}}_{\ell})\chi \cong \bigoplus_{\chi}\Perf(\overline{\mathbb{Z}}_{\ell}),$$
where $\chi$ runs over characters of $\mathbb{G}_m$ 
$$X^*(\mathbb{G}_m)=\{t \mapsto t^m \;|\; m \in \mathbb{Z}\} \cong \mathbb{Z}.$$

In conclusion, we have 
$$\Perf(C_{\varphi}) \cong \bigoplus_{\chi}\Perf(\mathbb{G}_m \times \mu_{\ell^k}),$$
where $\chi$ runs over characters of $\mathbb{G}_m$ 
$$X^*(\mathbb{G}_m)=\{t \mapsto t^m \;|\; m \in \mathbb{Z}\} \cong \mathbb{Z}.$$

\subsection{$\Bun_G$ side}

We refer the reader to \cite[Chapter 1]{fargues2021geometrization} for details of the notions used below. Let $\Bun_G$ be the stack of $G$-bundles on the Fargues-Fontaine curve.

Since $\varphi$ is irreducible, 
$$\mathcal{D}^{C_{\varphi}}_{\lis}(\Bun_G, \overline{\mathbb{Z}}_{\ell})^{\omega} \cong \mathcal{D}^{C_{\varphi}}_{\lis}(\Bun_G^{\sss}, \overline{\mathbb{Z}}_{\ell})^{\omega},\footnote{See \cite[Definition VII.6.1]{fargues2021geometrization} for the definition of $\mathcal{D}_{\lis}$.}$$
where $\omega$ means the subcategory of compact objects, and $\Bun_G^{\sss}$ is the semistable locus of $\Bun_G$.
See \cite[Section X.2]{fargues2021geometrization}.

Since
$$\Bun_G^{\sss} \cong \bigsqcup_{b \in B(G)_{\basic}}[*/G_b(F)],\footnote{See \cite[Theorem I.4.1]{fargues2021geometrization}.}$$
we have 
$$\mathcal{D}^{C_{\varphi}}_{\lis}(\Bun_G^{\sss}, \overline{\mathbb{Z}}_{\ell})^{\omega} \cong \bigoplus_{b \in B(G)_{\basic}}\mathcal{D}^{C_{\varphi}}(G_b(F), \overline{\mathbb{Z}}_{\ell})^{\omega},$$
where $B(G)_{\basic}$ is the subset of basic elements in the Kottwitz set $B(G)$ of $G$-isocrystals.

Let us look closer into each direct summand. In our case $G=GL_n$, a $G$-isocrystal is a rank $n$ isocrystal, and it is basic precisely when it has only one slope. So we have
$$B(G)_{\basic} \cong \pi_1(G)_{\Gamma} \cong \mathbb{Z}.$$ 

Let us first look at the summand for $b=1$ (corresponding to $0 \in \mathbb{Z} \cong B(G)_{\basic}$). For $b=1$, $G_b \cong GL_n$, and 
$$\mathcal{D}^{C_{\varphi}}(G_b(F), \overline{\mathbb{Z}}_{\ell})^{\omega} \cong \mathcal{D}^{C_{\varphi}}(GL_n(F), \overline{\mathbb{Z}}_{\ell})^{\omega} \cong \mathcal{D}(\Rep_{\overline{\mathbb{Z}}_{\ell}}(GL_n(F))_{[\pi]})^{\omega},$$
where $\pi \in \Rep_{\overline{\mathbb{F}}_{\ell}}(GL_n(F))$ is the representation with $L$-parameter $\varphi$, and $\Rep_{\overline{\mathbb{Z}}_{\ell}}(GL_n(F))_{[\pi]}$ is the block of $\Rep_{\overline{\mathbb{Z}}_{\ell}}(GL_n(F))$ containing $\pi$.
In addition, we've computed in Chapter \ref{Chapter GL_n} that
$$\Rep_{\overline{\mathbb{Z}}_{\ell}}(GL_n(F))_{[\pi]} \cong \overline{\mathbb{Z}}_{\ell}[t, t^{-1}] \otimes \overline{\mathbb{Z}}_{\ell}[s]/(s^{\ell^k}-1)\Modl \cong \QCoh(\mathbb{G}_m \times \mu_{\ell^k}),$$
where $k \in \mathbb{Z}_{\geq 0}$ is again maximal such that $\ell^k$ divides $p^n-1$. So we have
$$\mathcal{D}^{C_{\varphi}}(GL_n(F), \overline{\mathbb{Z}}_{\ell})^{\omega} \cong \mathcal{D}(\QCoh(\mathbb{G}_m \times \mu_{\ell^k}))^{\omega} \cong \Perf(\mathbb{G}_m \times \mu_{\ell^k}).$$

We can get a similar description of $\mathcal{D}^{C_{\varphi}}(G_b(F), \overline{\mathbb{Z}}_{\ell})$ (with arbitrary $b$) for free by the spectral action and its compatibility with $\pi_1(G)_{\Gamma}$-grading. For this, we consider the composition
$$q: C_{\varphi} \cong \mathbb{G}_m \times [*/\mathbb{G}_m] \times \mu_{\ell^k} \to [*/\mathbb{G}_m].$$
Recall that 
$$\Perf([*/\mathbb{G}_m]) \cong \bigoplus_{\chi}\Perf(\overline{\mathbb{Z}}_{\ell})\chi.$$
For any $\chi$, we denote by $\mathcal{M}_{\chi}$ the corresponding simple object in $\Perf([*/\mathbb{G}_m])$. Moreover, $\mathcal{M}_{\chi}$ pullbacks to a line bundle on $C_{\varphi}$
$$\mathcal{L}_{\chi}:=q^*\mathcal{M}_{\chi}.$$
We can now state the key proposition that allows us to get to arbitrary $b \in B(G)_{\basic}$ from the $b=1$ case, using the spectral action.
\begin{proposition}\label{Prop Spectral action}\
	\begin{enumerate}
		\item The restriction of the spectral action by $\mathcal{L}_{\chi}$ to $\mathcal{D}(G_b(F), \overline{\mathbb{Z}}_{\ell})$ factors through $\mathcal{D}(G_{b-\chi}(F), \overline{\mathbb{Z}}_{\ell})$.\footnote{Here we identify both $X^*(\mathbb{G}_m) \cong X^*(Z(\hat{G}))$ and $B(G)_{\basic} \cong \pi_1(G)_{\Gamma}$ with $\mathbb{Z}$. Hence $b-\chi$ makes sense.}
		\begin{tikzcd}
			{\mathcal{L}_{\chi}*-:} & {\mathcal{D}_{\lis}(\Bun_G, \overline{\mathbb{Z}}_{\ell})} && {\mathcal{D}_{\lis}(\Bun_G, \overline{\mathbb{Z}}_{\ell})} \\
			\\
			& {\mathcal{D}(G_b(F), \overline{\mathbb{Z}}_{\ell})} && {\mathcal{D}(G_{b-\chi}(F), \overline{\mathbb{Z}}_{\ell})}
			\arrow[from=1-2, to=1-4]
			\arrow[dashed, from=3-2, to=3-4]
			\arrow["\subseteq", sloped, from=3-2, to=1-2]
			\arrow["\subseteq", sloped, from=3-4, to=1-4]
		\end{tikzcd}
		\item $\mathcal{L}_{\chi}*-: \mathcal{D}(G_b(F), \overline{\mathbb{Z}}_{\ell}) \to \mathcal{D}(G_{b-\chi}(F), \overline{\mathbb{Z}}_{\ell})$ is an equivalence of categories, with inverse $\mathcal{L}_{\chi^{-1}}*-$.
	\end{enumerate}
\end{proposition}

\begin{proof}
	For the first assertion, see \cite[Lemma 5.3.2]{zou2022categorical}. For the second assertion, note that $\mathcal{L}_{\chi}$ and $\mathcal{L}_{\chi^{-1}}$ are inverse to each other once they are well-defined, since $q^*$ preserves tensor product.
\end{proof}
So we have 
$$\mathcal{D}^{C_{\varphi}}(\Bun_G, \overline{\mathbb{Z}}_{\ell})^{\omega} \cong \bigoplus_{b \in B(G)_{\basic}}\mathcal{D}^{C_{\varphi}}(G_b(F), \overline{\mathbb{Z}}_{\ell})^{\omega} \cong \bigoplus_{b \in B(G)_{\basic}}\Perf(\mathbb{G}_m \times \mu_{\ell^k}).$$

\subsection{The nilpotent singular support condition} \label{Subsection Nilp}
Now we prove Lemma \ref{Lemma 1}. 

The first isomorphism is because $C_{\varphi}$ is connected, hence the quasi-compact support condition $\qc$ is automatic. 

The second isomorphism needs some computation. For the definition and properties of the nilpotent singular support condition $\Nilp$, we refer to \cite[Section VIII.2]{fargues2021geometrization}. At the end of the day, it boils down to the fact that for any point $\varphi'$ in $C_{\varphi}$ valued in an algebraically closed $\Lambda$-field $k$,
$$\left(x_{\varphi'}^*\Sing_{[Z^1(W_F, \hat{G})/\hat{G}]/\Lambda}\right)\cap \left(\mathcal{N}_{\hat{G}}^*\otimes _{\mathbb{Z}_{\ell}}k\right) \cong H^0(W_F, \hat{\mathfrak{g}}^*\otimes_{\mathbb{Z}_{\ell}}k(1)) \cap \left(\mathcal{N}_{\hat{G}}^*\otimes _{\mathbb{Z}_{\ell}}k\right)=\{0\},$$
where $\hat{\mathfrak{g}}^*$ is the dual of the adjoint representation of $\hat{G}$, $W_F$ acts by conjugacy on $\hat{\mathfrak{g}}$ through $\varphi'$ (and then taking dual and Tate twist to get the action on $\hat{\mathfrak{g}}^*\otimes_{\mathbb{Z}_{\ell}}k(1)$), and $\mathcal{N}_{\hat{G}}^* \subseteq \hat{\mathfrak{g}}^*$ is the nilpotent cone.

In our case, $\hat{G}=GL_n$, $\hat{\mathfrak{g}}=M_{n\times n}$ is the set of $n \times n$ matrices. Take $\varphi'=\varphi$ for example (the similar argument works for any $\varphi'$ in $C_{\varphi}$). $W_F$ acts by conjugacy on $\hat{\mathfrak{g}}=M_{n\times n}$ through $\varphi$, hence induces an action of $W_F$ on the dual space with Tate twist $\hat{\mathfrak{g}}^*\otimes_{\mathbb{Z}_{\ell}}k(1)$. One can use the explicit matrices \ref{Matrices} of $s_0$ to compute that the fixed points $H^0(W_F, \hat{\mathfrak{g}}^*\otimes_{\mathbb{Z}_{\ell}}k(1))$ is contained in the (dual of) the diagonal torus of $M_{n\times n}^*$, the dual Lie algebra $\hat{\mathfrak{g}}^*$. On the other hand, the nilpotent cone $\mathcal{N}_{\hat{G}}^*$ is nothing else than the (dual of) nilpotent matrices in $M_{n\times n}^*$. So we conclude that 
$$H^0(W_F, \hat{\mathfrak{g}}^*\otimes_{\mathbb{Z}_{\ell}}k(1)) \cap \left(\mathcal{N}_{\hat{G}}^*\otimes _{\mathbb{Z}_{\ell}}k\right)=\{0\}.$$

The last isomorphism of Lemma \ref{Lemma 1} is \cite[Theorem VIII.2.9]{fargues2021geometrization}.

\section{The spectral action induces an equivalence of categories}
To summarize, we have (abstract) equivalences of categories
$$\mathcal{D}^{b, \qc}_{\Coh, \Nilp}(C_{\varphi}) \cong \bigoplus_{\chi \in \mathbb{Z}}\Perf(\mathbb{G}_m \times \mu_{\ell^k}) \cong \bigoplus_{b \in \mathbb{Z}}\Perf(\mathbb{G}_m \times \mu_{\ell^k}) \cong \mathcal{D}^{C_{\varphi}}_{\lis}(\Bun_G, \overline{\mathbb{Z}}_{\ell})^{\omega},$$
where we identified both $X^*(\mathbb{G}_m) \cong X^*(Z(\hat{G}))$ and $B(G)_{\basic} \cong \pi_1(G)_{\Gamma}$ with $\mathbb{Z}$. The next goal is to show that the spectral action induces an equivalence of categories
\begin{equation}\label{Equiv}
	\mathcal{D}_{\lis}^{C_{\varphi}}(\Bun_G, \overline{\mathbb{Z}}_{\ell})^{\omega} \cong \mathcal{D}^{b, \qc}_{\Coh, \Nilp}(C_{\varphi}).
\end{equation}

\subsection{Definition of the functor}

Let's first define the functor. For this, let's choose a Whittaker datum consisting of a Borel $B \subseteq G$ and a generic character $\vartheta: U(F) \to \overline{\mathbb{Z}}_{\ell}^*$, where $U$ is the unipotent radical of $B$. Let $\mathcal{W}_{\vartheta}$ be the sheaf concentrated on $\Bun_G^1$ corresponding to the representation $W_{\vartheta}:=\cInd_{U(F)}^{G(F)}\vartheta$. Let $W_{\vartheta, [\pi]}$ be the restriction of $W_{\vartheta}$ to the block $\Rep_{\overline{\mathbb{Z}}_{\ell}}(G(F))_{[\pi]}$, and $\mathcal{W}_{\vartheta, [\pi]}$ the corresponding sheaf.

We define our desired functor by spectral acting on $\mathcal{W}_{\vartheta, [\pi]}$:
$$\Theta: \mathcal{D}^{b, \qc}_{\Coh, \Nilp}(C_{\varphi}) \cong \Perf(C_{\varphi}) \longrightarrow \mathcal{D}_{\lis}^{C_{\varphi}}(\Bun_G, \overline{\mathbb{Z}}_{\ell})^{\omega}, \qquad A \mapsto A*\mathcal{W}_{\vartheta, [\pi]}.$$

\subsection{Equivalence on the degree zero part}

We now show that $\Theta$ induces a derived equivalence on the degree zero part. Before that, we do some preparations.

The main input is local Langlands in families (see \cite{helm2018converse}): For $G=GL_n$, there are natural isomorphisms
$$\mathcal{O}(Z^1(W_F, \hat{G})_{\Lambda}/\hat{G}) \cong \mathcal{Z}_{\Lambda}(G(F)) \cong \End_{G}(W_{\vartheta}),$$
where $\mathcal{Z}_{\Lambda}(G(F))$ is the Bernstein center of $\Rep_{\Lambda}(G(F))$; the first map is the unique map between $\mathcal{O}(Z^1(W_F, \hat{G})_{\Lambda}/\hat{G})$ and $\mathcal{Z}_{\Lambda}(G(F))$ that is compatible with the classical local Langlands correspondence for $GL_n$, hence also same as the map defined in \cite[Section VIII.4]{fargues2021geometrization}; the second map is given by the action of the Bernstein center on the representation $W_{\vartheta}$.

We shall also use the following two Lemmas: 
%(\textcolor{red}{need explain})
%\begin{enumerate}\label{Fact two facts}
%	\item The restriction of the Whittaker representation $W_{\vartheta, [\pi]}$ is a (finitely generated) projective generator of $\Rep_{\Lambda}(G(F))_{[\pi]}$. For projectivity, see \cite[Section 3]{helm2016whittaker}. For being a generator, I cann't find a reference for $\overline{\mathbb{Z}}_{\ell}$-coefficients. But I believe the same argument as the $\overline{\mathbb{Q}}_{\ell}$-coefficients (see \cite[Section 39]{bushnell2006local}; note their definition of Whittaker representation is dual to our definition, as an induction instead of compact induction. See also, \cite[Section 2.1 and others]{bushnell2003generalized}) works. (\textcolor{red}{prove if have time})
%	\item The spectral action is compatible with the map 
%	$$\mathcal{O}(Z^1(W_F, \hat{G})_{\Lambda}/\hat{G}) \cong \mathcal{Z}_{\Lambda}(G(F)).$$ See \cite[Section 5]{zou2022categorical}.
%\end{enumerate}

\begin{lemma}\label{Lemma Whittaker is proj gen}
	The restriction of the Whittaker representation $W_{\vartheta, [\pi]}$ is a finitely generated projective generator of $\Rep_{\Lambda}(G(F))_{[\pi]}$.
\end{lemma}

This lemma was proven in \cite{chan2019bernstein}. Alternatively, we sketch a proof as follows.

\begin{proof}
	For projectivity, see \cite[Section 4]{aizenbud2022strong}. Note their argument is with complex coefficients, but still goes through for $\overline{\mathbb{Z}}_{\ell}$-coefficients, because the Jacquet functor 
	$$r_{M, G}: \pi \mapsto \pi_U$$
	is still exact under the assumption that $p$ is invertible in $\overline{\mathbb{Z}}_{\ell}$ (see \cite[Section II.2.1]{vigneras1996representations}).
	
	For being a generator, in the $GL_2$ case, one can argue similarly as the $\overline{\mathbb{Q}}_{\ell}$-case in \cite[Section 39]{bushnell2006local}. (Note their definition of Whittaker representation is dual to our definition, as an induction instead of compact induction. But it still goes through by taking dual everywhere. See also, \cite[Section 2.1 and others]{bushnell2003generalized}.) See \cite{bushnell2003generalized} for the $GL_n$ case.
	
	For finite generation, it's enough to observe that $W_{\vartheta, [\pi]}$ has finitely many irreducible subquotients (by our explicit description of the block $\Rep_{\Lambda}(G(F))_{[\pi]})$ with multiplicity one (again, argue similarly as in \cite[Section 39]{bushnell2006local} for the multiplicity one property).
\end{proof}

\begin{lemma}\label{Lemma Spectral action Bern center}
	The spectral action is compatible with the map 
	$$\mathcal{O}(Z^1(W_F, \hat{G})_{\Lambda}/\hat{G}) \cong \mathcal{Z}_{\Lambda}(G(F)).$$
\end{lemma}

\begin{proof}
	See \cite[Section 5]{zou2022categorical}.
\end{proof}

Now we state the main result of this subsection.

By compatibility with $\pi_1(G)_{\Gamma}$-grading (see Proposition \ref{Prop Spectral action}), $\Theta$ restricts to a map between degree-$0$ parts of both sides
$$\Theta_0:=\Theta|_{\Perf(C_{\varphi})_{\chi=0}}: \Perf(C_{\varphi})_{\chi=0} \longrightarrow \mathcal{D}_{\lis}^{C_{\varphi}}(\Bun_G, \overline{\mathbb{Z}}_{\ell})^{\omega}_{b=0},$$
where $\Perf(C_{\varphi})_{\chi=0} \cong \Perf(\mathbb{G}_m \times \mu_{\ell^k})$ and 
$$\mathcal{D}_{\lis}^{C_{\varphi}}(\Bun_G, \overline{\mathbb{Z}}_{\ell})^{\omega}_{b=0} \cong \mathcal{D}(\Rep_{\overline{\mathbb{Z}}_{\ell}}(G(F))_{[\pi]})^{\omega}.$$

\begin{proposition}
	Under the above identifications, the functor
	$$\Theta_0: \Perf(\mathbb{G}_m \times \mu_{\ell^k}) \longrightarrow \mathcal{D}(\Rep_{\overline{\mathbb{Z}}_{\ell}}(G(F))_{[\pi]})^{\omega} \qquad A \mapsto A*W_{\vartheta, [\pi]}$$
	is an equivalence of derived categories.
\end{proposition}

\begin{proof}
	Let's first prove that $\Theta_0$ is fully faithful. The key observation is that fully faithfulness can be checked on generators of the triangulated category $\Perf(C_{\varphi})_{\chi=0} \cong \Perf(\mathbb{G}_m \times \mu_{\ell^k})$ (see Lemma \ref{Lemma Generator Triangulated Category}). In our case, the structure sheaf $\mathcal{O}$ is a generator of $\Perf(\mathbb{G}_m \times \mu_{\ell^k})$, hence it suffices to check fully faithfulness on the structure sheaf. Recall this map sends the structure sheaf $\mathcal{O} \in \Perf(\mathbb{G}_m \times \mu_{\ell^k})$ to the restriction of the Whittaker representation $W_{\vartheta, [\pi]}$. So it suffices to show that the map between $\Hom$-sets in the derived category
	$$\Hom(\mathcal{O}, \mathcal{O}[n]) \to \Hom(W_{\vartheta, [\pi]}, W_{\vartheta, [\pi]}[n])$$
	is a bijection for all $n \in \mathbb{Z}$. The case $n \neq 0$ follows from the vanishing of higher $\Ext$ for projective objects ($\mathcal{O}$ and $W_{\vartheta, [\pi]}$). For $n=0$, $\Hom(\mathcal{O}, \mathcal{O}) \cong \mathcal{O}(C_{\varphi})$, and the above map fits into the following commutative diagram by Lemma \ref{Lemma Spectral action Bern center}, hence a bijection.
		% https://q.uiver.app/#q=WzAsNCxbMCwwLCJcXG1hdGhjYWx7T30oWl4xKFdfRiwgXFxoYXR7R30pX3tcXExhbWJkYX0vXFxoYXR7R30pIl0sWzEsMCwiXFxFbmRfe0d9KFdfe1xcdmFydGhldGF9KSJdLFswLDEsIlxcbWF0aGNhbHtPfShDX3tcXHZhcnBoaX0pIl0sWzEsMSwiXFxFbmRfe0d9KFdfe1xcdmFydGhldGEsIFtcXHBpXX0pIl0sWzIsM10sWzIsMCwiXFxzdWJzZXRlcSIsMSx7InN0eWxlIjp7ImJvZHkiOnsibmFtZSI6Im5vbmUifSwiaGVhZCI6eyJuYW1lIjoibm9uZSJ9fX1dLFszLDEsIlxcc3Vic2V0ZXEiLDEseyJzdHlsZSI6eyJib2R5Ijp7Im5hbWUiOiJub25lIn0sImhlYWQiOnsibmFtZSI6Im5vbmUifX19XSxbMCwxLCJcXGNvbmciXV0=
		\[\begin{tikzcd}
			{\mathcal{O}(Z^1(W_F, \hat{G})_{\Lambda}/\hat{G})} & {\End_{G}(W_{\vartheta})} \\
			{\mathcal{O}(C_{\varphi})} & {\End_{G}(W_{\vartheta, [\pi]})}
			\arrow[from=2-1, to=2-2]
			\arrow["\subseteq"{description}, sloped, draw=none, from=2-1, to=1-1]
			\arrow["\subseteq"{description}, sloped, draw=none, from=2-2, to=1-2]
			\arrow["\cong", from=1-1, to=1-2]
		\end{tikzcd}\]

	The essentially surjectivity follows from Lemma \ref{Lemma Whittaker is proj gen} that $W_{\vartheta, [\pi]}$ is a finitely generated projective generator of $\Rep_{\Lambda}(G(F))_{[\pi]}$.
	
\end{proof}

\begin{remark}
	We remark that to use Lemma \ref{Lemma Generator Triangulated Category} in the above proof, we need the fact that the spectral action commutes with direct sums. Indeed, it commutes with colimits. This boils down to the fact that the Hecke operators commute with colimits, as they are defined using pullback, tensor product, and shriek pushforward, all of which are left adjoints, hence commutes with colimits.
\end{remark}

%\textcolor{red}{How much do the previous chapters help in this argument?} I guess if you accept $W_{\vartheta, [\pi]}$ is a projective generator (so we have control over the rep side), and if you accept LLIF (so in particular we have control over the $L$-parameter side), there is basically nothing to prove. (Maybe except the form $\mathbb{G}_m/\mathbb{G}_m \times \mu$ that helps to reduce to degree-zero part, which is a category of $\Perf$ over some ring.)

\begin{lemma}\label{Lemma Generator Triangulated Category}
	Let $F: \mathcal{D}_1 \to \mathcal{D}_2$ be a triangulated functor between triangulated categories that commutes with direct sums, and let $E$ be a generator of $\mathcal{D}_1$. Assume that $F$ induces isomorphisms
	$$\Hom(E, E[n]) \cong \Hom(F(E), F(E[n]))$$
	for all $n \in \mathbb{Z}$, then $F$ is fully faithful.
\end{lemma}

\begin{proof}
	We use the general lemma \cite[Stack, Tag 0ATH]{stacks-project} twice. 
	
	To check that condition (1) and (3) in the general lemma holds, we use $F$ commutes with direct sums.
	
	To check that condition (2) in the general lemma holds, we use the five lemma.
	
	We first apply it with the property $T=T_1$: an object $M \in \mathcal{D}_1$ has the property $T_1$ (written $T_1(M)$) if $F$ induces isomorphisms
	$$\Hom(M, E[n])=\Hom(F(M), F(E[n]))$$
	for all $n \in \mathbb{Z}$. The assumption implies that condition (4) in the general lemma holds: $T_1(E[n])$ for all $n \in \mathbb{Z}$. Therefore, $T_1(M)$ for all $M \in \mathcal{D}_1$.
	
	We then apply it with the property $T=T_2$: an object $N \in \mathcal{D}_1$ has the property $T_2$ (written $T_2(M)$) if $F$ induces isomorphisms
	$$\Hom(M, N)=\Hom(F(M), F(N))$$
	for all $M \in \mathcal{D}_1$. By the last paragraph, $T_1(M)$ for all $M \in \mathcal{D}_1$, i.e., $T_2(E[n])$ for all $n \in \mathbb{Z}$. Therefore, $T_2(N)$ for all $N \in \mathcal{D}_1$. In other words, $F$ is fully faithful.
\end{proof}

\subsection{The full equivalence}	

Finally, we use the spectral action to get the full equivalence. Indeed, on the $L$-parameter side, for any character $\chi' \in X^*(\mathbb{G}_m)$, tensoring with $\mathcal{\mathcal{L}_{\chi'}}$ induces an equivalence
$$\mathcal{\mathcal{L}_{\chi'}} \otimes -: \Perf(C_{\varphi})_{\chi=0} \cong \Perf(C_{\varphi})_{\chi=\chi'}.$$
Similarly, on the $\Bun_G$ side, by Proposition \ref{Prop Spectral action}, spectral acting by $\mathcal{\mathcal{L}_{\chi'}}$ induces an equivalence
$$\mathcal{\mathcal{L}_{\chi'}}*-: \mathcal{D}_{\lis}^{C_{\varphi}}(\Bun_G, \overline{\mathbb{Z}}_{\ell})^{\omega}_{b=0} \cong \mathcal{D}_{\lis}^{C_{\varphi}}(\Bun_G, \overline{\mathbb{Z}}_{\ell})^{\omega}_{b=-\chi'}.$$ Therefore, we get the full equivalence via the spectral action.

\chapter{Conclusions and questions}\label{Chapter conclusion}

%1 How Chapters \ref{Chapter MoLP}, \ref{Chapter Rep} help in Chapter \ref{Chapter CLLC}?
%2 It seems can use Chapters \ref{Chapter MoLP}, \ref{Chapter Rep} to reprove LLIF?
%3 The relation between our projective generator $\Pi_{x,1}$ and Bernstein's $\cInd_{G^0}^G\pi_0$?
%4 CLLC for depth-zero part of general groups? Identification of the torus?

In this last chapter, let us make some concluding remarks and raise some further questions.

\section{How do Chapters \ref{Chapter MoLP}, \ref{Chapter Rep} help?}

First, let us reflect on how Chapters \ref{Chapter MoLP}, \ref{Chapter Rep} help to prove the categorical conjecture in Chapter \ref{Chapter CLLC}. It helps to write the $L$-parameter side as a $\mathbb{Z}$-grading of derived categories of modules over some ring (so that we can reduce to the degree-zero case and that we can check fully faithfulness on the generator). It does not help much thereafter if one accepts the local Langlands in families (LLIF)
$$\mathcal{O}(Z^1(W_F, \hat{G})_{\Lambda}/\hat{G}) \cong \mathcal{Z}_{\Lambda}(G(F)) \cong \End_{G}(W_{\vartheta}).$$
Indeed, the description of the representation side can be reproved using LLIF and that $W_{\vartheta, [\pi]}$ is a projective generator of $\Rep_{\Lambda}(G(F))_{[\pi]}$.

However, we note that (assuming the compatibility of Fargues-Scholze with the usual local Langlands correspondence for $GL_n$) it's possible to use Chapters \ref{Chapter MoLP}, \ref{Chapter Rep} to reprove the first isomorphism in LLIF. It boils down to the fact that if we have a morphism 
$$f: \mathbb{G}_m \times \mu \longrightarrow \mathbb{G}_m \times \mu$$
over $\overline{\mathbb{Z}}_{\ell}$, which becomes an isomorphism after base change to $\overline{\mathbb{Q}}_{\ell}$, then $f$ is an isomorphism over $\overline{\mathbb{Z}}_{\ell}$.

Moreover, assuming the result in Chapter $\ref{Chapter Rep}$, the second isomorphism in LLIF (when restricted to the block $\Rep_{\Lambda}(G(F))_{[\pi]}$) is almost equivalent to the statement that $W_{\vartheta, [\pi]}$ is a projective generator of $\Rep_{\Lambda}(G(F))_{[\pi]}$ (together with the fact that the irreducible representation of the block $\Rep_{\Lambda}(G(F))_{[\pi]}$ occur with multiplicity one in $W_{\vartheta, [\pi]}$). The latter can be proven almost by hand as in Lemma \ref{Lemma Whittaker is proj gen}.

\section{Relation to Bernstein's projective generator}

In \cite[p46, Section 3.3]{bernsteindraft}, Bernstein constructed a certain projective generator $$\cInd_{G^0}^{G(F)}(\rho|_{G^0})$$
of a supercuspidal block of $G(F)$ by inducing from $G^0$, the subgroup generated by compact subgroups (for representations with $\mathbb{C} \cong \overline{\mathbb{Q}}_{\ell}$ coefficients). It is interesting to understand the relation between the projective generators constructed in Chapter \ref{Chapter Rep} and Bernstein's projective generators.

\section{The categorical conjecture for general groups}

Since our results in Chapter \ref{Chapter MoLP}, \ref{Chapter Rep} also work for general reductive groups (other than $GL_n$), it is expected that they can be used to prove the categorical local Langlands conjectures for the depth-zero regular supercuspidal blocks of general reductive groups. In particular, the $\mu$ occurring in the result
of the $L$-parameter side (see Theorem \ref{Thm X/G}) should match with the block $\mathcal{A}_{x,1}$ (see Section \ref{Section rep application}) occurring on the representation side: we should have
$$\QCoh(\mu) \cong \mathcal{A}_{x,1}.$$
Indeed, $\mu=(T^{\Fr=(-)^q})^0$ is certain fixed points of a torus (see Theorem \ref{Thm X}), and $\mathcal{A}_{x,1}$ is also a block of some finite torus via Broué's equivalence \ref{Thm Broué}. These two finite tori should match (using the identification that $\QCoh(\mu_{n, \Lambda}) \cong \Rep_{\Lambda}(\mathbb{Z}/n\mathbb{Z})$).

One possible way to do this is via the (so far unknown in general) compatibility of Fargues-Scholze with classical local Langlands correspondences for depth-zero regular supercuspidal representations, say the work of DeBacker-Reeder \cite{debacker2009depth}. Then these two finite torus should be related by local Langlands for tori (see \cite[Section 4.3]{debacker2009depth}).